\documentclass{article}

\usepackage{arxiv}

\usepackage[utf8]{inputenc} 
\usepackage[T1]{fontenc}    
\usepackage{hyperref}       
\usepackage{url}            
\usepackage{booktabs}       
\usepackage{amsfonts}       
\usepackage{nicefrac}       
\usepackage{microtype}      
\usepackage{lipsum}		
\usepackage{graphicx}
\usepackage{color}
\usepackage{doi}
 \usepackage{amsmath,amsthm,amsmath,amssymb}
 \usepackage{enumerate}
 \usepackage{hyperref}
\usepackage{caption}
\usepackage{amsthm}
\usepackage{enumitem}
\theoremstyle{plain}
\newtheorem{theorem}{Theorem} [section]
\newtheorem{lemma}{Lemma}[section]
\newtheorem{definition}{Definition}[section]
\allowdisplaybreaks
\newcommand{\myitem}{\noindent - }

\title{An Inverse Source Problem For a Time-Fractional
Mixed Wave-Diffusion-Wave Equation in a
Cylindrical Domain}

\author{ \href{https://orcid.org/0000-0003-4443-6300}{\includegraphics[scale=0.06]{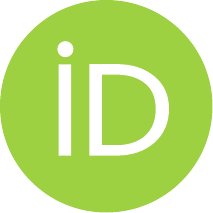}\hspace{1mm}Erkinjon Karimov}\\
	Department of Mathematics: Analysis,
Logic and Discrete Mathematics\\
	Ghent University\\
	Ghent, Belgium\\
	\texttt{erkinjon.karimov@ugent.be} \\
	\and
    \href{https://orcid.org/0009-0000-1887-795X}{\includegraphics[scale=0.06]{orcid.pdf}\hspace{1mm}Muzaffar Toshpulatov}\\
	Department of Algebra and Analysis\\
	Andijan State University\\
	Andijan, Uzbekistan\\
	\texttt{muzaffar.toshpulatov89@gmail.com} \\
}	 
\renewcommand{\shorttitle}{An inverse source problem for a time-fractional wave-diffusion-wave eqiation in ...}

\hypersetup{
pdftitle={An inverse source problem for ...},
pdfsubject={q-bio.NC, q-bio.QM},
pdfauthor={Erkinjon Karimov, Muzaffar Toshpulatov},
pdfkeywords={Inverse source problem, Wave-diffusion-wave equation, Bivariate Mittag-Leffler-type function}
}
\begin{document}
\maketitle

\begin{abstract}
This paper addresses the inverse source problem for a mixed-type fractional wave-diffusion-wave equation posed in a cylindrical domain. The governing equation involves a time-dependent variable-order fractional derivative, which enables the model to effectively capture temporal transitions between wave-like and diffusive behaviors. The solution is constructed in the form of a Fourier-Bessel series. By employing the method of separation of variables together with fundamental properties of Bessel functions, we analyze the uniform convergence of the resulting infinite series. This analysis ultimately leads to a rigorous proof of the existence of a solution.	
\end{abstract}


\section{Introduction}
In this section, we briefly outline the motivation behind our study of the wave-diffusion-wave process. We first explain how gas flow phenomena can be modeled using wave and diffusion equations. Next, we provide a concise justification for adopting a model that combines wave-diffusion-wave equations, arguing that this framework offers a more accurate representation of the underlying physical processes. Finally, we describe our approach to investigating the inverse source problem associated with the resulting mixed-type equation.
\subsection{Gas flow in porous media} 
Classical diffusion models based on Fick’s law assume that pressure disturbances propagate at infinite speed, an assumption that is physically unrealistic for gas flow in tight geological formations such as shale reservoirs. Experimental observations instead reveal anomalous transport behavior, commonly referred to as anomalous diffusion. In gas flow through porous media, anomalous diffusion arises from the complex and heterogeneous structure of the medium, causing particle transport to deviate from classical Fickian dynamics and exhibit a nonlinear dependence on time. The occurrence and specific type of anomalous diffusion depend on various properties of the porous medium \cite{1}, including porosity, specific surface area, and fractal dimension parameters. Consequently, the mathematical description of such processes extends beyond classical partial differential equations. Fractional-order diffusion equations, which replace traditional integer-order time and space derivatives with fractional derivatives, provide an effective modeling framework. This approach accurately captures memory effects and complex interactions within the system \cite{2}.
\subsection{Why wave-diffusion-wave equations?} 
When gas is injected into a porous medium or a pressure pulse is generated, the resulting disturbance propagates as a wave at very short time scales, characterized by finite propagation speed. As time progresses, the disturbance gradually spreads and exhibits diffusive behavior. This transition from wave-like to diffusive dynamics is precisely captured by wave-diffusion-wave equations. 

In particular, for a fractional time order $\alpha \in (1,2)$, the governing equation continuously interpolates between the classical wave equation as $\alpha \to 2^{-}$ and the diffusion equation as $\alpha \to 1^{+}$ \cite{3}. Such transitional behavior has been experimentally observed in gas dynamics, for instance, in isothermal gas dissolution processes \cite{4}. Consequently, employing a single fractional-order equation that encompasses both regimes provides a more realistic and unified modeling framework than switching between separate wave and diffusion models.

However, in some physical scenarios, the governing dynamics may temporarily reduce to a purely diffusive process before reverting to wave-like behavior. To account for such situations, we consider a time-switched system of wave-diffusion-wave equations, which allows for successive transitions between wave and diffusion regimes over different time intervals.
\subsection{Inverse source problem and our approach} 
We investigate the inverse source problem for a fractional-order wave-diffusion-wave equation, which generalizes classical second-order hyperbolic and parabolic models. Related inverse problems for fractional evolution equations have been studied in, for example, \cite{5,6}. In the present work, the time-fractional derivative is defined in the sense of the Prabhakar-Caputo operator, a nonlocal fractional derivative characterized by a three-parameter Mittag-Leffler kernel \cite{7,8}. This operator extends the classical Caputo derivative by incorporating additional degrees of freedom, allowing for a more accurate representation of complex memory effects, multi-scale relaxation phenomena, and temporal heterogeneity inherent in anomalous transport processes.

From a physical perspective, inverse source problems of this type arise in applications such as subsurface gas flow, where it is necessary to identify the location, intensity, or temporal profile of a gas leakage or injection source based on indirect boundary or interior measurements. 

To solve the problem, we employ the method of separation of variables, which reduces the original inverse problem to a sequence of fractional differential equations with respect to time and corresponding spatial eigenvalue problems. The spatial components are associated with a Sturm-Liouville system, while the temporal components involve fractional operators with Prabhakar-type memory kernels. The resulting solution is expressed as an infinite series expansion in terms of the eigenfunctions of the spatial operator. We derive explicit analytical representations for the unknown source function and establish sufficient conditions for the existence of a solution. Furthermore, by utilizing the asymptotic properties of special functions and appropriate estimates for the Mittag-Leffler kernel, we prove the uniform convergence of the obtained series solution.

\section{Formulation of a problem}

We consider the following mixed equation

\begin{equation} \label{1}
f(r)=
\begin{cases}
{}^{PC} D_{0t}^{\alpha_{1} ,\beta _{1} ,\gamma_{1} ,\delta } u(t,r)-\Delta u(t,r), & (t,r) \in \Omega_{1}, \\
{}^{PC} D_{T_{1}t}^{\alpha_{2} ,\beta _{2} ,\gamma_{2} ,\delta } u(t,r)-\Delta u(t,r), & (t,r) \in \Omega_{2}, \\
{}^{PC} D_{T_{2}t}^{\alpha_{3} ,\beta _{3} ,\gamma_{3} ,\delta } u(t,r)-\Delta u(t,r), & (t,r) \in \Omega_{3}
\end{cases}
\end{equation}
in a cylindrical domain $\Omega=\Omega_{1}\cup \Omega_{2} \cup \Omega_{3} \cup N_{1} \cup N_{2}$. Here $N_{1}=\left\{(t,r), t=T_{1}, 0 < r < 1  \right\}, N_{2}=\left\{(t,r), t=T_{2}, 0 < r < 1  \right\},\Omega_{1}=\left\{(t,r): 0 < r < 1 , 0 < t < T_{1} \right\}, \Omega_{2}=\left\{(t,r): 0 < r < 1 , T_{1} < t < T_{2} \right\},\Omega_{3}=\left\{(t,r): 0 < r < 1 ,  T_{2}< t < T \right\},\: T,T_{1},T_{2},r,\delta,\alpha_{i}, \gamma_{i},\beta_{i}(i=1,2,3) \in \mathbb{R}$ such that $0< T_{1} < T_{2} < T, 1<\beta_{1}\le2,0<\beta_{2}\le1, 1<\beta_{3}\le2$, 
$\Delta u(t,r)$ is the Laplacian which can be written in polar coordinates as
$$ \Delta u(t,r)=u_{rr}+\frac{1}r{u_{r}}, $$
\begin{equation*}  
	{}^{PC} D_{at}^{\alpha ,\beta ,\gamma ,\delta } y(t)={}^{P} I_{at}^{\alpha ,m-\beta ,-\gamma ,\delta } \frac{d^{m} }{dt^{m} } y(t)                      
\end{equation*} 
is the Prabhakar-Caputo fractional derivative \cite{DP18}, $m-1<\beta\le m$\, $(m\in \mathbb{N})$,
\begin{equation} \label{e2.4} 
	{}^{P} I_{at}^{\alpha ,\beta ,\gamma ,\delta } y(t)=\int _{a}^{t}(t-\xi )^{\beta -1} E_{\alpha ,\beta }^{\gamma } \left[\delta (t-\xi )^{\alpha } \right]y(\xi )d\xi ,t>a  
\end{equation} 
is the Prabhakar fractional integral  and $E_{\alpha ,\beta }^{\gamma } (z)=\sum\limits_{k=0}^\infty \dfrac{(\gamma)_k z^k}{\Gamma (\alpha k+\beta)k!}$ is the Prabhakar function \cite{7}.
\begin{figure}[h]
    \centering
    \includegraphics[width=0.6\textwidth]{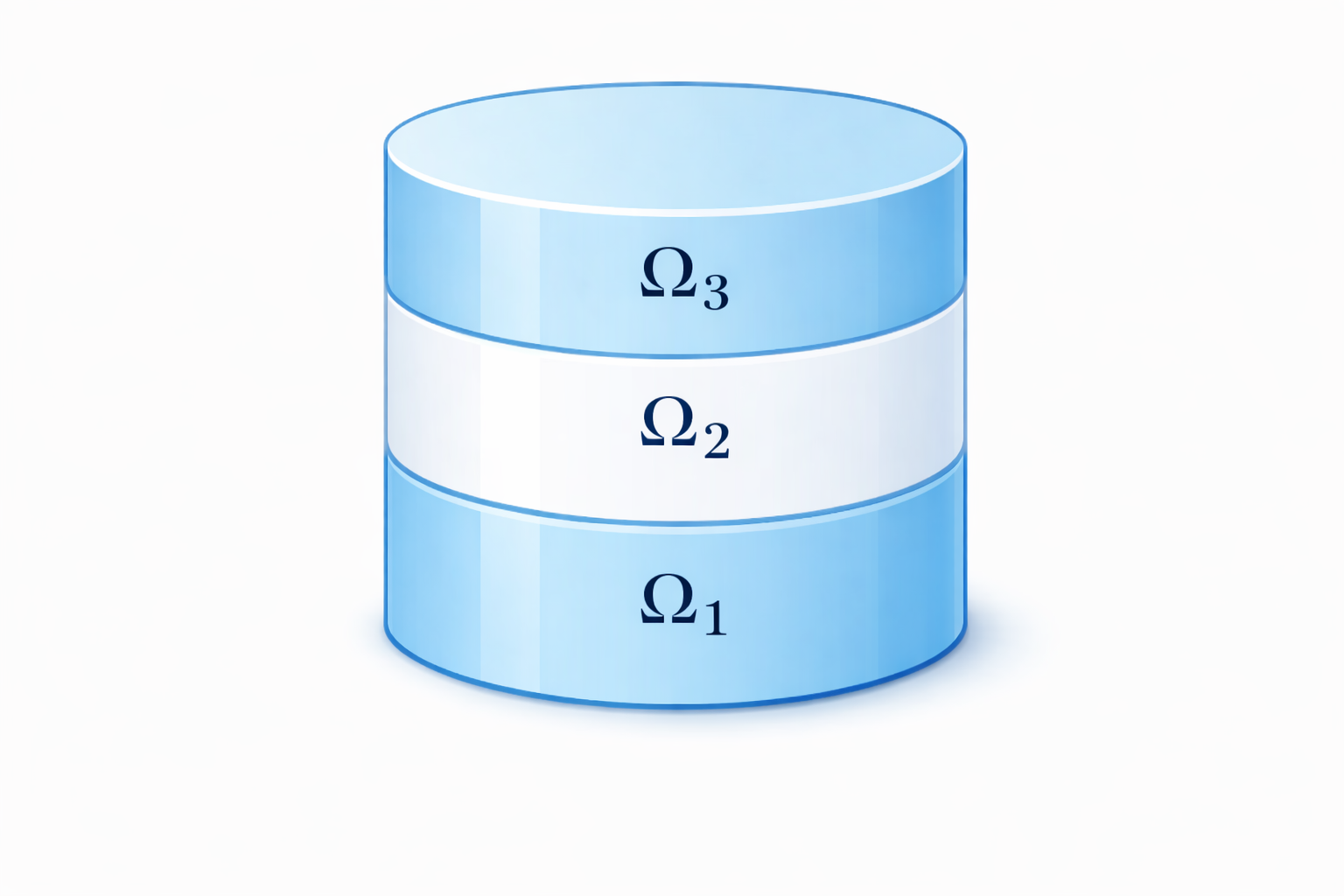}
    \captionsetup{labelformat=empty}
    \caption{Figure 1: Mixed domain $\Omega$.}
    \label{fig:1-rasm}
\end{figure}

\textbf{Direct problem.} Find a {\it regular solution} of \eqref{1} in $\Omega$ that satisfies the following

\myitem initial condition 
\begin{equation} \label{2}
    u(0,r)=\varphi (r), 0 \le r \le 1,
\end{equation}

\myitem boundary conditions
\begin{equation} \label{3}
    \Big[r\frac{\partial u(t,r)}{\partial r}\Big]_{r=0}=0 , u(t,r)|_{r=1}=0, 0 \le t \le T,
\end{equation}
 
\myitem transmitting conditions 

\begin{equation} \label{4}
    \lim_{t \to T_{1}+0} {}^{PC} D_{T_{1}t}^{\alpha_{2} ,\beta _{2} ,\gamma_{2} ,\delta } u(t,r)=\lim_{t \to T_{1}-0}u_{t}(t,r), \quad 0 \le r \le 1,
\end{equation}

\begin{equation} \label{5}
    \lim_{t \to T_{2}+0}u_{t}(t,r)=\lim_{t \to T_{2}-0} {}^{PC} D_{T_{1}t}^{\alpha_{2} ,\beta _{2} ,\gamma_{2} ,\delta } u(t,r), \quad 0 \le r \le 1.
\end{equation}
Here $f(r),\varphi(r)$ are given functions.

The similar direct problem for almost similar equation in a rectangular domain was investigated by us in \cite{KT25}. 

\begin{definition}
  Function $u(t,r)$ is called a {\bf \it regular solution} of \eqref{1} in $\Omega$, if $ u(t,r) \in C (\overline \Omega), u_{rr}(t,r) \in C ( \Omega)$ , ${}^{PC} D_{0t}^{\alpha_{1} ,\beta _{1} ,\gamma_{1} ,\delta } u(t,r)\in C(\Omega_{1}),{}^{PC} D_{T_{1}t}^{\alpha_{2} ,\beta _{2} ,\gamma_{2} ,\delta } u(t,r)\in C(\Omega_{2} \cup J_{1} \cup J_{2}), {}^{PC} D_{T_{2}t}^{\alpha_{3} ,\beta _{3} ,\gamma_{3} ,\delta } u(t,r)\in C(\Omega_{3})$ and satisfies \eqref{1} in $\Omega$.  
\end{definition}
 
If the function $f (r) $ is unknown, we can formulate the following inverse problem by imposing an additional condition on the time variable.

\textbf{Inverse source problem.} To find a pair of functions $\left\{ u(t,r),f(r) \right\}$ satisfying \eqref{1}-\eqref{5} together with the final-time measurement 
\begin{equation} \label{6}
    u(\xi, r)=\psi(r), T_{1} \le \xi \le T_{2},
\end{equation}
where $\psi(r)$ is a given function.

Using the Fourier method, we seek a solution for the homogeneous case of equation \eqref{1} in the following form:
\begin{equation} \label{7}
    u(t,r)=R(r)T(t).
\end{equation}

Substituting \eqref{7} into \eqref{1}, we get the following equalities:

$$ \frac{{}^{PC} D_{0t}^{\alpha_{1} ,\beta _{1} ,\gamma_{1} ,\delta } T(t)}{T(t)}=\frac{R'(r)}{rR(r)}+\frac{R''(r)}{R(r)}=-\mu^{2}, \qquad 0< t < T_{1},$$

$$ \frac{{}^{PC} D_{T_{1}t}^{\alpha_{2} ,\beta _{2} ,\gamma_{2} ,\delta } T(t)}{T(t)}=\frac{R'(r)}{rR(r)}+\frac{R''(r)}{R(r)}=-\mu^{2}, \qquad T_{1}< t < T_{2},$$

$$ \frac{{}^{PC} D_{T_{2}t}^{\alpha_{3} ,\beta _{3} ,\gamma_{3} ,\delta } T(t)}{T(t)}=\frac{R'(r)}{rR(r)}+\frac{R''(r)}{R(r)}=-\mu^{2}, \qquad T_{2}< t < T.$$
Here $\mu$ is a real number. Easy to see that from these equalities one can get
\begin{equation} \label{8}
    r^{2}R''(r)+rR'(r)+(\mu^{2}r^2)R(r)=0.
\end{equation}

The resulting equation is the Bessel equation \cite{9}. Solution of \eqref{8} satisfying the conditions $rR'(r)|_{r=0}=0,R(1)=0$ has the form
$$ R_{k}=J_{0}(\mu_{k}r), k=1,2,3,...,$$
where $J_{0}$ is zero-order Bessel functions of the first kind \cite{9}. 
Due to the asymptotic form of the Bessel function for the case $\nu=0$ \cite{9}, the value of $\mu_{k}$ can be written as
$$\mu_{k}=k\pi-\frac{\pi}{4}.$$

Let us note on inverse source problems for a classical combination of wave and diffusion equations called as parabolic-hyperbolic equations. In \cite{SS10}, inverse source problem was investigated  for a classical parabolic-hyperbolic equation, later on in \cite{SS15}, similar problem was studied for degenerate parabolic-hyperbolic equation with nonlocal boundary conditions.

We would like also to note several works, in which inverse source problems for a combination of fractional wave and sub-diffusion equations were investigated. In \cite{FK15} inverse source problem was studied for a system of sub-diffusion equation and wave equation with uniformly elliptic operator in space and Caputo fractional-order operator in time-variable. the uniqueness of the solution to nonlocal inverse source problem for mixed-type equation was proved in \cite{SK16} and another inverse source problem with different nonlocal conditions for such equation was targeted in \cite{NSK17}. In \cite{YK21},  the authors consider space-depending inverse source problem for a combination of time-fractional pseudo-parabolic equation with the hyperbolic equation with mixed derivatives. The time-dependent inverse source problem for a system of time-fractional wave and diffusion equations was considered in \cite{KTH24}.   We also note the work \cite{D22}, in which the inverse source problem for parabolic-hyperbolic equation with time-fractional derivative was studied in cylindrical domain.    Similar inverse source problems for parabolic-hyperbolic equations with loaded terms were investigated in \cite{Ab22} (linear) and in \cite{Ab25} (nonlinear). 

\section{Formal solution.}

We search the solution to the problem in the form of a Fourier-Bessel series as follows:
\begin{equation} \label{9}
    u(t,r)=\sum_{k=1}^{\infty}{}_{1}U_{k}(t)J_{0}(\mu_{k}r), (t,r)\in \Omega_{1},
\end{equation}
\begin{equation} \label{10}
    u(t,r)=\sum_{k=1}^{\infty}{}_{2}U_{k}(t)J_{0}(\mu_{k}r), (t,r)\in \Omega_{2},
\end{equation}
\begin{equation} \label{11}
    u(t,r)=\sum_{k=1}^{\infty}{}_{3}U_{k}(t)J_{0}(\mu_{k}r), (t,r)\in \Omega_{3}.
\end{equation}
Here ${}_{1}U_{k}(t),{}_{2}U_{k}(t),{}_{3}U_{k}(t)$ are unknown functions to be found.  

Substituting \eqref{9},\eqref{10},\eqref{11} into the equation \eqref{1}, we obtain the following fractional differential equations:
\begin{equation} \label{12}
    {}^{PC} D_{0t}^{\alpha_{1} ,\beta _{1} ,\gamma_{1} ,\delta } {}_{1}U_{k}(t)+\lambda_{k} \cdot {}_{1}U_{k}(t)=f_{k}, 0< t< T_{1},
\end{equation}
\begin{equation} \label{13}
  {}^{PC} D_{T_{1}t}^{\alpha_{2} ,\beta _{2} ,\gamma_{2} ,\delta } {}_{2}U_{k}(t)+\lambda_{k} \cdot {}_{2}U_{k}(t)=f_{k}, T_{1} \le t \le T_{2},  
\end{equation}
\begin{equation} \label{14}
    {}^{PC} D_{T_{2}t}^{\alpha_{3} ,\beta _{3} ,\gamma_{3} ,\delta } {}_{3}U_{k}(t)+\lambda_{k} \cdot {}_{3}U_{k}(t)=f_{k}, T_{2}< t< T.
\end{equation}
The value of $\lambda_{k}$, which appears in differential equations (10), (11), and (12), is equal to
$$\lambda_{k}=-\left( 4k-1 \right)^2\left( \frac{\pi}{4} \right)^{2},
$$
$f_k$ are the Fourier-Bessel coefficients of the function $f(r)$ given by
$$f_{k}=\frac{2}{J_{1}^2(\mu_{k})}\int_{0}^{1}f(r)J_{0}(\mu_{k}r)rdr.$$

For convenience in our calculations, we will continue to write this value as $ \lambda_{k}$ instead of substituting it. 

General solutions of \eqref{12}, \eqref{13} and  \eqref{14} can be written as follows \cite{11}:

\begin{equation} \label{15}
    \begin{split}
    {}_{1}U_{k} & \left(t\right)= {}_{1} A_{k} +{}_{2} A_{k} t+{}_{1} A_{k}\cdot  \lambda_{k} \cdot t^{\beta _{1} } \cdot \Gamma\left(\gamma_{1} \right)  
\cdot E_{2} \left(\left. \begin{array}{l} {\gamma_{1} ,\gamma_{1} ,1;1,0} \\ {\beta _{1} +1,\beta _{1} ,\alpha_1 ;\gamma_{1} ,\gamma_{1};1,1} \end{array}\right|\begin{array}{c} {\lambda_k t^{\beta _{1} } } \\ {\delta t^{\alpha_{1} } } \end{array}\right)+\\&+{}_{2} A_{k} \cdot \lambda_{k}   
 \cdot t^{\beta _{1} +1}  \Gamma\left(\gamma_{1} \right)\cdot E_{2} \left(\left. \begin{array}{l} {\gamma_{1} ,\gamma_{1} ,1;1,0} \\ {\beta _{1} +2,\beta _{1} ,\alpha_{1} ;\gamma_{1} ,\gamma_{1}; 1,1} \end{array}\right|\begin{array}{c} {\lambda_{k} t^{\beta _{1} } } \\ {\delta t^{\alpha_{1} } } \end{array}\right)+\Gamma\left(\gamma_{1} \right) \cdot f_{k} \cdot t^{\beta_{1}}\times\\ & \times E_{2} \left(\left. \begin{array}{l} {\gamma_{1} ,\gamma_{1} ,1;1,0} \\ {\beta _{1} +1,\beta _{1} ,\alpha_{1} ;\gamma_{1} ,\gamma_{1} ;1,1} \end{array}\right|\begin{array}{c} {\lambda_{k} t^{\beta _{1} } } \\ {\delta t^{\alpha_{1} } } \end{array}\right),
  \end{split}
\end{equation}
 
\begin{equation} \label{16}
    \begin{split}
        {}_{2}U_{k}& \left(t\right)= {}_{3} A_{k} +{}_{3} A_{k} \cdot \lambda_{k} \cdot \left(t-T_{1} \right)^{\beta _{2} } \cdot \Gamma\left(\gamma_{2} \right)
\cdot E_{2} \left(\left. \begin{array}{l} {\gamma_{2} ,\gamma_{2} ,1;1,0} \\ {\beta _{2} +1,\beta _{2} ,\alpha_{2} ;\gamma_{2} ,\gamma_{2} ;1,1} \end{array}\right|\begin{array}{c} {\lambda_{k} \left(t-T_{1} \right)^{\beta _{2} } } \\ {\delta \left(t-T_{1} \right)^{\alpha_{2} } } \end{array}\right)+\\
&+\Gamma\left(\gamma_{2} \right)\cdot f_{k}
\cdot \left(t-T_{1} \right)^{\beta _{2}} E_{2} \left(\left. \begin{array}{l} {\gamma_{2} ,\gamma_{2} ,1;1,0} \\ {\beta _{2} +1,\beta _{2} ,\alpha_{2} ;\gamma_{2} ,\gamma_{2} ;1,1} \end{array}\right|\begin{array}{c} {\lambda_{k} \left(t-T_{1} \right)^{\beta _{2} } } \\ {\delta \left(t-T_{1} \right)^{\alpha_{2} } } \end{array}\right),
    \end{split}
\end{equation}

\begin{equation} \label{17}
    \begin{split}
        {}_{3}U_{k} \left(t\right)&= {}_{4} A_{k} +{}_{5} A_{k} \left(t-T_{2} \right)+{}_{4} A_{k}  \lambda_{k}   \left(t-T_{2} \right)^{\beta _{3} }  \Gamma\left(\gamma_{3} \right) 
 E_{2} \left(\left. \begin{array}{l} {\gamma_{3} ,\gamma_{3} ,1;1,0} \\ {\beta _{3} +1,\beta _{3} ,\alpha_{3} ;\gamma_{3} ,\gamma_{3} ;1,1} \end{array}\right|\begin{array}{c} \lambda_{k}\left(t-T_{2} \right)^{\beta _{3}  } \\ {\delta \left(t-T_{2} \right)^{\alpha_{3} } } \end{array}\right)+\\
&+{}_{5} A_{k}  \lambda_{k}  \left(t-T_{2} \right)^{\beta _{3}+1 } \Gamma (\gamma_{3}) E_{2} \left(\left. \begin{array}{l} {\gamma_{3} ,\gamma_{3} ,1;1,0} \\ {\beta _{3} +2,\beta _{3} ,\alpha_{3} ;\gamma_{3} ,\gamma_{3} ;1,1} \end{array}\right|\begin{array}{c} \lambda_{k} \left(t-T_{2} \right)^{\beta _{3}  } \\ {\delta \left(t-T_{2} \right)^{\alpha_{3} } } \end{array}\right)+\Gamma (\gamma_{3})\left(t-T_{2} \right)^{\beta _{3} } \times\\& \times f_{k }\cdot E_{2} \left(\left. \begin{array}{l} {\gamma_{3} ,\gamma_{3} ,1;1,0} \\ {\beta _{3} +1,\beta _{3} ,\alpha_{3} ;\gamma_{3} ,\gamma_{3} ;1,1} \end{array}\right|\begin{array}{c} \lambda_{k}\left(t-T_{2} \right)^{\beta _{3}  } \\ {\delta \left(t-T_{2} \right)^{\alpha_{3} } } \end{array}\right).
    \end{split}
\end{equation}

Here 
\begin{align*}
    &E_{2} \left(\left. \begin{array}{l} {\gamma _{1} ,\alpha _{1} ,\beta _{1} ;\gamma _{2} ,\alpha _{2} } \\ {\delta _{1} ,\alpha _{3} ,\beta _{2} ;\delta _{2} ,\alpha _{4} ;\delta _{3} ,\beta _{3} } \end{array}\right|\begin{array}{c} {x} \\ {y} \end{array}\right)=\\
&=\sum _{m=0}^{\infty }\sum_{n=0}^{\infty}\frac{\Gamma\left(\gamma _{1} +{\alpha _{1} m+\beta _{1} n}\right) \, \Gamma\left(\gamma _{2} +{\alpha _{2} m}\right) }{\Gamma(\gamma_1)\Gamma(\gamma_2)\Gamma \left(\delta _{1} +\alpha _{3} m+\beta _{2} n\right)} \cdot \frac{x^{m} }{\Gamma \left(\delta _{2} +\alpha _{4} m\right)}  \cdot \frac{y^{n} }{\Gamma \left(\delta _{3} +\beta _{3} n\right)}
\end{align*}
is a bivariate Mittag-Leffler function \cite{14}, ${}_{1} A_{k},{}_{2} A_{k},{}_{3} A_{k},{}_{4} A_{k},{}_{5} A_{k}$ are unknown constants to be found. 

Using $\lim\limits_{t \to T_{1}-}{}_{1} U_{k}(t)=\lim\limits_{t \to T_{1}+}{}_{2} U_{k}(t) $ (which follows from regularity condition $u(t,r)\in C(\overline{\Omega})$ and \eqref{9}-\eqref{10}) and due to \eqref{15}, \eqref{16} we will obtain
\begin{equation} \label{18}
    \begin{split}
        & {}_{1} A_{k}+{}_{2} A_{k} T_{1}+{}_{1} A_{k}\cdot  \lambda_{k} \cdot T_{1}^{\beta _{1} } \cdot \Gamma\left(\gamma_{1} \right)  
\cdot E_{2} \left(\left. \begin{array}{l} {\gamma_{1} ,\gamma_{1} ,1;1,0} \\ {\beta _{1} +1,\beta _{1} ,\alpha_{1} ;\gamma_{1} ,\gamma_{1} ;1,1} \end{array}\right|\begin{array}{c} {\lambda_k T_{1}^{\beta _{1} } } \\ {\delta T_{1}^{\alpha_{1} } } \end{array}\right)+\\&
+{}_{2} A_{k} \cdot \lambda_{k} \cdot  
  T_{1}^{\beta _{1} +1}  \Gamma\left(\gamma_{1} \right) E_{2} \left(\left. \begin{array}{l} {\gamma_{1} ,\gamma_{1} ,1;1,0} \\ {\beta _{1} +2,\beta _{1} ,\alpha_{1} ;\gamma_{1} ,\gamma_{1} ;1,1} \end{array}\right|\begin{array}{c} {\lambda_{k} T_{1}^{\beta _{1} } } \\ {\delta T_{1}^{\alpha_{1} } } \end{array}\right)+\Gamma\left(\gamma_{1} \right) \cdot f_{k} \cdot T_{1}^{\beta_{1}}\times\\&
  \times E_{2} \left(\left. \begin{array}{l} {\gamma_{1} ,\gamma_{1} ,1;1,0} \\ {\beta _{1} +1,\beta _{1} ,\alpha_{1} ;\gamma_{1} ,\gamma_{1} ;1,1} \end{array}\right|\begin{array}{c} {\lambda_{k} T_{1}^{\beta _{1} } } \\ {\delta T_{1}^{\alpha_{1} } } \end{array}\right)={}_{3} A_{k}.
    \end{split}
\end{equation}

Now we use the condition \eqref{2} which can be written as ${}_{1} U_{k}(0)=\varphi_{k}$, where  $\varphi_{k}=\frac{2}{J_{1}(\mu_{k})}\int_{0}^{1}r \varphi(r)J_{0}(\mu_{k}r)dr$. This will give us ${}_{1} A_{k}=\varphi_{k}$.  To use condition \eqref{4}, we need to take the derivative of ${}_{1} U_{k}(t)$ with respect to the variable $t$. After performing the calculations, we obtain the following expression for the derivative of ${}_{1} U_{k}(t)$:
\begin{equation*} 
    \begin{split}
       & {}_{1}U_{k}'   \left(t\right)= {}_{2} A_{k} +{}_{1} A_{k}\cdot  \lambda_{k} \cdot t^{\beta _{1}-1 } \cdot \Gamma\left(\gamma_{1} \right)  
\cdot E_{2} \left(\left. \begin{array}{l} {\gamma_{1} ,\gamma_{1} ,1;1,0} \\ {\beta _{1} ,\beta _{1} ,\alpha_{1} ;\gamma_{1} ,\gamma_{1} ;1,1} \end{array}\right|\begin{array}{c} {\lambda_k t^{\beta _{1} } } \\ {\delta t^{\alpha_{1} } } \end{array}\right)+\\&+{}_{2} A_{k} \cdot \lambda_{k}   
 \cdot t^{\beta _{1} }  \Gamma\left(\gamma_{1} \right)\cdot E_{2} \left(\left. \begin{array}{l} {\gamma_{1} ,\gamma_{1} ,1;1,0} \\ {\beta _{1} +1,\beta _{1} ,\alpha_{1} ;\gamma_{1} ,\gamma_{1} ;1,1} \end{array}\right|\begin{array}{c} {\lambda_{k} t^{\beta _{1} } } \\ {\delta t^{\alpha_{1} } } \end{array}\right)+\Gamma\left(\gamma_{1} \right) \cdot f_{k} \cdot t^{\beta_{1}-1}\times\\ & \times E_{2} \left(\left. \begin{array}{l} {\gamma_{1} ,\gamma_{1} ,1;1,0} \\ {\beta _{1} ,\beta _{1} ,\alpha_{1} ;\gamma_{1} ,\gamma_{1} ;1,1} \end{array}\right|\begin{array}{c} {\lambda_{k} t^{\beta _{1} } } \\ {\delta t^{\alpha_{1} } } \end{array}\right).
    \end{split}
\end{equation*}

Now we need to find $\lim\limits_{t \to T_{1}+0} {}^{PC} D_{T_{1}t}^{\alpha_{2} ,\beta _{2} ,\gamma_{2} ,\delta } {}_{2}U_{k}(t)$. For this aim we let $t \to T_{1}+$ and from the equation \eqref{16} considering ${}_{2}U_{k}(T_{1})={}_{3}A_{k}$, we will get
\begin{equation*}
    \lim\limits_{t \to T_{1}+} {}^{PC} D_{T_{1}t}^{\alpha_{2} ,\beta _{2} ,\gamma_{2} ,\delta } {}_{2}U_{k}(t)=\lim\limits_{t \to T_{1}+}[f_{k}-\lambda_{k} \cdot {}_{2}U_{k}(t)]=f_{k}-\lambda_{k} \cdot {}_{3}A_{k}.
\end{equation*}
Hence, the transmitting condition \eqref{4} gives us
\begin{equation} \label{19}
    \begin{split}
        &f_{k}-\lambda_{k} \cdot {}_{3}A_{k}={}_{2} A_{k} +{}_{1} A_{k}\cdot  \lambda_{k} \cdot T_{1}^{\beta _{1}-1 } \cdot \Gamma\left(\gamma_{1} \right)  
\cdot E_{2} \left(\left. \begin{array}{l} {\gamma_{1} ,\gamma_{1} ,1;1,0} \\ {\beta _{1} ,\beta _{1} ,\alpha_{1} ;\gamma_{1} ,\gamma_{1} ;1,1} \end{array}\right|\begin{array}{c} {\lambda_k T_{1}^{\beta _{1} } } \\ {\delta T_{1}^{\alpha_{1} } } \end{array}\right)+\\&+{}_{2} A_{k} \cdot \lambda_{k}   
 \cdot T_{1}^{\beta _{1} }  \Gamma\left(\gamma_{1} \right)\cdot E_{2} \left(\left. \begin{array}{l} {\gamma_{1} ,\gamma_{1} ,1;1,0} \\ {\beta _{1} +1,\beta _{1} ,\alpha_{1} ;\gamma_{1} ,\gamma_{1} ;1,1} \end{array}\right|\begin{array}{c} {\lambda_{k} T_{1}^{\beta _{1} } } \\ {\delta T_{1}^{\alpha_{1} } } \end{array}\right)+\Gamma\left(\gamma_{1} \right) \cdot f_{k} \cdot T_{1}^{\beta_{1}-1}\times\\ & \times E_{2} \left(\left. \begin{array}{l} {\gamma_{1} ,\gamma_{1} ,1;1,0} \\ {\beta _{1} ,\beta _{1} ,\alpha_{1} ;\gamma_{1} ,\gamma_{1} ;1,1} \end{array}\right|\begin{array}{c} {\lambda_{k} T_{1}^{\beta _{1} } } \\ {\delta T_{1}^{\alpha_{1} } } \end{array}\right).
    \end{split}
\end{equation}
By substituting equation \eqref{18} into left side of equation \eqref{19}, we obtain the following expression:
\begin{align} 
    &f_{k}-\lambda_{k}  \Bigg[ {}_{1} A_{k}+{}_{2} A_{k} T_{1}+{}_{1} A_{k}\cdot  \lambda_{k} \cdot T_{1}^{\beta _{1} } \cdot \Gamma\left(\gamma_{1} \right)  
\cdot E_{2} \left(\left. \begin{array}{l} {\gamma_{1} ,\gamma_{1} ,1;1,0} \\ {\beta _{1} +1,\beta _{1} ,\alpha_{1} ;\gamma_{1} ,\gamma_{1} ;1,1} \end{array}\right|\begin{array}{c} {\lambda_k T_{1}^{\beta _{1} } } \\ {\delta T_{1}^{\alpha_{1} } } \end{array}\right)+\nonumber\\&
+{}_{2} A_{k} \cdot \lambda_{k} \cdot  
  T_{1}^{\beta _{1} +1}  \Gamma\left(\gamma_{1} \right) E_{2} \left(\left. \begin{array}{l} {\gamma_{1} ,\gamma_{1} ,1;1,0} \\ {\beta _{1} +2,\beta _{1} ,\alpha_{1} ;\gamma_{1} ,\gamma_{1} ;1,1} \end{array}\right|\begin{array}{c} {\lambda_{k} T_{1}^{\beta _{1} } } \\ {\delta T_{1}^{\alpha_{1} } } \end{array}\right)+\Gamma\left(\gamma_{1} \right) \cdot f_{k} \cdot T_{1}^{\beta_{1}}\times\nonumber  \\&
  \times E_{2} \left(\left. \begin{array}{l} {\gamma_{1} ,\gamma_{1} ,1;1,0} \\ {\beta _{1} +1,\beta _{1} ,\alpha_{1} ;\gamma_{1} ,\gamma_{1} ;1,1} \end{array}\right|\begin{array}{c} {\lambda_{k} T_{1}^{\beta _{1} } } \\ {\delta T_{1}^{\alpha_{1} } } \end{array}\right)\Bigg]={}_{2} A_{k} +{}_{1} A_{k}\cdot  \lambda_{k} \cdot T_{1}^{\beta _{1}-1 } \cdot \Gamma\left(\gamma_{1} \right)  
\times \nonumber \\& \times E_{2} \left(\left. \begin{array}{l} {\gamma_{1} ,\gamma_{1} ,1;1,0} \\ {\beta _{1} ,\beta _{1} ,\alpha_{1} ;\gamma_{1} ,\gamma_{1} ;1,1} \end{array}\right|\begin{array}{c} {\lambda_k T_{1}^{\beta _{1} } } \\ {\delta T_{1}^{\alpha_{1} } } \end{array}\right)+{}_{2} A_{k}  \lambda_{k}   
  T_{1}^{\beta _{1} }  \Gamma\left(\gamma_{1} \right) E_{2} \left(\left. \begin{array}{l} {\gamma_{1} ,\gamma_{1} ,1;1,0} \\ {\beta _{1} +1,\beta _{1} ,\alpha_{1} ;\gamma_{1} ,\gamma_{1} ;1,1} \end{array}\right|\begin{array}{c} {\lambda_{k} T_{1}^{\beta _{1} } } \\ {\delta T_{1}^{\alpha_{1} } } \end{array}\right)+ \nonumber \\&+\Gamma\left(\gamma_{1} \right) \cdot f_{k} \cdot T_{1}^{\beta_{1}-1} \cdot E_{2} \left(\left. \begin{array}{l} {\gamma_{1} ,\gamma_{1} ,1;1,0} \\ {\beta _{1} ,\beta _{1} ,\alpha_{1} ;\gamma_{1} ,\gamma_{1} ;1,1} \end{array}\right|\begin{array}{c} {\lambda_{k} T_{1}^{\beta _{1} } } \\ {\delta T_{1}^{\alpha_{1} } } \end{array}\right) \nonumber .   
\end{align}

We divide the last equation by $ \lambda_{k} $ :
\begin{equation*}
    \begin{split}
         &{}_{2} A_{k}  \Bigg[\frac{1}{\lambda_{k}}+T_{1}+ T_{1}^{\beta _{1} } \cdot \Gamma\left(\gamma_{1} \right)  
\cdot E_{2} \left(\left. \begin{array}{l} {\gamma_{1} ,\gamma_{1} ,1;1,0} \\ {\beta _{1} +1,\beta _{1} ,\alpha_{1} ;\gamma_{1} ,\gamma_{1} ;1,1} \end{array}\right|\begin{array}{c} {\lambda_k T_{1}^{\beta _{1} } } \\ {\delta T_{1}^{\alpha_{1} } } \end{array}\right)+ T_{1}^{\beta _{1}+1 } \cdot \lambda_{k} \cdot \Gamma\left(\gamma_{1} \right) \times\\
& \times E_{2} \left(\left. \begin{array}{l} {\gamma_{1} ,\gamma_{1} ,1;1,0} \\ {\beta _{1} +2,\beta _{1} ,\alpha_{1} ;\gamma_{1} ,\gamma_{1} ;1,1} \end{array}\right|\begin{array}{c} {\lambda_k T_{1}^{\beta _{1} } } \\ {\delta T_{1}^{\alpha_{1} } } \end{array}\right) \Bigg]=\frac{f_{k}}{\lambda_{k}}-{}_{1} A_{k}-{}_{1} A_{k}\cdot \lambda_{k}\cdot T_{1}^{\beta _{1} } \cdot \Gamma\left(\gamma_{1} \right)  
\times \\ &\times E_{2} \left(\left. \begin{array}{l} {\gamma_{1} ,\gamma_{1} ,1;1,0} \\ {\beta _{1} +1,\beta _{1} ,\alpha_{1} ;\gamma_{1} ,\gamma_{1} ;1,1} \end{array}\right|\begin{array}{c} {\lambda_k T_{1}^{\beta _{1} } } \\ {\delta T_{1}^{\alpha_{1} } } \end{array}\right)-{}_{1} A_{k}T_{1}^{\beta _{1}-1 } \Gamma\left(\gamma_{1} \right) E_{2} \left(\left. \begin{array}{l} {\gamma_{1} ,\gamma_{1} ,1;1,0} \\ {\beta _{1},\beta _{1} ,\alpha_{1} ;\gamma_{1} ,\gamma_{1} ;1,1} \end{array}\right|\begin{array}{c} {\lambda_k T_{1}^{\beta _{1} } } \\ {\delta T_{1}^{\alpha_{1} } } \end{array}\right)-\\
& -\Gamma\left(\gamma_{1} \right) \cdot f_{k} \cdot T_{1}^{\beta _{1} }  E_{2} \left(\left. \begin{array}{l} {\gamma_{1} ,\gamma_{1} ,1;1,0} \\ {\beta _{1} +1,\beta _{1} ,\alpha_{1} ;\gamma_{1} ,\gamma_{1} ;1,1} \end{array}\right|\begin{array}{c} {\lambda_k T_{1}^{\beta _{1} } } \\ {\delta T_{1}^{\alpha_{1} } } \end{array}\right)- \frac{\Gamma\left(\gamma_{1} \right) f_{k}}{\lambda_{k}}T_{1}^{\beta _{1}-1 }  \times\\
& \times E_{2} \left(\left. \begin{array}{l} {\gamma_{1} ,\gamma_{1} ,1;1,0} \\ {\beta _{1},\beta _{1} ,\alpha_{1} ;\gamma_{1} ,\gamma_{1} ;1,1} \end{array}\right|\begin{array}{c} {\lambda_k T_{1}^{\beta _{1} } } \\ {\delta T_{1}^{\alpha_{1} } } \end{array}\right).
    \end{split}
\end{equation*}
If $\Delta_{k} \neq 0 $ we could find ${}_{2} A_{k}$ as follows:
\begin{align} \label{20}
         &{}_{2} A_{k}=\frac{1}{\Delta_{k}}\Bigg[\frac{f_{k}}{\lambda_{k}}-{}_{1} A_{k}-{}_{1} A_{k}\cdot \lambda_{k}\cdot T_{1}^{\beta _{1} } \cdot \Gamma\left(\gamma_{1} \right)  
 E_{2} \left(\left. \begin{array}{l} {\gamma_{1} ,\gamma_{1} ,1;1,0} \\ {\beta _{1} +1,\beta _{1} ,\alpha_{1} ;\gamma_{1} ,\gamma_{1} ;1,1} \end{array}\right|\begin{array}{c} {\lambda_k T_{1}^{\beta _{1} } } \\ {\delta T_{1}^{\alpha_{1} } } \end{array}\right)+\nonumber\\
&-{}_{1} A_{k}T_{1}^{\beta _{1}-1 } \Gamma\left(\gamma_{1} \right) E_{2} \left(\left. \begin{array}{l} {\gamma_{1} ,\gamma_{1} ,1;1,0} \\ {\beta _{1},\beta _{1} ,\alpha_{1} ;\gamma_{1} ,\gamma_{1} ;1,1} \end{array}\right|\begin{array}{c} {\lambda_k T_{1}^{\beta _{1} } } \\ {\delta T_{1}^{\alpha_{1} } } \end{array}\right) -\Gamma\left(\gamma_{1} \right)  f_{k} T_{1}^{\beta_{1}}\times \\
& \times E_{2} \left(\left. \begin{array}{l} {\gamma_{1} ,\gamma_{1} ,1;1,0} \\ {\beta _{1} +1,\beta _{1} ,\alpha_{1} ;\gamma_{1} ,\gamma_{1} ;1,1} \end{array}\right|\begin{array}{c} {\lambda_k T_{1}^{\beta _{1} } } \\ {\delta T_{1}^{\alpha_{1} } } \end{array}\right)- \frac{\Gamma\left(\gamma_{1} \right) f_{k}T_{1}^{\beta _{1}-1 } }{\lambda_{k}}  E_{2} \left(\left. \begin{array}{l} {\gamma_{1} ,\gamma_{1} ,1;1,0} \\ {\beta _{1},\beta _{1} ,\alpha_{1} ;\gamma_{1} ,\gamma_{1} ;1,1} \end{array}\right|\begin{array}{c} {\lambda_k T_{1}^{\beta _{1} } } \\ {\delta T_{1}^{\alpha_{1} } } \end{array}\right)\Bigg]. \nonumber
\end{align}

Here 
\begin{equation} \label{21}
    \begin{split}
        \Delta_{k}=&\frac{1}{\lambda_{k}}+T_{1}+ T_{1}^{\beta _{1} } \cdot \Gamma\left(\gamma_{1} \right)  
\cdot E_{2} \left(\left. \begin{array}{l} {\gamma_{1} ,\gamma_{1} ,1;1,0} \\ {\beta _{1} +1,\beta _{1} ,\alpha_{1} ;\gamma_{1} ,\gamma_{1} ;1,1} \end{array}\right|\begin{array}{c} {\lambda_k T_{1}^{\beta _{1} } } \\ {\delta T_{1}^{\alpha_{1} } } \end{array}\right)+ T_{1}^{\beta _{1}+1 } \cdot \lambda_{k} \cdot \Gamma\left(\gamma_{1} \right)\times\\
& \times E_{2} \left(\left. \begin{array}{l} {\gamma_{1} ,\gamma_{1} ,1;1,0} \\ {\beta _{1} +2,\beta _{1} ,\alpha_{1} ;\gamma_{1} ,\gamma_{1} ;1,1} \end{array}\right|\begin{array}{c} {\lambda_k T_{1}^{\beta _{1} } } \\ {\delta T_{1}^{\alpha_{1} } } \end{array}\right) \Bigg].
    \end{split}
\end{equation}

\begin{lemma} \cite{12}
If $\beta_{1}>\gamma_{1}, \alpha_{1}>0$, then the equality $\lim\limits_{k \to \infty} \Delta_{k}=T_{1}>0$ holds. 
\end{lemma}

Now, let's start finding the coefficients ${}_{4} A_{k}$ and ${}_{5} A_{k}$. For this aim we will use $u(t,r) \in C(\overline{\Omega})$ which gives us: 
\begin{equation} \label{22}
    \lim\limits_{t \to T_{2}-} {}_{2} U_{k} (t)=\lim\limits_{t \to T_{2}+} {}_{3} U_{k} (t).
\end{equation}

Based on relation \eqref{22} and equalities \eqref{16}, \eqref{17}, we find the coefficient ${}_{4} A_{k}$ as follows:
\begin{equation} \label{23}
    \begin{split}
        &{}_{3} A_{k} +{}_{3} A_{k} \cdot \lambda_{k} \cdot \left(T_{2}-T_{1} \right)^{\beta _{2} } \cdot \Gamma\left(\gamma_{2} \right)
\cdot E_{2} \left(\left. \begin{array}{l} {\gamma_{2} ,\gamma_{2} ,1;1,0} \\ {\beta _{2} +1,\beta _{2} ,\alpha_{2} ;\gamma_{2} ,\gamma_{2} ;1,1} \end{array}\right|\begin{array}{c} {\lambda_{k} \left(T_{2}-T_{1} \right)^{\beta _{2} } } \\ {\delta \left(T_{2}-T_{1} \right)^{\alpha_{2} } } \end{array}\right)+\\
&+\Gamma\left(\gamma_{2} \right)\cdot f_{k}
\cdot \left(T_{2}-T_{1} \right)^{\beta _{2}} E_{2} \left(\left. \begin{array}{l} {\gamma_{2} ,\gamma_{2} ,1;1,0} \\ {\beta _{2} +1,\beta _{2} ,\alpha_{2} ;\gamma_{2} ,\gamma_{2} ;1,1} \end{array}\right|\begin{array}{c} {\lambda_{k} \left(T_{2}-T_{1} \right)^{\beta _{2} } } \\ {\delta \left(T_{2}-T_{1} \right)^{\alpha_{2} } } \end{array}\right)={}_{4} A_{k}.
    \end{split}
\end{equation}
In order to use transmitting condition \eqref{5}, we need value limit of ${}^{PC} D_{T_{1}t}^{\alpha_{2} ,\beta _{2} ,\gamma_{2} ,\delta } u(t,r)$ as $t \to T_{2}-0$. For this aim, we pass to the limit as $t \to T_{2}-0$ and from the equation \eqref{13} we obtain:
\begin{equation} \label{24}
    \lim\limits_{t\to T_{2}-0}{}^{PC} D_{T_{1}t}^{\alpha_{2} ,\beta _{2} ,\gamma_{2} ,\delta } {}_{2} U_{k}(t)=\lim\limits_{t\to T_{2}-0}-\lambda_{k}\cdot {}_{2} U_{k}(t)+f_{k}=-\lambda_{k}\lim\limits_{t\to T_{2}-0} {}_{2} U_{k}(t)+f_{k}=-\lambda_{k}\cdot {}_{4}  A_{k}+f_{k}.
\end{equation}
To find ${}_{3} U'_{k}(t)$,we differentiate \eqref{17} and obtain
\begin{equation} \label{25}
    \begin{split}
        &{}_{3}U'_{k} \left(t\right)= {}_{5} A_{k} +{}_{4} A_{k}  \lambda_{k}   \left(t-T_{2} \right)^{\beta _{3}-1 }  \Gamma\left(\gamma_{3} \right)  
 E_{2} \left(\left. \begin{array}{l} {\gamma_{3} ,\gamma_{3} ,1;1,0} \\ {\beta _{3} ,\beta _{3} ,\alpha_{3} ;\gamma_{3} ,\gamma_{3} ;1,1} \end{array}\right|\begin{array}{c} \lambda_{k}\left(t-T_{2} \right)^{\beta _{3}  } \\ {\delta \left(t-T_{2} \right)^{\alpha_{3}} } \end{array}\right)+\\
&+{}_{5} A_{k}  \lambda_{k}  \left(t-T_{2} \right)^{\beta _{3} } \Gamma (\gamma_{3}) E_{2} \left(\left. \begin{array}{l} {\gamma_{3} ,\gamma_{3} ,1;1,0} \\ {\beta _{3} +1,\beta _{3} ,\alpha_{3} ;\gamma_{3} ,\gamma_{3} ;1,1} \end{array}\right|\begin{array}{c} \lambda_{k} \left(t-T_{2} \right)^{\beta _{3}  } \\ {\delta \left(t-T_{2} \right)^{\alpha_{3} } } \end{array}\right)+\Gamma (\gamma_{3})\left(t-T_{2} \right)^{\beta _{3}-1 } \times\\& \times f_{k }\cdot E_{2} \left(\left. \begin{array}{l} {\gamma_{3} ,\gamma_{3} ,1;1,0} \\ {\beta _{3} ,\beta _{3} ,\alpha_{3} ;\gamma_{3} ,\gamma_{3} ;1,1} \end{array}\right|\begin{array}{c} \lambda_{k}\left(t-T_{2} \right)^{\beta _{3}  } \\ {\delta \left(t-T_{2} \right)^{\alpha_{3} } } \end{array}\right).
    \end{split}
\end{equation}
Based on the transmitting condition \eqref{5} and equations \eqref{24},\eqref{25} we deduce that
\begin{equation} \label{26}
  {}_{5} A_{k}=  -\lambda_{k}\cdot {}_{4}  A_{k}+f_{k}.
\end{equation}

\section{Inverse source problem.} 

Using condition \eqref{6} and equality \eqref{16}, we can write down the following:
\begin{align*}
     \psi_{k}=&{}_{3} A_{k} +{}_{3} A_{k} \cdot \lambda_{k} \cdot \left(\xi-T_{1} \right)^{\beta _{2} } \cdot \Gamma\left(\gamma_{2} \right)
\cdot E_{2} \left(\left. \begin{array}{l} {\gamma_{2} ,\gamma_{2} ,1;1,0} \\ {\beta _{2} +1,\beta _{2} ,\alpha_{2} ;\gamma_{2} ,\gamma_{2} ;1,1} \end{array}\right|\begin{array}{c} {\lambda_{k} \left(\xi-T_{1} \right)^{\beta _{2} } } \\ {\delta \left(\xi-T_{1} \right)^{\alpha_{2} } } \end{array}\right)+\\
&+\Gamma\left(\gamma_{2} \right)\cdot f_{k}
\cdot \left(\xi-T_{1} \right)^{\beta _{2}} E_{2} \left(\left. \begin{array}{l} {\gamma_{2} ,\gamma_{2} ,1;1,0} \\ {\beta _{2} +1,\beta _{2} ,\alpha_{2} ;\gamma_{2} ,\gamma_{2} ;1,1} \end{array}\right|\begin{array}{c} {\lambda_{k} \left(\xi-T_{1} \right)^{\beta _{2} } } \\ {\delta \left(\xi-T_{1} \right)^{\alpha _{2}} } \end{array}\right).
\end{align*}
From this equation, we determine $f_{k}$:
 \begin{equation} \label{27}
     \begin{split}
          & f_{k}= \frac{\psi_{k}-{}_{3} A_{k}\left( 1 +\lambda_{k} \cdot \left(\xi-T_{1} \right)^{\beta _{2} } \cdot \Gamma\left(\gamma_{2} \right)
\cdot E_{2} \left(\left. \begin{array}{l} {\gamma_{2} ,\gamma_{2} ,1;1,0} \\ {\beta _{2} +1,\beta _{2} ,\alpha_{2} ;\gamma_{2} ,\gamma_{2} ;1,1} \end{array}\right|\begin{array}{c} {\lambda_{k} \left(\xi-T_{1} \right)^{\beta _{2} } } \\ {\delta \left(\xi-T_{1} \right)^{\alpha_{2} } } \end{array}\right) \right)}{\Gamma\left(\gamma_{2} \right)\left(\xi-T_{1} \right)^{\beta _{2}} E_{2} \left(\left. \begin{array}{l} {\gamma_{2} ,\gamma_{2} ,1;1,0} \\ {\beta _{2} +1,\beta _{2} ,\alpha_{2} ;\gamma_{2} ,\gamma_{2} ;1,1} \end{array}\right|\begin{array}{c} {\lambda_{k} \left(\xi-T_{1} \right)^{\beta _{2} } } \\ {\delta \left(\xi-T_{1} \right)^{\alpha_{2} } } \end{array}\right)}.
     \end{split}
 \end{equation}
We will introduce the following notation:
\begin{equation} \label{28}
    {}_{1} M_{k}=\Gamma\left(\gamma_{2} \right)\left(\xi-T_{1} \right)^{\beta _{2}} E_{2} \left(\left. \begin{array}{l} {\gamma_{2} ,\gamma_{2} ,1;1,0} \\ {\beta _{2} +1,\beta _{2} ,\alpha_{2} ;\gamma_{2} ,\gamma_{2} ;1,1} \end{array}\right|\begin{array}{c} {\lambda_{k} \left(\xi-T_{1} \right)^{\beta _{2} } } \\ {\delta \left(\xi-T_{1} \right)^{\alpha_{2} } } \end{array}\right).
\end{equation}
As a result, we can rewrite equation \eqref{27} as follows:
\begin{equation} \label{29}
    f_{k}= \frac{\psi_{k}-{}_{3} A_{k}\left( 1+\lambda_{k} \cdot {}_{1} M_{k} \right)  }{{}_{1} M_{k}}  \,\,\, \text{or} \,\,\, {}_{3} A_{k}\left( 1+\lambda_{k} \cdot {}_{1} M_{k} \right)+ f_{k} \cdot {}_{1} M_{k}=\psi_{k}.
\end{equation}
Note that here ${}_1M_k\neq 0$, if $\alpha_2=1$ and $\gamma_2=\beta_2$ (see for the proof in \cite{KKT26}).  

Based on equation \eqref{18}, we can write ${}_{3} A_{k}$ in a simpler form as follows:
\begin{equation} \label{30}
    \begin{split}
         & {}_{3} A_{k}={}_{1} A_{k}\left( 1+ \lambda_{k} \cdot T_{1}^{\beta _{1} } \cdot \Gamma\left(\gamma_{1} \right)  
\cdot E_{2} \left(\left. \begin{array}{l} {\gamma_{1} ,\gamma_{1} ,1;1,0} \\ {\beta _{1} +1,\beta _{1} ,\alpha_{1} ;\gamma_{1} ,\gamma_{1} ;1,1} \end{array}\right|\begin{array}{c} {\lambda_k T_{1}^{\beta _{1} } } \\ {\delta T_{1}^{\alpha_{1} } } \end{array}\right) \right)+\\&
+{}_{2} A_{k} \left( T_{1}+ \lambda_{k} 
  T_{1}^{\beta _{1} +1}  \Gamma\left(\gamma_{1} \right) E_{2} \left(\left. \begin{array}{l} {\gamma_{1} ,\gamma_{1} ,1;1,0} \\ {\beta _{1} +2,\beta _{1} ,\alpha_{1} ;\gamma_{1} ,\gamma_{1} ;1,1} \end{array}\right|\begin{array}{c} {\lambda_{k} T_{1}^{\beta _{1} } } \\ {\delta T_{1}^{\alpha_{1} } } \end{array}\right) \right)+\Gamma\left(\gamma_{1} \right)  f_{k} T_{1}^{\beta_{1}}\times\\&
  \times E_{2} \left(\left. \begin{array}{l} {\gamma_{1} ,\gamma_{1} ,1;1,0} \\ {\beta _{1} +1,\beta _{1} ,\alpha_{1} ;\gamma_{1} ,\gamma_{1} ;1,1} \end{array}\right|\begin{array}{c} {\lambda_{k} T_{1}^{\beta _{1} } } \\ {\delta T_{1}^{\alpha_{1} } } \end{array}\right).
    \end{split}
\end{equation}
For convenience, we will introduce the following notations:
\begin{equation} \label{31}
    {}_{2}M_{k}= T_{1}^{\beta _{1} } \cdot \Gamma\left(\gamma_{1} \right)  
\cdot E_{2} \left(\left. \begin{array}{l} {\gamma_{1} ,\gamma_{1} ,1;1,0} \\ {\beta _{1} +1,\beta _{1} ,\alpha_{1} ;\gamma_{1} ,\gamma_{1} ;1,1} \end{array}\right|\begin{array}{c} {\lambda_k T_{1}^{\beta _{1} } } \\ {\delta T_{1}^{\alpha_{1} } } \end{array}\right),
\end{equation}
\begin{equation} \label{32}
    {}_{3}M_{k}= T_{1}+T_{1}^{\beta _{1}+1 } \cdot \lambda_{k} \cdot \Gamma\left(\gamma_{1} \right)  
\cdot E_{2} \left(\left. \begin{array}{l} {\gamma_{1} ,\gamma_{1} ,1;1,0} \\ {\beta _{1} +2,\beta _{1} ,\alpha_1 ;\gamma_{1} ,\gamma_{1} ;1,1} \end{array}\right|\begin{array}{c} {\lambda_k T_{1}^{\beta _{1} } } \\ {\delta T_{1}^{\alpha_{1} } } \end{array}\right).
\end{equation}
Let us rewrite the form of equation \eqref{30}:
\begin{equation} \label{33}
    {}_{2} A_{k}\cdot  {}_{3} M_{k}-{}_{3} A_{k}+ f_{k}\cdot{}_{2} M_{k}=-{}_{1} A_{k}\left( 1+ \lambda_{k}\cdot {}_{2} M_{k} \right)
\end{equation}
We simplify \eqref{19} as follows and introduce the notation:
\begin{equation*}
\begin{aligned}
&{}_{1} A_{k}\cdot  \lambda_{k} \cdot T_{1}^{\beta _{1}-1 } \cdot \Gamma\left(\gamma_{1} \right)  
\cdot E_{2} \left(\left. \begin{array}{l} {\gamma_{1} ,\gamma_{1} ,1;1,0} \\ {\beta _{1} ,\beta _{1} ,\alpha_{1} ;\gamma_{1} ,\gamma_{1} ;1,1} \end{array}\right|\begin{array}{c} {\lambda_k T_{1}^{\beta _{1} } } \\ {\delta T_{1}^{\alpha_{1} } } \end{array}\right)+{}_{2} A_{k}\Bigg( 1+ \lambda_{k}   
 \cdot T_{1}^{\beta _{1} }  \Gamma\left(\gamma_{1} \right)\times\\
& E_{2} \left(\left. \begin{array}{l} {\gamma_{1} ,\gamma_{1} ,1;1,0} \\ {\beta _{1} +1,\beta _{1} ,\alpha_{1} ;\gamma_{1} ,\gamma_{1} ;1,1} \end{array}\right|\begin{array}{c} {\lambda_{k} T_{1}^{\beta _{1} } } \\ {\delta T_{1}^{\alpha _{1}} } \end{array}\right)\Bigg)+\lambda_{k} \cdot {}_{3} A_{k}+ f_{k} \Bigg(\Gamma\left(\gamma_{1} \right) \cdot  T_{1}^{\beta_{1}-1}\times\\
& \times E_{2} \left(\left. \begin{array}{l} {\gamma_{1} ,\gamma_{1} ,1;1,0} \\ {\beta _{1} ,\beta _{1} ,\alpha_{1} ;\gamma_{1} ,\gamma_{1} ;1,1} \end{array}\right|\begin{array}{c} {\lambda_{k} T_{1}^{\beta _{1} } } \\ {\delta T_{1}^{\alpha_{1} } } \end{array}\right)-1\Bigg)=0
,\end{aligned}
\end{equation*}

\begin{equation} \label{34}
  {}_{4} M_{k} =T_{1}^{\beta _{1}-1 } \cdot \Gamma\left(\gamma_{1} \right)  
\cdot E_{2} \left(\left. \begin{array}{l} {\gamma_{1} ,\gamma_{1} ,1;1,0} \\ {\beta _{1} ,\beta _{1} ,\alpha_{1} ;\gamma_{1} ,\gamma_{1} ;1,1} \end{array}\right|\begin{array}{c} {\lambda_k T_{1}^{\beta _{1} } } \\ {\delta T_{1}^{\alpha_{1} } } \end{array}\right).
\end{equation}
As a result, we have the following expression:
\begin{equation} \label{35}
     {}_{2}A_{k}(1+\lambda_{k}\cdot {}_{2}M_{k})+{}_{3}A_{k}\lambda_{k}+f_{k}({}_{4}M_{k}-1)=-\lambda_{k} \cdot {}_{1}A_{k}\cdot{}_{4}M_{k}
\end{equation}
Using equations \eqref{29},\eqref{33} and \eqref{35}, we form a system of equations with respect to unknown coefficients and then, based on condition \eqref{2}, we substitute $\varphi_{k}$ instead of ${}_{1}A_{k}$.

\begin{equation} \label{36}
\begin{cases}
{}_{3} A_{k}\left( 1+\lambda_{k} \cdot {}_{1} M_{k} \right)+ f_{k} \cdot {}_{1} M_{k}=\psi_{k}, \\
{}_{2} A_{k}\cdot  {}_{3} M_{k}-{}_{3} A_{k}+ f_{k}\cdot{}_{2} M_{k}=-\varphi_{k}\left( 1+ \lambda_{k}\cdot {}_{2} M_{k} \right),\\
{}_{2}A_{k}(1+\lambda_{k}\cdot {}_{2}M_{k})+{}_{3}A_{k}\lambda_{k}+f_{k}({}_{4}M_{k}-1)=-\lambda_{k} \cdot \varphi_{k}\cdot{}_{4}M_{k}.
\end{cases}
\end{equation}

From the resulting system of equations \eqref{36}, we determine the unknowns ${}_{2} A_{k},{}_{3} A_{k}$ and $f_{k}$ using Cramer's formulas.

In this case, we first identify the main determinant of the system. It is known that showing that this determinant is different from zero means that the system has a unique solution.
\begin{equation} \label{37}
    \begin{split}
&        \widetilde{\Delta}_{k}=\begin{vmatrix}  0 & 1+\lambda_{k}\cdot {}_{1}M_{k} & {}_{1}M_{k}\\ {}_{3}M_{k} & -1 & {}_{2}M_{k} \\ 1+\lambda_{k}\cdot {}_{2}M_{k} & \lambda_{k} & {}_{4}M_{k}-1 \end{vmatrix}=0 \cdot(-1) \cdot \left( {}_{4}M_{k}-1 \right)+\left( 1+\lambda_{k}\cdot {}_{1}M_{k} \right) {}_{2}M_{k}\times\\ \\
&\times (1+\lambda_{k}\cdot {}_{2}M_{k})+\lambda_{k} \cdot{}_{3}M_{k}\cdot{}_{1}M_{k}+{}_{1}M_{k}\left( 1+\lambda_{k}\cdot {}_{2}M_{k} \right)-(1+\lambda_{k}\cdot {}_{1}M_{k}){}_{3}M_{k}\left( {}_{4}M_{k}-1 \right)-0 \cdot \lambda_{k} \cdot {}_{2}M_{k}=\\ \\
&={}_{1}M_{k}+{}_{2}M_{k}+{}_{3}M_{k}+2\lambda_{k}\cdot{}_{1}M_{k}\cdot{}_{2}M_{k}+\lambda_{k}\cdot{}_{2}M_{k}^{2}+\lambda_{k}^{2}\cdot {}_{1}M_{k} \cdot {}_{2}M_{k}^{2}+2\lambda_{k}\cdot {}_{1}M_{k} \cdot{}_{3}M_{k}-{}_{3}M_{k}\cdot {}_{4}M_{k}-\\ \\
&-\lambda_{k}
\cdot {}_{1}M_{k}\cdot{}_{3}M_{k}\cdot{}_{4}M_{k}.
    \end{split}
\end{equation}
The expression \eqref{37} we rewrite as follows (see A1 in Appendix section):
\begin{equation}\label{39}
    \begin{split}
        &\widetilde{\Delta}_{k}=\Gamma\left(\gamma_{2} \right)\left(\xi-T_{1} \right)^{\beta _{2}} E_{2} \left(\left. \begin{array}{l} {\gamma_{2} ,\gamma_{2} ,1;1,0} \\ {\beta _{2} +1,\beta _{2} ,\alpha_{2} ;\gamma_{2} ,\gamma_{2} ;1,1} \end{array}\right|\begin{array}{c} {\lambda_{k} \left(\xi-T_{1} \right)^{\beta _{2} } } \\ {\delta \left(\xi-T_{1} \right)^{\alpha_{2} } } \end{array}\right)+T_{1}+T_{1}^{\beta_{1}} \Gamma(\gamma_{1})\times\\
&\times E_{2} \left(\left. \begin{array}{l} {\gamma_{1} ,\gamma_{1} ,1;1,0} \\ {\beta _{1} +1,\beta _{1} ,\alpha_{1} ;\gamma_{1} ,\gamma_{1} ;1,1} \end{array}\right|\begin{array}{c} {\lambda_k T_{1}^{\beta _{1} } } \\ {\delta T_{1}^{\alpha_{1} } } \end{array}\right)+T_{1}^{\beta_{1}+1} \lambda_{k} \Gamma(\gamma_{1})E_{2} \left(\left. \begin{array}{l} {\gamma_{1} ,\gamma_{1} ,1;1,0} \\ {\beta _{1} +2,\beta _{1} ,\alpha_{1} ;\gamma_{1} ,\gamma_{1} ;1,1} \end{array}\right|\begin{array}{c} {\lambda_k T_{1}^{\beta _{1} } } \\ {\delta T_{1}^{\alpha_{1} } } \end{array}\right)+ \\
&+2 \lambda_{k}T_{1}^{\beta_{1}}\Gamma\left(\gamma_{1} \right)\Gamma\left(\gamma_{2} \right)\left(\xi-T_{1} \right)^{\beta _{2}} E_{2} \left(\left. \begin{array}{l} {\gamma_{2} ,\gamma_{2} ,1;1,0} \\ {\beta _{2} +1,\beta _{2} ,\alpha_{2} ;\gamma_{2} ,\gamma_{2} ;1,1} \end{array}\right|\begin{array}{c} {\lambda_{k} \left(\xi-T_{1} \right)^{\beta _{2} } } \\ {\delta \left(\xi-T_{1} \right)^{\alpha_{2} } } \end{array}\right)\times \\
&\times E_{2} \left(\left. \begin{array}{l} {\gamma_{1} ,\gamma_{1} ,1;1,0} \\ {\beta _{1} +1,\beta _{1} ,\alpha_{1} ;\gamma_{1} ,\gamma_{1} ;1,1} \end{array}\right|\begin{array}{c} {\lambda_k T_{1}^{\beta _{1} } } \\ {\delta T_{1}^{\alpha_{1} } } \end{array}\right)+
\Gamma^{2}\left(\gamma_{1} \right)\left( E_{2} \left(\left. \begin{array}{l} {\gamma_{1} ,\gamma_{1} ,1;1,0} \\ {\beta _{1} +1,\beta _{1} ,\alpha_{1} ;\gamma_{1} ,\gamma_{1} ;1,1} \end{array}\right|\begin{array}{c} {\lambda_k T_{1}^{\beta _{1} } } \\ {\delta T_{1}^{\alpha_{1} } } \end{array}\right) \right)^{2} \times\\
& \times \lambda_{k}T_{1}^{2\beta_{1}}+\lambda_{k}^{2} \cdot\Gamma\left(\gamma_{2} \right)\cdot T_{1}^{2\beta_{1}}\Gamma^{2}\left(\gamma_{1} \right)\left(\xi-T_{1} \right)^{\beta _{2}} E_{2} \left(\left. \begin{array}{l} {\gamma_{2} ,\gamma_{2} ,1;1,0} \\ {\beta _{2} +1,\beta _{2} ,\alpha_{2} ;\gamma_{2} ,\gamma_{2} ;1,1} \end{array}\right|\begin{array}{c} {\lambda_{k} \left(\xi-T_{1} \right)^{\beta _{2} } } \\ {\delta \left(\xi-T_{1} \right)^{\alpha_{2} } } \end{array}\right) \times \\
&\times\left( E_{2} \left(\left. \begin{array}{l} {\gamma_{1} ,\gamma_{1} ,1;1,0} \\ {\beta _{1} +1,\beta _{1} ,\alpha_{1} ;\gamma_{1} ,\gamma_{1} ;1,1} \end{array}\right|\begin{array}{c} {\lambda_k T_{1}^{\beta _{1} } } \\ {\delta T_{1}^{\alpha_{1} } } \end{array}\right) \right)^{2}+2\lambda_{k}\Gamma\left(\gamma_{2} \right)\left(\xi-T_{1} \right)^{\beta _{2}} T_{1}\times \\
& \times E_{2} \left(\left. \begin{array}{l} {\gamma_{2} ,\gamma_{2} ,1;1,0} \\ {\beta _{2} +1,\beta _{2} ,\alpha_{2} ;\gamma_{2} ,\gamma_{2} ;1,1} \end{array}\right|\begin{array}{c} {\lambda_{k} \left(\xi-T_{1} \right)^{\beta _{2} } } \\ {\delta \left(\xi-T_{1} \right)^{\alpha_{2} } } \end{array}\right)+2\lambda_{k}^{2}\Gamma\left(\gamma_{1} \right)\Gamma\left(\gamma_{2} \right)\left(\xi-T_{1} \right)^{\beta _{2}} T_{1}^{\beta_{1}+1}\times\\
& \times E_{2} \left(\left. \begin{array}{l} {\gamma_{2} ,\gamma_{2} ,1;1,0} \\ {\beta _{2} +1,\beta _{2} ,\alpha_{2} ;\gamma_{2} ,\gamma_{2} ;1,1} \end{array}\right|\begin{array}{c} {\lambda_{k} \left(\xi-T_{1} \right)^{\beta _{2} } } \\ {\delta \left(\xi-T_{1} \right)^{\alpha_{2} } } \end{array}\right)E_{2} \left(\left. \begin{array}{l} {\gamma_{1} ,\gamma_{1} ,1;1,0} \\ {\beta _{1} +2,\beta _{1} ,\alpha_{1} ;\gamma_{1} ,\gamma_{1} ;1,1} \end{array}\right|\begin{array}{c} {\lambda_k T_{1}^{\beta _{1} } } \\ {\delta T_{1}^{\alpha_{1} } } \end{array}\right)-\\
&-T_{1}^{\beta_{1}}  \Gamma(\gamma_{1})E_{2} \left(\left. \begin{array}{l} {\gamma_{1} ,\gamma_{1} ,1;1,0} \\ {\beta _{1} ,\beta _{1} ,\alpha_{1} ;\gamma_{1} ,\gamma_{1} ;1,1} \end{array}\right|\begin{array}{c} {\lambda_k T_{1}^{\beta _{1} } } \\ {\delta T_{1}^{\alpha } } \end{array}\right)-  \Gamma^{2}(\gamma_{1})E_{2} \left(\left. \begin{array}{l} {\gamma_{1} ,\gamma_{1} ,1;1,0} \\ {\beta _{1}+2 ,\beta _{1} ,\alpha_{1} ;\gamma_{1} ,\gamma_{1} ;1,1} \end{array}\right|\begin{array}{c} {\lambda_k T_{1}^{\beta _{1} } } \\ {\delta T_{1}^{\alpha_{1} } } \end{array}\right) \times \\
&\times \lambda_{k} T_{1}^{2\beta_{1}} E_{2} \left(\left. \begin{array}{l} {\gamma_{1} ,\gamma_{1} ,1;1,0} \\ {\beta _{1} ,\beta _{1} ,\alpha_{1} ;\gamma_{1} ,\gamma_{1} ;1,1} \end{array}\right|\begin{array}{c} {\lambda_k T_{1}^{\beta _{1} } } \\ {\delta T_{1}^{\alpha_{1} } } \end{array}\right)-T_{1}^{\beta_{1}} E_{2} \left(\left. \begin{array}{l} {\gamma_{2} ,\gamma_{2} ,1;1,0} \\ {\beta _{2} +1,\beta _{2} ,\alpha_{2} ;\gamma_{2} ,\gamma_{2} ;1,1} \end{array}\right|\begin{array}{c} {\lambda_{k} \left(\xi-T_{1} \right)^{\beta _{2} } } \\ {\delta \left(\xi-T_{1} \right)^{\alpha_{2} } } \end{array}\right)\times\\
&\times \Gamma(\gamma_{1})\Gamma(\gamma_{2})\lambda_{k} \left(\xi-T_{1} \right)^{\beta _{2}}E_{2} \left(\left. \begin{array}{l} {\gamma_{1} ,\gamma_{1} ,1;1,0} \\ {\beta _{1} ,\beta _{1} ,\alpha_{1} ;\gamma_{1} ,\gamma_{1} ;1,1} \end{array}\right|\begin{array}{c} {\lambda_k T_{1}^{\beta _{1} } } \\ {\delta T_{1}^{\alpha_{1} } } \end{array}\right)-\lambda^{2}_{k}\Gamma^{2}(\gamma_{1})\Gamma(\gamma_{2})T_{1}^{2\beta_{1}}\left(\xi-T_{1} \right)^{\beta _{2}}\times\\
&\times E_{2} \left(\left. \begin{array}{l} {\gamma_{1} ,\gamma_{1} ,1;1,0} \\ {\beta _{1} ,\beta _{1} ,\alpha_{1} ;\gamma_{1} ,\gamma_{1} ;1,1} \end{array}\right|\begin{array}{c} {\lambda_k T_{1}^{\beta _{1} } } \\ {\delta T_{1}^{\alpha_{1} } } \end{array}\right)E_{2} \left(\left. \begin{array}{l} {\gamma_{1} ,\gamma_{1} ,1;1,0} \\ {\beta _{1}+2 ,\beta _{1} ,\alpha_{1} ;\gamma_{1} ,\gamma_{1} ;1,1} \end{array}\right|\begin{array}{c} {\lambda_k T_{1}^{\beta _{1} } } \\ {\delta T_{1}^{\alpha_{1} } } \end{array}\right) \times\\
&\times E_{2} \left(\left. \begin{array}{l} {\gamma_{2} ,\gamma_{2} ,1;1,0} \\ {\beta _{2} +1,\beta _{2} ,\alpha_{2} ;\gamma_{2} ,\gamma_{2} ;1,1} \end{array}\right|\begin{array}{c} {\lambda_{k} \left(\xi-T_{1} \right)^{\beta _{2} } } \\ {\delta \left(\xi-T_{1} \right)^{\alpha_{2} } } \end{array}\right).
    \end{split}
\end{equation}
\begin{lemma}
If $\beta_{1}>\gamma_{1},\beta_{2}>\gamma_{2}, \alpha_{1}>0, \alpha_{2}>0,$ then the equality $ \lim\limits_{k \to \infty} \widetilde{\Delta}_{k}=T_{1}>0 $ holds.
\end{lemma}
\begin{proof}
 Based on the expression of $\widetilde{\Delta}_{k}$ given above, we can write the following 
\begin{align*}
&\lim\limits_{k \to \infty}\widetilde\Delta_{k}=\Gamma\left(\gamma_{2} \right)\left(\xi-T_{1} \right)^{\beta _{2}} \lim\limits_{k \to \infty}E_{2} \left(\left. \begin{array}{l} {\gamma_{2} ,\gamma_{2} ,1;1,0} \\ {\beta _{2} +1,\beta _{2} ,\alpha_{2} ;\gamma_{2} ,\gamma_{2} ;1,1} \end{array}\right|\begin{array}{c} {\lambda_{k} \left(\xi-T_{1} \right)^{\beta _{2} } } \\ {\delta \left(\xi-T_{1} \right)^{\alpha_{2} } } \end{array}\right)+ \displaybreak\\ 
&+T_{1}^{\beta_{1}} \Gamma(\gamma_{1})\lim\limits_{k \to \infty}E_{2} \left(\left. \begin{array}{l} {\gamma_{1} ,\gamma_{1} ,1;1,0} \\ {\beta _{1} +1,\beta _{1} ,\alpha_{1} ;\gamma_{1} ,\gamma_{1} ;1,1} \end{array}\right|\begin{array}{c} {\lambda_k T_{1}^{\beta _{1} } } \\ {\delta T_{1}^{\alpha } } \end{array}\right)+T_{1}+T_{1}  \Gamma(\gamma_{1})\times\\ &\times\lim\limits_{k \to \infty}\left[ \lambda_{k}T_{1}^{\beta_{1}} \right]E_{2} \left(\left. \begin{array}{l} {\gamma_{1} ,\gamma_{1} ,1;1,0} \\ {\beta _{1} +2,\beta _{1} ,\alpha_{1} ;\gamma_{1} ,\gamma_{1} ;1,1} \end{array}\right|\begin{array}{c} {\lambda_k T_{1}^{\beta _{1} } } \\ {\delta T_{1}^{\alpha_{1} } } \end{array}\right)+2 T_{1}^{\beta_{1}}\Gamma\left(\gamma_{1} \right)\Gamma\left(\gamma_{2} \right)\times  \\
&\times \lim\limits_{k \to \infty}\left[ \lambda_{k}\left(\xi-T_{1} \right)^{\beta _{2}}
 \right]E_{2} \left(\left. \begin{array}{l} {\gamma_{2} ,\gamma_{2} ,1;1,0} \\ {\beta _{2} +1,\beta _{2} ,\alpha_{2} ;\gamma_{2} ,\gamma_{2} ;1,1} \end{array}\right|\begin{array}{c} {\lambda_{k} \left(\xi-T_{1} \right)^{\beta _{2} } } \\ {\delta \left(\xi-T_{1} \right)^{\alpha_{2} } } \end{array}\right)\times\\ &\times\lim\limits_{k \to \infty}E_{2} \left(\left. \begin{array}{l} {\gamma_{1} ,\gamma_{1} ,1;1,0} \\ {\beta _{1} +1,\beta _{1} ,\alpha_{1} ;\gamma_{1} ,\gamma_{1} ;1,1} \end{array}\right|\begin{array}{c} {\lambda_k T_{1}^{\beta _{1} } } \\ {\delta T_{1}^{\alpha_{1} } } \end{array}\right) +\Gamma^{2}\left(\gamma_{1} \right)T_{1}^{\beta_{1}}\times\\ &\times
\lim\limits_{k \to \infty}\left[ \lambda_{k} T_{1}^{\beta_{1}}\right]E_{2} \left(\left. \begin{array}{l} {\gamma_{1} ,\gamma_{1} ,1;1,0} \\ {\beta _{1} +1,\beta _{1} ,\alpha_{1} ;\gamma_{1} ,\gamma_{1} ;1,1} \end{array}\right|\begin{array}{c} {\lambda_k T_{1}^{\beta _{1} } } \\ {\delta T_{1}^{\alpha_{1} } } \end{array}\right)\lim\limits_{k \to \infty}E_{2} \left(\left. \begin{array}{l} {\gamma_{1} ,\gamma_{1} ,1;1,0} \\ {\beta _{1} +1,\beta _{1} ,\alpha_{1} ;\gamma_{1} ,\gamma_{1} ;1,1} \end{array}\right|\begin{array}{c} {\lambda_k T_{1}^{\beta _{1} } } \\ {\delta T_{1}^{\alpha_{1} } } \end{array}\right)+\\&
+\Gamma^{2}\left(\gamma_{1} \right)\Gamma\left(\gamma_{2} \right)T_{1}^{\beta_{1}} \lim\limits_{k \to \infty}\left[ \lambda_{k}\left(\xi-T_{1} \right)^{\beta _{2}}
 \right]E_{2} \left(\left. \begin{array}{l} {\gamma_{2} ,\gamma_{2} ,1;1,0} \\ {\beta _{2} +1,\beta _{2} ,\alpha_{2} ;\gamma_{2} ,\gamma_{2} ;1,1} \end{array}\right|\begin{array}{c} {\lambda_{k} \left(\xi-T_{1} \right)^{\beta _{2} } } \\ {\delta \left(\xi-T_{1} \right)^{\alpha_{2} } } \end{array}\right)\times\\
&\times \lim\limits_{k \to \infty}\left[ \lambda_{k} T_{1}^{\beta_{1}}\right]E_{2} \left(\left. \begin{array}{l} {\gamma_{1} ,\gamma_{1} ,1;1,0} \\ {\beta _{1} +1,\beta _{1} ,\alpha_{1} ;\gamma_{1} ,\gamma_{1} ;1,1} \end{array}\right|\begin{array}{c} {\lambda_k T_{1}^{\beta _{1} } } \\ {\delta T_{1}^{\alpha_{1} } } \end{array}\right)\lim\limits_{k \to \infty}E_{2} \left(\left. \begin{array}{l} {\gamma_{1} ,\gamma_{1} ,1;1,0} \\ {\beta _{1} +1,\beta _{1} ,\alpha_{1} ;\gamma_{1} ,\gamma_{1} ;1,1} \end{array}\right|\begin{array}{c} {\lambda_k T_{1}^{\beta _{1} } } \\ {\delta T_{1}^{\alpha_{1} } } \end{array}\right)+\\&
+2 \Gamma(\gamma_{2})T_{1}\lim\limits_{k \to \infty}\left[ \lambda_{k}\left(\xi-T_{1} \right)^{\beta _{2}} \right] E_{2} \left(\left. \begin{array}{l} {\gamma_{2} ,\gamma_{2} ,1;1,0} \\ {\beta _{2} +1,\beta _{2} ,\alpha_{2} ;\gamma_{2} ,\gamma_{2} ;1,1} \end{array}\right|\begin{array}{c} {\lambda_{k} \left(\xi-T_{1} \right)^{\beta _{2} } } \\ {\delta \left(\xi-T_{1} \right)^{\alpha_{2} } } \end{array}\right)+ \\&+
2 \Gamma(\gamma_{1})\Gamma(\gamma_{2})T_{1} \lim\limits_{k \to \infty}\left[ \lambda_{k}\left(\xi-T_{1} \right)^{\beta _{2}} \right] E_{2} \left(\left. \begin{array}{l} {\gamma_{2} ,\gamma_{2} ,1;1,0} \\ {\beta _{2} +1,\beta _{2} ,\alpha_{2} ;\gamma_{2} ,\gamma_{2} ;1,1} \end{array}\right|\begin{array}{c} {\lambda_{k} \left(\xi-T_{1} \right)^{\beta _{2} } } \\ {\delta \left(\xi-T_{1} \right)^{\alpha_{2} } } \end{array}\right) \times \\&
\lim\limits_{k \to \infty }\left[\lambda_{k}T_{1}^{\beta_{1}}\right]E_{2} \left(\left. \begin{array}{l} {\gamma_{1} ,\gamma_{1} ,1;1,0} \\ {\beta _{1} +2,\beta _{1} ,\alpha_{1} ;\gamma_{1} ,\gamma_{1} ;1,1} \end{array}\right|\begin{array}{c} {\lambda_k T_{1}^{\beta _{1} } } \\ {\delta T_{1}^{\alpha_{1} } } \end{array}\right)-\Gamma\left(\gamma_{1} \right)T_{1}^{\beta_{1}} \times \\&
\times \lim\limits_{k \to \infty}E_{2} \left(\left. \begin{array}{l} {\gamma_{1} ,\gamma_{1} ,1;1,0} \\ {\beta _{1},\beta _{1} ,\alpha_{1} ;\gamma_{1} ,\gamma_{1} ;1,1} \end{array}\right|\begin{array}{c} {\lambda_k T_{1}^{\beta _{1} } } \\ {\delta T_{1}^{\alpha } } \end{array}\right)-\Gamma^{2}(\gamma_{1})T_{1}^{\beta_{1}}\lim\limits_{k \to \infty}E_{2} \left(\left. \begin{array}{l} {\gamma_{1} ,\gamma_{1} ,1;1,0} \\ {\beta _{1},\beta _{1} ,\alpha_{1};\gamma_{1} ,\gamma_{1} ;1,1} \end{array}\right|\begin{array}{c} {\lambda_k T_{1}^{\beta _{1} } } \\ {\delta T_{1}^{\alpha_{1} } } \end{array}\right)\times\\&
\times \lim\limits_{k \to \infty}\left[  \lambda_k T_{1}^{\beta _{1} }\right]E_{2} \left(\left. \begin{array}{l} {\gamma_{1} ,\gamma_{1} ,1;1,0} \\ {\beta _{1}+2,\beta _{1} ,\alpha_{1} ;\gamma_{1} ,\gamma_{1} ;1,1} \end{array}\right|\begin{array}{c} {\lambda_k T_{1}^{\beta _{1} } } \\ {\delta T_{1}^{\alpha_{1} } } \end{array}\right)-\Gamma(\gamma_{1})\Gamma(\gamma_{2})T_{1}^{\beta_{1}} \times \\
&\times \lim\limits_{k \to \infty}\left[ \lambda_{k}\left(\xi-T_{1} \right)^{\beta _{2}} \right] E_{2} \left(\left. \begin{array}{l} {\gamma_{2} ,\gamma_{2} ,1;1,0} \\ {\beta _{2} +1,\beta _{2} ,\alpha_{2} ;\gamma_{2} ,\gamma_{2} ;1,1} \end{array}\right|\begin{array}{c} {\lambda_{k} \left(\xi-T_{1} \right)^{\beta _{2} } } \\ {\delta \left(\xi-T_{1} \right)^{\alpha_{2} } } \end{array}\right)\times\\
 &\times\lim\limits_{k \to \infty}E_{2} \left(\left. \begin{array}{l} {\gamma_{1} ,\gamma_{1} ,1;1,0} \\ {\beta _{1},\beta _{1} ,\alpha_{1} ;\gamma_{1} ,\gamma_{1} ;1,1} \end{array}\right|\begin{array}{c} {\lambda_k T_{1}^{\beta _{1} } } \\ {\delta T_{1}^{\alpha_{1} } } \end{array}\right)-\lim\limits_{k \to \infty}E_{2} \left(\left. \begin{array}{l} {\gamma_{1} ,\gamma_{1} ,1;1,0} \\ {\beta _{1}+2,\beta _{1} ,\alpha_{1} ;\gamma_{1} ,\gamma_{1} ;1,1} \end{array}\right|\begin{array}{c} {\lambda_k T_{1}^{\beta _{1} } } \\ {\delta T_{1}^{\alpha_{1} } } \end{array}\right)\times \\
&\times\Gamma^{2}\left(\gamma_{1} \right)\Gamma\left(\gamma_{2} \right)T_{1}^{\beta _{1} }\lim\limits_{k \to \infty}\left[ \lambda_{k}\left(\xi-T_{1} \right)^{\beta _{2}} \right] E_{2} \left(\left. \begin{array}{l} {\gamma_{2} ,\gamma_{2} ,1;1,0} \\ {\beta _{2} +1,\beta _{2} ,\alpha_{2} ;\gamma_{2} ,\gamma_{2} ;1,1} \end{array}\right|\begin{array}{c} {\lambda_{k} \left(\xi-T_{1} \right)^{\beta _{2} } } \\ {\delta \left(\xi-T_{1} \right)^{\alpha_{2} } } \end{array}\right)\times\\ &\times
\lim\limits_{k \to \infty }\left[\lambda_{k}T_{1}^{\beta_{1}}\right]E_{2} \left(\left. \begin{array}{l} {\gamma_{1} ,\gamma_{1} ,1;1,0} \\ {\beta _{1},\beta _{1} ,\alpha_{1} ;\gamma_{1} ,\gamma_{1} ;1,1} \end{array}\right|\begin{array}{c} {\lambda_k T_{1}^{\beta _{1} } } \\ {\delta T_{1}^{\alpha_{1} } } \end{array}\right).
\end{align*}

Now, we will use the following integral representations for $E_{2}$:
\begin{align*}
&{E_2}\left( {\left. \begin{array}{l}
			\gamma_{2} ,\gamma_{2} ,1;1,0\\
			{\beta _2} + 1,{\beta _2},\alpha_{2} ;\gamma_{2} ,\gamma_{2} ;1,1
		\end{array} \right|\begin{array}{*{20}{c}} \lambda_{k}
			{{\left( \xi-T_{1} \right)^{{\beta _2}}}}\\
			{\delta {\left( \xi-T_{1} \right)^{\alpha_{2} }}}
	\end{array}} \right) =\frac{1}{{\Gamma \left( \gamma_{2}  \right)}}\int\limits_0^1 {\int\limits_0^\infty  {{\zeta ^{\frac{{{\beta _2+1}}}{2}-1}}{\eta ^{\gamma_2  - 1}}} } {(1 - \zeta )^{\frac{{{\beta _2}+1}}{2}-1}}\times\\ &\times
	{e^{ - \eta }}\,e_{{\beta_2}, - \gamma_{2} }^{\frac{{{\beta _2+1}}}{2},\gamma_{2} }\left( { \lambda_{k} {{\left( \xi-T_{1} \right)^{{\beta _2}}}}{\eta ^{\gamma_{2}} }{\xi ^{{\beta _2}}}} \right)e_{\alpha_{2} , - 1}^{\frac{{{\beta _2+1}}}{2} ,1}\left( {\delta {\left( \xi-T_{1} \right)^{\alpha_{2}} }{{(1 - \zeta )}^{\alpha_{2} }}\eta } \right)d\zeta d\eta,\\
	&{E_2}\left( {\left. \begin{array}{l}
			\gamma_{1} ,\gamma_{1} ,1;1,0\\
			{\beta _1} ,{\beta _1},\alpha_{1} ;\gamma_{1} ,\gamma_{1} ;1,1
		\end{array} \right|\begin{array}{*{20}{c}}
			\lambda_{k}{{T_{1}^{{\beta _1}}}}\\
			{\delta {T_{1}^{\alpha_{1}} }}
	\end{array}} \right) = \dfrac{1}{{\Gamma (\gamma_{1} )}}\int\limits_0^1 {\int\limits_0^\infty  {{\zeta ^{\frac{{{\beta _1} }}{2} - 1}}\times}} \\ & \times
	{\eta ^{\gamma_{1}  - 1}}{{(1 - \zeta )}^{\frac{{{\beta _1}}}{2} - 1}}{e^{ - \eta }}\,e_{{\beta _1}, - \gamma_{1} }^{\frac{{{\beta _1} }}{2},\gamma_{1} }\left( {{\lambda_{k}{{T_{1}^{{\beta _1}}}}\zeta ^{{\beta _1}}}{\eta^ {\gamma_{1}} }} \right)  \,e_{\alpha_{1} , - 1}^{\frac{{{\beta _1} }}{2},1}\left( {\delta {T_{1}^{\alpha_{2}} }{{(1 - \zeta )}^{\alpha_{2} }}\eta } \right)d\zeta d\eta,\\
&{E_2}\left( {\left. \begin{array}{l}
			\gamma_{1} ,\gamma_{1} ,1;1,0\\
			{\beta _1} + 1,{\beta _1},\alpha_{1} ;\gamma_{1} ,\gamma_{1} ;1,1
		\end{array} \right|\begin{array}{*{20}{c}}
			\lambda_{k}{T_{1}^{{\beta _1}}}\\
			{\delta {T_{1}^{\alpha_{1}} }}
	\end{array}} \right) = \dfrac{1}{{\Gamma (\gamma_{1} )}}\int\limits_0^1 {\int\limits_0^\infty  {{\zeta ^{\frac{{{\beta _1} + 1}}{2} - 1}}\times}} \\
& \times {\eta ^{\gamma_{1}  - 1}}{{(1 - \zeta )}^{\frac{{{\beta _1} + 1}}{2} - 1}}{e^{ - \eta }}\,e_{{\beta _1}, - \gamma_{1} }^{\frac{{{\beta _1} + 1}}{2},\gamma_{1} }\left( { \lambda_{k}{T_{1}^{\beta_{1}}}{\zeta ^{{\beta _1}}}{\eta ^{\gamma_{1} }}} \right)  \,e_{\alpha_{1} , - 1}^{\frac{{{\beta _1} + 1}}{2},1}\left( {\delta {T_{1}^{\alpha_{1}} }{{(1 - \zeta )}^{\alpha_{1}} }\eta } \right)d\zeta d\eta,\\
&{E_2}\left( {\left. \begin{array}{l}
			\gamma_{1} ,\gamma_{1} ,1;1,0\\
			{\beta _1} + 2,{\beta _1},\alpha_{1} ;\gamma_{1} ,\gamma_{1} ;1,1
		\end{array} \right|\begin{array}{*{20}{c}}
			{\lambda_{k}{T_{1}^{{\beta_1}}}}\\
			{\delta {T_{1}^{{\alpha }_{1}}}}
	\end{array}} \right) =\frac{1}{{\Gamma \left( \gamma_{1}  \right)}}\int\limits_0^1 {\int\limits_0^\infty  {{\zeta ^{\frac{{{\beta _1}}}{2}}}{\eta ^{\gamma_{1}  - 1}}} } {(1 - \zeta )^{\frac{{{\beta _1}}}{2}}}\times\\
	&\times {e^{ - \eta }}\,e_{{\beta _1}, - \gamma_{1} }^{\frac{{{\beta _1}}}{2} + 1,\gamma_{1} }\left( { \lambda_{k} T_{1}^{\beta_{1}}{\eta ^{\gamma_{1}} }{\zeta ^{{\beta _1}}}} \right)e_{\alpha_{1} , - 1}^{\frac{{{\beta _1}}}{2} + 1,1}\left( {\delta {T_{1}^{\alpha_{1}} }{{(1 - \zeta )}^{\alpha_{1}} }\eta } \right)d\zeta d\eta.
\end{align*}

    Then we will use the following asymptotic formulas \cite{13}:
\begin{equation*}
	\mathop {\lim }\limits_{\left| z \right| \to \infty } e_{\alpha ,\beta }^{\mu ,\delta }(z) = 0, \,\,\,\,\mathop {\lim }\limits_{\left| z \right| \to \infty } ze_{\alpha ,\beta }^{\mu ,\delta }(z) =  - \frac{1}{{\Gamma \left( {\mu  - \alpha } \right)\Gamma \left( {\delta  + \beta } \right)}}
\end{equation*}
in order to get
$$\lim\limits_{k \to \infty}\widetilde\Delta_{k}=T_{1}>0.$$
\end{proof}

We have demonstrated that the limit of $\widetilde\Delta_{k}$ as $k$ approaches infinity is non-zero and positive. Now we will determine the unknown coefficients ${}_{2}A_{k},{}_{3}A_{k}$ and the unknown function $f_{k}$. To accomplish this, we first need to find ${}_{2}\Delta_{k},{}_{3}\Delta_{k},{}_{f}\Delta_{k}$.
\begin{equation} \label{399}
    {}_{2}A_{k}=\frac{{}_{2}\Delta_{k}}{\widetilde\Delta_{k}}, {}_{3}A_{k}=\frac{{}_{3}\Delta_{k}}{\widetilde\Delta_{k}},
     f_{k}=\frac{{}_{f}\Delta_{k}}{\widetilde\Delta_{k}}.
\end{equation}
To find ${}_{2}\Delta_{k}$, we construct the following determinant
\begin{align*}
    &       {}_{2} {\Delta}_{k}=\begin{vmatrix}  \psi_{k} & 1+\lambda_{k}\cdot {}_{1}M_{k} & {}_{1}M_{k}\\ -\varphi_{k}\left( 1+\lambda_{k}\cdot {}_{2}M_{k} \right) & -1 & {}_{2}M_{k} \\ -\lambda_{k} \cdot \varphi_{k}\cdot {}_{4}M_{k} & \lambda_{k} & {}_{4}M_{k}-1 \end{vmatrix}=\psi_{k} \left( 1-{}_{4}M_{k} \right)-\lambda_{k} \cdot \varphi_{k}\cdot {}_{4}M_{k}\cdot{}_{2}M_{k} \times \\ \\
& \left( 1+\lambda_{k}\cdot {}_{1}M_{k} \right)-\varphi_{k}\left( 1+\lambda_{k}\cdot {}_{2}M_{k} \right)\lambda_{k}\cdot{}_{1}M_{k}-\Big(\lambda_{k} \cdot \varphi_{k}\cdot {}_{1}M_{k}\cdot{}_{4}M_{k}-\varphi_{k}\left( 1+\lambda_{k}\cdot {}_{2}M_{k} \right) \times \\ \\
&\times \left( 1+\lambda_{k}\cdot {}_{1}M_{k} \right)\left( {}_{4}M_{k}-1 \right)+\lambda_{k}\cdot \psi_{k}\cdot{}_{2}M_{k}\Big)=\psi_{k}-\varphi_{k}-{}_{4}M_{k}\left( \psi_{k}-\varphi_{k} \right)-\lambda_{k}\cdot \varphi_{k}\left( {}_{1}M_{k}+{}_{2}M_{k} \right)-\\ \\
& -2\lambda^{2}_{k} \cdot \varphi_{k} \cdot {}_{1}M_{k}\cdot {}_{2}M_{k}-\lambda_{k} \cdot \varphi_{k}\cdot {}_{1}M_{k}-\lambda_{k} \cdot \psi_{k}\cdot {}_{2}M_{k}.
\end{align*}

Now, substituting the designations \eqref{28}, \eqref{31}, \eqref{32},\eqref{34} into the last equality and simplifying, we find ${}_{2} {\Delta}_{k}$ as follows

\begin{equation*}
    \begin{split}
&{}_{2}\Delta_{k}=\psi_{k}-\psi_{k}T_{1}^{\beta_{1}-1} \Gamma(\gamma_{1}) E_{2} \left(\left. \begin{array}{l} {\gamma_{1} ,\gamma_{1} ,1;1,0} \\ {\beta _{1},\beta _{1} ,\alpha_{1} ;\gamma_{1} ,\gamma_{1} ;1,1} \end{array}\right|\begin{array}{c} {\lambda_k T_{1}^{\beta _{1} } } \\ {\delta T_{1}^{\alpha_{1} } } \end{array}\right)+E_{2} \left(\left. \begin{array}{l} {\gamma_{1} ,\gamma_{1} ,1;1,0} \\ {\beta _{1},\beta _{1} ,\alpha ;\gamma_{1} ,\gamma_{1} ;1,1} \end{array}\right|\begin{array}{c} {\lambda_k T_{1}^{\beta _{1} } } \\ {\delta T_{1}^{\alpha_{1} } } \end{array}\right)\times\\
&\times \varphi_{k}T_{1}^{\beta_{1}-1} \Gamma(\gamma_{1}) -\varphi_{k}-2\lambda_{k} \varphi_{k}\Gamma\left(\gamma_{2} \right)\left(\xi-T_{1} \right)^{\beta _{2}} E_{2} \left(\left. \begin{array}{l} {\gamma_{2} ,\gamma_{2} ,1;1,0} \\ {\beta _{2} +1,\beta _{2} ,\alpha_{2} ;\gamma_{2} ,\gamma_{2} ;1,1} \end{array}\right|\begin{array}{c} {\lambda_{k} \left(\xi-T_{1} \right)^{\beta _{2} } } \\ {\delta \left(\xi-T_{1} \right)^{\alpha_{2} } } \end{array}\right)-\\
&-\lambda_{k} \varphi_{k}T_{1}^{\beta_{1}} \Gamma(\gamma_{1}) E_{2} \left(\left. \begin{array}{l} {\gamma_{1} ,\gamma_{1} ,1;1,0} \\ {\beta _{1}+1,\beta _{1} ,\alpha_{1} ;\gamma_{1} ,\gamma_{1} ;1,1} \end{array}\right|\begin{array}{c} {\lambda_k T_{1}^{\beta _{1} } } \\ {\delta T_{1}^{\alpha_{1} } } \end{array}\right)-2\lambda^{2}_{k} \varphi_{k}\Gamma\left(\gamma_{1} \right)\Gamma\left(\gamma_{2} \right)\left(\xi-T_{1} \right)^{\beta _{2}}T_{1}^{\beta _{1} }\times\\
&\times E_{2} \left(\left. \begin{array}{l} {\gamma_{2} ,\gamma_{2} ,1;1,0} \\ {\beta _{2} +1,\beta _{2} ,\alpha_{2} ;\gamma_{2} ,\gamma_{2} ;1,1} \end{array}\right|\begin{array}{c} {\lambda_{k} \left(\xi-T_{1} \right)^{\beta _{2} } } \\ {\delta \left(\xi-T_{1} \right)^{\alpha_{2} } } \end{array}\right)E_{2} \left(\left. \begin{array}{l} {\gamma_{1} ,\gamma_{1} ,1;1,0} \\ {\beta _{1}+1,\beta _{1} ,\alpha_{1} ;\gamma_{1} ,\gamma_{1} ;1,1} \end{array}\right|\begin{array}{c} {\lambda_k T_{1}^{\beta _{1} } } \\ {\delta T_{1}^{\alpha_{1} } } \end{array}\right)-\\
&-\lambda_{k}\psi_{k}T_{1}^{\beta _{1} }\Gamma(\gamma_{1})E_{2} \left(\left. \begin{array}{l} {\gamma_{1} ,\gamma_{1} ,1;1,0} \\ {\beta _{1}+1,\beta _{1} ,\alpha_{1} ;\gamma_{1} ,\gamma_{1} ;1,1} \end{array}\right|\begin{array}{c} {\lambda_k T_{1}^{\beta _{1} } } \\ {\delta T_{1}^{\alpha_{1} } } \end{array}\right).   
    \end{split}
\end{equation*}
Now using formula \eqref{399},we find the coefficient ${}_{2}A_{k}$
\begin{align} \label{40}
    &{}_{2}A_{k}=\frac{{}_{2}\Delta_{k}}{\widetilde{\Delta}_{k}}=\frac{1}{\widetilde{\Delta}_{k}}\Bigg[\psi_{k}-\psi_{k}T_{1}^{\beta_{1}-1} \Gamma(\gamma_{1}) E_{2} \left(\left. \begin{array}{l} {\gamma_{1} ,\gamma_{1} ,1;1,0} \\ {\beta _{1},\beta _{1} ,\alpha_{1} ;\gamma_{1} ,\gamma_{1} ;1,1} \end{array}\right|\begin{array}{c} {\lambda_k T_{1}^{\beta _{1} } } \\ {\delta T_{1}^{\alpha_{1} } } \end{array}\right)+\\
& +\varphi_{k}T_{1}^{\beta_{1}-1} \Gamma(\gamma_{1})E_{2} \left(\left. \begin{array}{l} {\gamma_{1} ,\gamma_{1} ,1;1,0} \\ {\beta _{1},\beta _{1} ,\alpha_{1} ;\gamma_{1} ,\gamma_{1} ;1,1} \end{array}\right|\begin{array}{c} {\lambda_k T_{1}^{\beta _{1} } } \\ {\delta T_{1}^{\alpha_{1} } } \end{array}\right)-\varphi_{k}-2\lambda_{k} \varphi_{k}\Gamma\left(\gamma_{2} \right)\left(\xi-T_{1} \right)^{\beta _{2}} \times \nonumber\\
& \times E_{2} \left(\left. \begin{array}{l} {\gamma_{2} ,\gamma_{2} ,1;1,0} \\ {\beta _{2} +1,\beta _{2} ,\alpha_{2} ;\gamma_{2} ,\gamma_{2} ;1,1} \end{array}\right|\begin{array}{c} {\lambda_{k} \left(\xi-T_{1} \right)^{\beta _{2} } } \\ {\delta \left(\xi-T_{1} \right)^{\alpha_{2} } } \end{array}\right)- \Gamma(\gamma_{1}) E_{2} \left(\left. \begin{array}{l} {\gamma_{1} ,\gamma_{1} ,1;1,0} \\ {\beta _{1}+1,\beta _{1} ,\alpha_{1} ;\gamma_{1} ,\gamma_{1} ;1,1} \end{array}\right|\begin{array}{c} {\lambda_k T_{1}^{\beta _{1} } } \\ {\delta T_{1}^{\alpha_{1} } } \end{array}\right)\times\nonumber\\
&\times\lambda_{k} \varphi_{k}T_{1}^{\beta_{1}}-2\lambda^{2}_{k} \varphi_{k}\Gamma\left(\gamma_{1} \right)\Gamma\left(\gamma_{2} \right)\left(\xi-T_{1} \right)^{\beta _{2}}T_{1}^{\beta _{1} } E_{2} \left(\left. \begin{array}{l} {\gamma_{2} ,\gamma_{2} ,1;1,0} \\ {\beta _{2} +1,\beta _{2} ,\alpha_{2} ;\gamma_{2} ,\gamma_{2} ;1,1} \end{array}\right|\begin{array}{c} {\lambda_{k} \left(\xi-T_{1} \right)^{\beta _{2} } } \\ {\delta \left(\xi-T_{1} \right)^{\alpha_{2} } } \end{array}\right)\times\nonumber\\
&\times E_{2} \left(\left. \begin{array}{l} {\gamma_{1} ,\gamma_{1} ,1;1,0} \\ {\beta _{1}+1,\beta _{1} ,\alpha_{1};\gamma_{1} ,\gamma_{1} ;1,1} \end{array}\right|\begin{array}{c} {\lambda_k T_{1}^{\beta _{1} } } \\ {\delta T_{1}^{\alpha_{1} } } \end{array}\right)-\lambda_{k}\psi_{k}T_{1}^{\beta _{1} }\Gamma(\gamma_{1})E_{2} \left(\left. \begin{array}{l} {\gamma_{1} ,\gamma_{1} ,1;1,0} \\ {\beta _{1}+1,\beta _{1} ,\alpha_{1} ;\gamma_{1} ,\gamma_{1} ;1,1} \end{array}\right|\begin{array}{c} {\lambda_k T_{1}^{\beta _{1} } } \\ {\delta T_{1}^{\alpha_{1} } } \end{array}\right)\Bigg].\nonumber
\end{align}

In the same manner, we determine the coefficient ${}_{3}A_{k}$

\begin{equation*}
\begin{split}
&{}_{3} {\Delta}_{k}=\begin{vmatrix}  0 & \psi_{k}& {}_{1}M_{k}\\ {}_{3}M_{k} & -\varphi_{k}\left( 1+\lambda_{k}\cdot {}_{2}M_{k} \right) & {}_{2}M_{k} \\  1+\lambda_{k}\cdot {}_{2}M_{k} & -\lambda_{k} \cdot \varphi_{k}\cdot {}_{4}M_{k}  & {}_{4}M_{k}-1 \end{vmatrix}=0+\psi_{k} \cdot {}_{2}M_{k}\left( 1+{}_{4}M_{k} \right)-\lambda_{k} \cdot \varphi_{k}\cdot {}_{4}M_{k} \times \\ \\
& \times{}_{1}M_{k}\cdot {}_{3}M_{k}-\Big(-\varphi_{k}\left( 1+\lambda_{k}\cdot {}_{2}M_{k} \right){}_{1}M_{k}\left( 1+\lambda_{k}\cdot {}_{2}M_{k} \right)+\psi_{k} \cdot {}_{3}M_{k}\left( {}_{4}M_{k}-1 \right)+0\Big)=\psi_{k} \cdot {}_{2}M_{k}+\\
\\
&+\lambda_{k}\cdot \psi_{k}\cdot {}_{2}M^{2}_{k}-\lambda_{k}\cdot \varphi_{k}\cdot{}_{1}M_{k}\cdot {}_{3}M_{k}\cdot {}_{4}M_{k}+\varphi_{k} \cdot {}_{1}M_{k}+2\lambda_{k}\cdot \varphi_{k}\cdot{}_{1}M_{k}\cdot {}_{2}M_{k}+\varphi_{k} \cdot\lambda_{k}^{2}\cdot{}_{1}M_{k}\times\\ \\
& \times{}_{2}M_{k}^{2}-\psi_{k}\cdot {}_{3}M_{k}\cdot {}_{4}M_{k}+\psi_{k}\cdot {}_{3}M_{k}.
\end{split}
\end{equation*}
Substituting the above notations \eqref{28}, \eqref{31},\eqref{32} and \eqref{34} into the expression obtained in ${}_{3} {\Delta}_{k}$ as 
\begin{align*}
 &{}_{3}{\Delta}_{k}=\psi_{k}T_{1}^{\beta_{1}} \Gamma(\gamma_{1}) E_{2} \left(\left. \begin{array}{l} {\gamma_{1} ,\gamma_{1} ,1;1,0} \\ {\beta _{1} +1,\beta _{1} ,\alpha_{1} ;\gamma_{1} ,\gamma_{1} ;1,1} \end{array}\right|\begin{array}{c} {\lambda_k T_{1}^{\beta _{1} } } \\ {\delta T_{1}^{\alpha_{1} } } \end{array}\right)+ \lambda_{k} \psi_{k}T_{1}^{2\beta_{1}} \Gamma^{2}(\gamma_{1}) \times\\
&\times\left( E_{2} \left(\left. \begin{array}{l} {\gamma_{1} ,\gamma_{1} ,1;1,0} \\ {\beta _{1} +1,\beta _{1} ,\alpha_{1} ;\gamma_{1} ,\gamma_{1} ;1,1} \end{array}\right|\begin{array}{c} {\lambda_k T_{1}^{\beta _{1} } } \\ {\delta T_{1}^{\alpha_{1} } } \end{array}\right) \right)^{2}-\lambda_{k}\varphi_{k}\Gamma\left(\gamma_{1} \right)\Gamma\left(\gamma_{2} \right)T_{1}^{\beta_{1}}\left(\xi-T_{1} \right)^{\beta _{2}} \times \\
&\times E_{2} \left(\left. \begin{array}{l} {\gamma_{1} ,\gamma_{1} ,1;1,0} \\ {\beta _{1},\beta _{1} ,\alpha_{1} ;\gamma_{1} ,\gamma_{1} ;1,1} \end{array}\right|\begin{array}{c} {\lambda_k T_{1}^{\beta _{1} } } \\ {\delta T_{1}^{\alpha_{1}} } \end{array}\right)E_{2} \left(\left. \begin{array}{l} {\gamma_{2} ,\gamma_{2} ,1;1,0} \\ {\beta _{2} +1,\beta _{2} ,\alpha_{2} ;\gamma_{2} ,\gamma_{2} ;1,1} \end{array}\right|\begin{array}{c} {\lambda_{k} \left(\xi-T_{1} \right)^{\beta _{2} } } \\ {\delta \left(\xi-T_{1} \right)^{\alpha_{2} } } \end{array}\right)-\lambda_{k}^{2}T_{1}^{2\beta_{1}}\Gamma\left(\gamma_{2} \right)\times\\
& \times\varphi_{k} \Gamma^{2}\left(\gamma_{1} \right)\left(\xi-T_{1} \right)^{\beta _{2}}E_{2} \left(\left. \begin{array}{l} {\gamma_{1} ,\gamma_{1} ,1;1,0} \\ {\beta _{1},\beta _{1} ,\alpha_{1} ;\gamma_{1} ,\gamma_{1} ;1,1} \end{array}\right|\begin{array}{c} {\lambda_k T_{1}^{\beta _{1} } } \\ {\delta T_{1}^{\alpha } } \end{array}\right)E_{2} \left(\left. \begin{array}{l} {\gamma_{1} ,\gamma_{1} ,1;1,0} \\ {\beta _{1}+2,\beta _{1} ,\alpha_{1} ;\gamma_{1} ,\gamma_{1} ;1,1} \end{array}\right|\begin{array}{c} {\lambda_k T_{1}^{\beta _{1} } } \\ {\delta T_{1}^{\alpha_{1} } } \end{array}\right)\times\\
& \times E_{2} \left(\left. \begin{array}{l} {\gamma_{2} ,\gamma_{2} ,1;1,0} \\ {\beta _{2} +1,\beta _{2} ,\alpha_{2} ;\gamma_{2} ,\gamma_{2} ;1,1} \end{array}\right|\begin{array}{c} {\lambda_{k} \left(\xi-T_{1} \right)^{\beta _{2} } } \\ {\delta \left(\xi-T_{1} \right)^{\alpha_{2} } } \end{array}\right)+\varphi_{k}\Gamma\left(\gamma_{2} \right)\left(\xi-T_{1} \right)^{\beta _{2}}\times \\
&\times E_{2} \left(\left. \begin{array}{l} {\gamma_{2} ,\gamma_{2} ,1;1,0} \\ {\beta _{2} +1,\beta _{2} ,\alpha_{2} ;\gamma_{2} ,\gamma_{2} ;1,1} \end{array}\right|\begin{array}{c} {\lambda_{k} \left(\xi-T_{1} \right)^{\beta _{2} } } \\ {\delta \left(\xi-T_{1} \right)^{\alpha_{2} } } \end{array}\right)+2\varphi_{k}\lambda_{k}\Gamma\left(\gamma_{1} \right)\Gamma\left(\gamma_{2} \right)T_{1}^{\beta_{1}}\left(\xi-T_{1} \right)^{\beta _{2}}\times \\
& \times E_{2} \left(\left. \begin{array}{l} {\gamma_{2} ,\gamma_{2} ,1;1,0} \\ {\beta _{2} +1,\beta _{2} ,\alpha_{2} ;\gamma_{2} ,\gamma_{2} ;1,1} \end{array}\right|\begin{array}{c} {\lambda_{k} \left(\xi-T_{1} \right)^{\beta _{2} } } \\ {\delta \left(\xi-T_{1} \right)^{\alpha_{2} } } \end{array}\right)E_{2} \left(\left. \begin{array}{l} {\gamma_{1} ,\gamma_{1} ,1;1,0} \\ {\beta _{1} +1,\beta _{1} ,\alpha_{1} ;\gamma_{1} ,\gamma_{1} ;1,1} \end{array}\right|\begin{array}{c} {\lambda_k T_{1}^{\beta _{1} } } \\ {\delta T_{1}^{\alpha_{1} } } \end{array}\right)+\\
&+\varphi_{k}\lambda_{k}^{2}\Gamma^{2}\left(\gamma_{1} \right)\Gamma\left(\gamma_{2} \right)T_{1}^{2\beta_{1}}\left(\xi-T_{1} \right)^{\beta _{2}}E_{2} \left(\left. \begin{array}{l} {\gamma_{2} ,\gamma_{2} ,1;1,0} \\ {\beta _{2} +1,\beta _{2} ,\alpha_{2} ;\gamma_{2} ,\gamma_{2} ;1,1} \end{array}\right|\begin{array}{c} {\lambda_{k} \left(\xi-T_{1} \right)^{\beta _{2} } } \\ {\delta \left(\xi-T_{1} \right)^{\alpha_{2} } } \end{array}\right) \times\\
& \times\left( E_{2} \left(\left. \begin{array}{l} {\gamma_{1} ,\gamma_{1} ,1;1,0} \\ {\beta _{1} +1,\beta _{1} ,\alpha_{1} ;\gamma_{1} ,\gamma_{1} ;1,1} \end{array}\right|\begin{array}{c} {\lambda_k T_{1}^{\beta _{1} } } \\ {\delta T_{1}^{\alpha_{1} } } \end{array}\right) \right)^{2}
- \psi_{k}T_{1}^{\beta_{1}} \Gamma(\gamma_{1}) E_{2} \left(\left. \begin{array}{l} {\gamma_{1} ,\gamma_{1} ,1;1,0} \\ {\beta _{1} ,\beta _{1} ,\alpha_{1} ;\gamma_{1} ,\gamma_{1} ;1,1} \end{array}\right|\begin{array}{c} {\lambda_k T_{1}^{\beta _{1} } } \\ {\delta T_{1}^{\alpha_{1} } } \end{array}\right)-\\
& -\psi_{k}T_{1}^{\beta_{1}} \Gamma(\gamma_{1})E_{2} \left(\left. \begin{array}{l} {\gamma_{1} ,\gamma_{1} ,1;1,0} \\ {\beta _{1},\beta _{1} ,\alpha_{1} ;\gamma_{1} ,\gamma_{1} ;1,1} \end{array}\right|\begin{array}{c} {\lambda_k T_{1}^{\beta _{1} } } \\ {\delta T_{1}^{\alpha_{1} } } \end{array}\right)-\psi_{k} \Gamma^{2}(\gamma_{1})\lambda_{k}E_{2} \left(\left. \begin{array}{l} {\gamma_{1} ,\gamma_{1} ,1;1,0} \\ {\beta _{1},\beta _{1} ,\alpha_{1} ;\gamma_{1} ,\gamma_{1} ;1,1} \end{array}\right|\begin{array}{c} {\lambda_k T_{1}^{\beta _{1} } } \\ {\delta T_{1}^{\alpha_{1} } } \end{array}\right) \times \\
&\times T_{1}^{2\beta_{1}}E_{2} \left(\left. \begin{array}{l} {\gamma_{1} ,\gamma_{1} ,1;1,0} \\ {\beta _{1}+2,\beta _{1} ,\alpha_{1} ;\gamma_{1} ,\gamma_{1} ;1,1} \end{array}\right|\begin{array}{c} {\lambda_k T_{1}^{\beta _{1} } } \\ {\delta T_{1}^{\alpha_{1} } } \end{array}\right)+\psi_{k}T_{1}+\psi_{k} \lambda_{k}T_{1}^{\beta_{1}+1}\Gamma(\gamma_{1})\times\\
& \times E_{2} \left(\left. \begin{array}{l} {\gamma_{1} ,\gamma_{1} ,1;1,0} \\ {\beta _{1}+2,\beta _{1} ,\alpha_{1} ;\gamma_{1} ,\gamma_{1} ;1,1} \end{array}\right|\begin{array}{c} {\lambda_k T_{1}^{\beta _{1} } } \\ {\delta T_{1}^{\alpha_{1} } } \end{array}\right),
\end{align*}
according to formula \eqref{399}, we find ${}_{3}A_{k}$ as
\begin{equation} \label{41}
    \begin{split}
    & {}_{3}A_{k}=\frac{{}_{3}\Delta_{k}}{\widetilde{\Delta}_{k}}=\frac{1}{\widetilde{\Delta}_{k}}\Bigg[\psi_{k}T_{1}^{\beta_{1}} \Gamma(\gamma_{1}) E_{2} \left(\left. \begin{array}{l} {\gamma_{1} ,\gamma_{1} ,1;1,0} \\ {\beta _{1} +1,\beta _{1} ,\alpha_{1} ;\gamma_{1} ,\gamma_{1} ;1,1} \end{array}\right|\begin{array}{c} {\lambda_k T_{1}^{\beta _{1} } } \\ {\delta T_{1}^{\alpha_{1} } } \end{array}\right)+ \lambda_{k} \psi_{k}T_{1}^{2\beta_{1}} \Gamma^{2}(\gamma_{1}) \times\\
&\times\left( E_{2} \left(\left. \begin{array}{l} {\gamma_{1} ,\gamma_{1} ,1;1,0} \\ {\beta _{1} +1,\beta _{1} ,\alpha_{1} ;\gamma_{1} ,\gamma_{1} ;1,1} \end{array}\right|\begin{array}{c} {\lambda_k T_{1}^{\beta _{1} } } \\ {\delta T_{1}^{\alpha_{1} } } \end{array}\right) \right)^{2}-\lambda_{k}\varphi_{k}\Gamma\left(\gamma_{1} \right)\Gamma\left(\gamma_{2} \right)T_{1}^{\beta_{1}}\left(\xi-T_{1} \right)^{\beta _{2}} \times \\
&\times E_{2} \left(\left. \begin{array}{l} {\gamma_{1} ,\gamma_{1} ,1;1,0} \\ {\beta _{1},\beta _{1} ,\alpha_{1} ;\gamma_{1} ,\gamma_{1} ;1,1} \end{array}\right|\begin{array}{c} {\lambda_k T_{1}^{\beta _{1} } } \\ {\delta T_{1}^{\alpha_{1} } } \end{array}\right)E_{2} \left(\left. \begin{array}{l} {\gamma_{2} ,\gamma_{2} ,1;1,0} \\ {\beta _{2} +1,\beta _{2} ,\alpha_{2} ;\gamma_{2} ,\gamma_{2} ;1,1} \end{array}\right|\begin{array}{c} {\lambda_{k} \left(\xi-T_{1} \right)^{\beta _{2} } } \\ {\delta \left(\xi-T_{1} \right)^{\alpha_{2} } } \end{array}\right)-\lambda_{k}^{2}T_{1}^{2\beta_{1}}\Gamma\left(\gamma_{2} \right)\times\\
& \times\varphi_{k} \Gamma^{2}\left(\gamma_{1} \right)\left(\xi-T_{1} \right)^{\beta _{2}}E_{2} \left(\left. \begin{array}{l} {\gamma_{1} ,\gamma_{1} ,1;1,0} \\ {\beta _{1},\beta _{1} ,\alpha_{1} ;\gamma_{1} ,\gamma_{1} ;1,1} \end{array}\right|\begin{array}{c} {\lambda_k T_{1}^{\beta _{1} } } \\ {\delta T_{1}^{\alpha_{1} } } \end{array}\right)E_{2} \left(\left. \begin{array}{l} {\gamma_{1} ,\gamma_{1} ,1;1,0} \\ {\beta _{1}+2,\beta _{1} ,\alpha_{1} ;\gamma_{1} ,\gamma_{1} ;1,1} \end{array}\right|\begin{array}{c} {\lambda_k T_{1}^{\beta _{1} } } \\ {\delta T_{1}^{\alpha_{1} } } \end{array}\right)\times\\
& \times E_{2} \left(\left. \begin{array}{l} {\gamma_{2} ,\gamma_{2} ,1;1,0} \\ {\beta _{2} +1,\beta _{2} ,\alpha_{2} ;\gamma_{2} ,\gamma_{2} ;1,1} \end{array}\right|\begin{array}{c} {\lambda_{k} \left(\xi-T_{1} \right)^{\beta _{2} } } \\ {\delta \left(\xi-T_{1} \right)^{\alpha_{2} } } \end{array}\right)+\varphi_{k}\Gamma\left(\gamma_{2} \right)\left(\xi-T_{1} \right)^{\beta _{2}}\times \\
&\times E_{2} \left(\left. \begin{array}{l} {\gamma_{2} ,\gamma_{2} ,1;1,0} \\ {\beta _{2} +1,\beta _{2} ,\alpha_{2} ;\gamma_{2} ,\gamma_{2} ;1,1} \end{array}\right|\begin{array}{c} {\lambda_{k} \left(\xi-T_{1} \right)^{\beta _{2} } } \\ {\delta \left(\xi-T_{1} \right)^{\alpha_{2} } } \end{array}\right)+2\varphi_{k}\lambda_{k}\Gamma\left(\gamma_{1} \right)\Gamma\left(\gamma_{2} \right)T_{1}^{\beta_{1}}\left(\xi-T_{1} \right)^{\beta _{2}}\times \\
& \times E_{2} \left(\left. \begin{array}{l} {\gamma_{2} ,\gamma_{2} ,1;1,0} \\ {\beta _{2} +1,\beta _{2} ,\alpha_{2} ;\gamma_{2} ,\gamma_{2} ;1,1} \end{array}\right|\begin{array}{c} {\lambda_{k} \left(\xi-T_{1} \right)^{\beta _{2} } } \\ {\delta \left(\xi-T_{1} \right)^{\alpha_{2} } } \end{array}\right)E_{2} \left(\left. \begin{array}{l} {\gamma_{1} ,\gamma_{1} ,1;1,0} \\ {\beta _{1} +1,\beta _{1} ,\alpha_{1} ;\gamma_{1} ,\gamma_{1} ;1,1} \end{array}\right|\begin{array}{c} {\lambda_k T_{1}^{\beta _{1} } } \\ {\delta T_{1}^{\alpha_{1}} } \end{array}\right)+\\
&+\varphi_{k}\lambda_{k}^{2}\Gamma^{2}\left(\gamma_{1} \right)\Gamma\left(\gamma_{2} \right)T_{1}^{2\beta_{1}}\left(\xi-T_{1} \right)^{\beta _{2}}E_{2} \left(\left. \begin{array}{l} {\gamma_{2} ,\gamma_{2} ,1;1,0} \\ {\beta _{2} +1,\beta _{2} ,\alpha_{2} ;\gamma_{2} ,\gamma_{2} ;1,1} \end{array}\right|\begin{array}{c} {\lambda_{k} \left(\xi-T_{1} \right)^{\beta _{2} } } \\ {\delta \left(\xi-T_{1} \right)^{\alpha_{2} } } \end{array}\right) \times\\
& \times\left( E_{2} \left(\left. \begin{array}{l} {\gamma_{1} ,\gamma_{1} ,1;1,0} \\ {\beta _{1} +1,\beta _{1} ,\alpha_{1} ;\gamma_{1} ,\gamma_{1} ;1,1} \end{array}\right|\begin{array}{c} {\lambda_k T_{1}^{\beta _{1} } } \\ {\delta T_{1}^{\alpha_{1} } } \end{array}\right) \right)^{2}
- \psi_{k}T_{1}^{\beta_{1}} \Gamma(\gamma_{1}) E_{2} \left(\left. \begin{array}{l} {\gamma_{1} ,\gamma_{1} ,1;1,0} \\ {\beta _{1} ,\beta _{1} ,\alpha_{1} ;\gamma_{1} ,\gamma_{1} ;1,1} \end{array}\right|\begin{array}{c} {\lambda_k T_{1}^{\beta _{1} } } \\ {\delta T_{1}^{\alpha_{1} } } \end{array}\right)-\\
& -\psi_{k} T_{1}^{2\beta_{1}}\Gamma^{2}(\gamma_{1})\lambda_{k}E_{2} \left(\left. \begin{array}{l} {\gamma_{1} ,\gamma_{1} ,1;1,0} \\ {\beta _{1},\beta _{1} ,\alpha_{1} ;\gamma_{1} ,\gamma_{1} ;1,1} \end{array}\right|\begin{array}{c} {\lambda_k T_{1}^{\beta _{1} } } \\ {\delta T_{1}^{\alpha } } \end{array}\right) E_{2} \left(\left. \begin{array}{l} {\gamma_{1} ,\gamma_{1} ,1;1,0} \\ {\beta _{1}+2,\beta _{1} ,\alpha_{1} ;\gamma_{1} ,\gamma_{1} ;1,1} \end{array}\right|\begin{array}{c} {\lambda_k T_{1}^{\beta _{1} } } \\ {\delta T_{1}^{\alpha_{1} } } \end{array}\right)+ \\
&+\psi_{k}T_{1}+\psi_{k} \lambda_{k}T_{1}^{\beta_{1}+1}\Gamma(\gamma_{1}) E_{2} \left(\left. \begin{array}{l} {\gamma_{1} ,\gamma_{1} ,1;1,0} \\ {\beta _{1}+2,\beta _{1} ,\alpha_{1} ;\gamma_{1} ,\gamma_{1} ;1,1} \end{array}\right|\begin{array}{c} {\lambda_k T_{1}^{\beta _{1} } } \\ {\delta T_{1}^{\alpha_{1} } } \end{array}\right)\Bigg].
\end{split}
\end{equation}
Now, we will find the unknown function $f_{k}$ in a similar way:
\begin{align} \label{42}
& f_{k}=\frac{{}_{f}\Delta_{k}}{\widetilde{\Delta}_{k}}=\frac{1}{\widetilde{\Delta}_{k}}\Bigg[\psi_{k}-\varphi_{k}-2\lambda_{k} \varphi_{k}T_{1}^{\beta_{1}} \Gamma(\gamma_{1}) E_{2} \left(\left. \begin{array}{l} {\gamma_{1} ,\gamma_{1} ,1;1,0} \\ {\beta _{1} +1,\beta _{1} ,\alpha_{1} ;\gamma_{1} ,\gamma_{1} ;1,1} \end{array}\right|\begin{array}{c} {\lambda_k T_{1}^{\beta _{1} } } \\ {\delta T_{1}^{\alpha_{1} } } \end{array}\right)-\lambda_{k}^{2} \varphi_{k}  \times \nonumber\\
&\times T_{1}^{2\beta_{1}}\Gamma^{2}(\gamma_{1})\left( E_{2} \left(\left. \begin{array}{l} {\gamma_{1} ,\gamma_{1} ,1;1,0} \\ {\beta _{1} +1,\beta _{1} ,\alpha_{1} ;\gamma_{1} ,\gamma_{1} ;1,1} \end{array}\right|\begin{array}{c} {\lambda_k T_{1}^{\beta _{1} } } \\ {\delta T_{1}^{\alpha_{1} } } \end{array}\right) \right)^{2}-\lambda_{k}\varphi_{k}\Gamma\left(\gamma_{2} \right)\left(\xi-T_{1} \right)^{\beta _{2}} \times \nonumber\\
&\times E_{2} \left(\left. \begin{array}{l} {\gamma_{2} ,\gamma_{2} ,1;1,0} \\ {\beta _{2} +1,\beta _{2} ,\alpha_{2} ;\gamma_{2} ,\gamma_{2} ;1,1} \end{array}\right|\begin{array}{c} {\lambda_{k} \left(\xi-T_{1} \right)^{\beta _{2} } } \\ {\delta \left(\xi-T_{1} \right)^{\alpha_{2} } }\end{array}\right)-2\lambda_{k}^{2}\varphi_{k}T_{1}^{\beta_{1}}\Gamma\left(\gamma_{1} \right)\Gamma\left(\gamma_{2} \right)\left(\xi-T_{1} \right)^{\beta _{2}}\times\nonumber\\
&\times E_{2} \left(\left. \begin{array}{l} {\gamma_{2} ,\gamma_{2} ,1;1,0} \\ {\beta _{2} +1,\beta _{2} ,\alpha_{2} ;\gamma_{2} ,\gamma_{2} ;1,1} \end{array}\right|\begin{array}{c} {\lambda_{k} \left(\xi-T_{1} \right)^{\beta _{2} } } \\ {\delta \left(\xi-T_{1} \right)^{\alpha_{2} } } \end{array}\right)E_{2} \left(\left. \begin{array}{l} {\gamma_{1} ,\gamma_{1} ,1;1,0} \\ {\beta _{1} +1,\beta _{1} ,\alpha_{1} ;\gamma_{1} ,\gamma_{1} ;1,1} \end{array}\right|\begin{array}{c} {\lambda_k T_{1}^{\beta _{1} } } \\ {\delta T_{1}^{\alpha_{1} } } \end{array}\right)-\\
& -\lambda_{k}^{3}T_{1}^{2\beta_{1}}\varphi_{k}\Gamma^{2}\left(\gamma_{1} \right)\Gamma\left(\gamma_{2} \right)\left(\xi-T_{1} \right)^{\beta _{2}}E_{2} \left(\left. \begin{array}{l} {\gamma_{2} ,\gamma_{2} ,1;1,0} \\ {\beta _{2} +1,\beta _{2} ,\alpha_{2} ;\gamma_{2} ,\gamma_{2} ;1,1} \end{array}\right|\begin{array}{c} {\lambda_{k} \left(\xi-T_{1} \right)^{\beta _{2} } } \\ {\delta \left(\xi-T_{1} \right)^{\alpha_{2} } } \end{array}\right)\times\nonumber\\
& \times \left( E_{2} \left(\left. \begin{array}{l} {\gamma_{1} ,\gamma_{1} ,1;1,0} \\ {\beta _{1} +1,\beta _{1} ,\alpha_{1} ;\gamma_{1} ,\gamma_{1} ;1,1} \end{array}\right|\begin{array}{c} {\lambda_k T_{1}^{\beta _{1} } } \\ {\delta T_{1}^{\alpha } } \end{array}\right) \right)^{2}+\psi_{k}\lambda_k T_{1}^{\beta _{1} }+\psi_{k}\lambda_k^{2}T_{1}^{\beta _{1}+1 }\Gamma\left(\gamma_{1} \right) \times\nonumber\\
& \times E_{2} \left(\left. \begin{array}{l} {\gamma_{1} ,\gamma_{1} ,1;1,0} \\ {\beta _{1} +2,\beta _{1} ,\alpha_{1} ;\gamma_{1} ,\gamma_{1} ;1,1} \end{array}\right|\begin{array}{c} {\lambda_k T_{1}^{\beta _{1} } } \\ {\delta T_{1}^{\alpha_{1} } } \end{array}\right)+\psi_{k}\lambda_kT_{1}^{\beta _{1} }\Gamma\left(\gamma_{1} \right) \times \nonumber\\
&\times E_{2} \left(\left. \begin{array}{l} {\gamma_{1} ,\gamma_{1} ,1;1,0} \\ {\beta _{1} +1,\beta _{1} ,\alpha_{1};\gamma_{1} ,\gamma_{1} ;1,1} \end{array}\right|\begin{array}{c} {\lambda_k T_{1}^{\beta _{1} } } \\ {\delta T_{1}^{\alpha_{1} } } \end{array}\right)+\lambda_{k}\varphi_{k}\Gamma(\gamma_{1})T_{1}^{\beta_{1}}E_{2} \left(\left. \begin{array}{l} {\gamma_{1} ,\gamma_{1} ,1;1,0} \\ {\beta _{1},\beta _{1} ,\alpha_{1} ;\gamma_{1} ,\gamma_{1} ;1,1} \end{array}\right|\begin{array}{c} {\lambda_k T_{1}^{\beta _{1} } } \\ {\delta T_{1}^{\alpha_{1}} } \end{array}\right)+\nonumber\\
&+\lambda_{k}^{2} \varphi_{k}\Gamma^{2}(\gamma_{1})T_{1}^{2\beta_{1}} E_{2} \left(\left. \begin{array}{l} {\gamma_{1} ,\gamma_{1} ,1;1,0} \\ {\beta _{1},\beta _{1} ,\alpha_{1} ;\gamma_{1} ,\gamma_{1} ;1,1} \end{array}\right|\begin{array}{c} {\lambda_k T_{1}^{\beta _{1} } } \\ {\delta T_{1}^{\alpha_{1} } } \end{array}\right)E_{2} \left(\left. \begin{array}{l} {\gamma_{1} ,\gamma_{1} ,1;1,0} \\ {\beta _{1}+2,\beta _{1} ,\alpha_{1} ;\gamma_{1} ,\gamma_{1} ;1,1} \end{array}\right|\begin{array}{c} {\lambda_k T_{1}^{\beta _{1} } } \\ {\delta T_{1}^{\alpha_{1} } } \end{array}\right)+\nonumber\\
&+\lambda_{k}^{2}\varphi_{k}\Gamma(\gamma_{1})\Gamma\left(\gamma_{2} \right)\left(\xi-T_{1} \right)^{\beta _{2}}T_{1}^{\beta_{1}}E_{2} \left(\left. \begin{array}{l} {\gamma_{2} ,\gamma_{2} ,1;1,0} \\ {\beta _{2} +1,\beta _{2} ,\alpha_{2} ;\gamma_{2} ,\gamma_{2} ;1,1} \end{array}\right|\begin{array}{c} {\lambda_{k} \left(\xi-T_{1} \right)^{\beta _{2} } } \\ {\delta \left(\xi-T_{1} \right)^{\alpha_{2} } } \end{array}\right)\times\nonumber\\
& \times E_{2} \left(\left. \begin{array}{l} {\gamma_{1} ,\gamma_{1} ,1;1,0} \\ {\beta _{1},\beta _{1} ,\alpha_{1} ;\gamma_{1} ,\gamma_{1} ;1,1} \end{array}\right|\begin{array}{c} {\lambda_k T_{1}^{\beta _{1} } } \\ {\delta T_{1}^{\alpha } } \end{array}\right)+T_{1}^{2\beta_{1}}E_{2} \left(\left. \begin{array}{l} {\gamma_{2} ,\gamma_{2} ,1;1,0} \\ {\beta _{2} +1,\beta _{2} ,\alpha_{2} ;\gamma_{2} ,\gamma_{2} ;1,1} \end{array}\right|\begin{array}{c} {\lambda_{k} \left(\xi-T_{1} \right)^{\beta _{2} } } \\ {\delta \left(\xi-T_{1} \right)^{\alpha_{2} } } \end{array}\right)\times\nonumber\\
&\times\lambda_{k}^{3}\varphi_{k}\Gamma^{2}(\gamma_{1})\Gamma(\gamma_{2})E_{2} \left(\left. \begin{array}{l} {\gamma_{1} ,\gamma_{1} ,1;1,0} \\ {\beta _{1},\beta _{1} ,\alpha_{1} ;\gamma_{1} ,\gamma_{1} ;1,1} \end{array}\right|\begin{array}{c} {\lambda_k T_{1}^{\beta _{1} } } \\ {\delta T_{1}^{\alpha_{1} } } \end{array}\right)E_{2} \left(\left. \begin{array}{l} {\gamma_{1} ,\gamma_{1} ,1;1,0} \\ {\beta _{1}+2,\beta _{1} ,\alpha_{1} ;\gamma_{1} ,\gamma_{1} ;1,1} \end{array}\right|\begin{array}{c} {\lambda_k T_{1}^{\beta _{1} } } \\ {\delta T_{1}^{\alpha } } \end{array}\right)\Bigg]. \nonumber
\end{align}
Now we substitute \eqref{41} and \eqref{42} into equation \eqref{23}, after certain simplifications we determine the coefficient ${}_{4}A_{k}$, as follows (see A2 in Appendix section):

\begin{align} \label{43}
         & {}_{4}A_{k}=\frac{1}{\widetilde{\Delta}_{k}}\Bigg[\psi_{k}T_{1}^{\beta_{1}} \Gamma(\gamma_{1}) E_{2} \left(\left. \begin{array}{l} {\gamma_{1} ,\gamma_{1} ,1;1,0} \\ {\beta _{1} +1,\beta _{1} ,\alpha_{1} ;\gamma_{1} ,\gamma_{1} ;1,1} \end{array}\right|\begin{array}{c} {\lambda_k T_{1}^{\beta _{1} } } \\ {\delta T_{1}^{\alpha_{1} } } \end{array}\right)+ \lambda_{k} \psi_{k}T_{1}^{2\beta_{1}} \Gamma^{2}(\gamma_{1}) \times\nonumber\\
&\times\left( E_{2} \left(\left. \begin{array}{l} {\gamma_{1} ,\gamma_{1} ,1;1,0} \\ {\beta _{1} +1,\beta _{1} ,\alpha_{1} ;\gamma_{1} ,\gamma_{1} ;1,1} \end{array}\right|\begin{array}{c} {\lambda_k T_{1}^{\beta _{1} } } \\ {\delta T_{1}^{\alpha_{1} } } \end{array}\right) \right)^{2}-\lambda_{k}\varphi_{k}\Gamma\left(\gamma_{1} \right)\Gamma\left(\gamma_{2} \right)T_{1}^{\beta_{1}}\left(\xi-T_{1} \right)^{\beta _{2}} \times \nonumber \\
&\times E_{2} \left(\left. \begin{array}{l} {\gamma_{1} ,\gamma_{1} ,1;1,0} \\ {\beta _{1},\beta _{1} ,\alpha_{1} ;\gamma_{1} ,\gamma_{1} ;1,1} \end{array}\right|\begin{array}{c} {\lambda_k T_{1}^{\beta _{1} } } \\ {\delta T_{1}^{\alpha_{1} } } \end{array}\right)E_{2} \left(\left. \begin{array}{l} {\gamma_{2} ,\gamma_{2} ,1;1,0} \\ {\beta _{2} +1,\beta _{2} ,\alpha_{2} ;\gamma_{2} ,\gamma_{2} ;1,1} \end{array}\right|\begin{array}{c} {\lambda_{k} \left(\xi-T_{1} \right)^{\beta _{2} } } \\ {\delta \left(\xi-T_{1} \right)^{\alpha_{2} } } \end{array}\right)-\lambda_{k}^{2}T_{1}^{2\beta_{1}}\Gamma\left(\gamma_{2} \right)\times \nonumber\\
& \times\varphi_{k} \Gamma^{2}\left(\gamma_{1} \right)\left(\xi-T_{1} \right)^{\beta _{2}}E_{2} \left(\left. \begin{array}{l} {\gamma_{1} ,\gamma_{1} ,1;1,0} \\ {\beta _{1},\beta _{1} ,\alpha_{1} ;\gamma_{1} ,\gamma_{1} ;1,1} \end{array}\right|\begin{array}{c} {\lambda_k T_{1}^{\beta _{1} } } \\ {\delta T_{1}^{\alpha_{1} } } \end{array}\right)E_{2} \left(\left. \begin{array}{l} {\gamma_{1} ,\gamma_{1} ,1;1,0} \\ {\beta _{1}+2,\beta _{1} ,\alpha ;\gamma_{1} ,\gamma_{1} ;1,1} \end{array}\right|\begin{array}{c} {\lambda_k T_{1}^{\beta _{1} } } \\ {\delta T_{1}^{\alpha_{1} } } \end{array}\right)\times\nonumber\\
& \times E_{2} \left(\left. \begin{array}{l} {\gamma_{2} ,\gamma_{2} ,1;1,0} \\ {\beta _{2} +1,\beta _{2} ,\alpha_{2} ;\gamma_{2} ,\gamma_{2} ;1,1} \end{array}\right|\begin{array}{c} {\lambda_{k} \left(\xi-T_{1} \right)^{\beta _{2} } } \\ {\delta \left(\xi-T_{1} \right)^{\alpha_{2} } } \end{array}\right)+\varphi_{k}\Gamma\left(\gamma_{2} \right)\left(\xi-T_{1} \right)^{\beta _{2}}\times\nonumber \\
&\times E_{2} \left(\left. \begin{array}{l} {\gamma_{2} ,\gamma_{2} ,1;1,0} \\ {\beta _{2} +1,\beta _{2} ,\alpha_{2} ;\gamma_{2} ,\gamma_{2} ;1,1} \end{array}\right|\begin{array}{c} {\lambda_{k} \left(\xi-T_{1} \right)^{\beta _{2} } } \\ {\delta \left(\xi-T_{1} \right)^{\alpha_{2} } } \end{array}\right)+2\varphi_{k}\lambda_{k}\Gamma\left(\gamma_{1} \right)\Gamma\left(\gamma_{2} \right)T_{1}^{\beta_{1}}\left(\xi-T_{1} \right)^{\beta _{2}}\times \nonumber\\
& \times E_{2} \left(\left. \begin{array}{l} {\gamma_{2} ,\gamma_{2} ,1;1,0} \\ {\beta _{2} +1,\beta _{2} ,\alpha_{2}; \gamma_{2} ,\gamma_{2} ; 1,1} \end{array}\right|\begin{array}{c} {\lambda_{k} \left(\xi-T_{1} \right)^{\beta _{2} } } \\ {\delta \left(\xi-T_{1} \right)^{\alpha_{2} } } \end{array}\right)E_{2} \left(\left. \begin{array}{l} {\gamma_{1} ,\gamma_{1} ,1;1,0} \\ {\beta _{1} +1,\beta _{1} ,\alpha_{1} ;\gamma_{1} ,\gamma_{1} ;1,1} \end{array}\right|\begin{array}{c} {\lambda_k T_{1}^{\beta _{1} } } \\ {\delta T_{1}^{\alpha_{1} } } \end{array}\right)+\nonumber\\
&+\varphi_{k}\lambda_{k}^{2}\Gamma^{2}\left(\gamma_{1} \right)\Gamma\left(\gamma_{2} \right)T_{1}^{2\beta_{1}}\left(\xi-T_{1} \right)^{\beta _{2}}E_{2} \left(\left. \begin{array}{l} {\gamma_{2} ,\gamma_{2} ,1;1,0} \\ {\beta _{2} +1,\beta _{2} ,\alpha_{2} ;\gamma_{2} ,\gamma_{2} ;1,1} \end{array}\right|\begin{array}{c} {\lambda_{k} \left(\xi-T_{1} \right)^{\beta _{2} } } \\ {\delta \left(\xi-T_{1} \right)^{\alpha_{2} } } \end{array}\right) \times\nonumber\\
& \times\left( E_{2} \left(\left. \begin{array}{l} {\gamma_{1} ,\gamma_{1} ,1;1,0} \\ {\beta _{1} +1,\beta _{1} ,\alpha_{1};\gamma_{1} ,\gamma_{1} ;1,1} \end{array}\right|\begin{array}{c} {\lambda_k T_{1}^{\beta _{1} } } \\ {\delta T_{1}^{\alpha_{1} } } \end{array}\right) \right)^{2}
- \psi_{k}T_{1}^{\beta_{1}} \Gamma(\gamma_{1}) E_{2} \left(\left. \begin{array}{l} {\gamma_{1} ,\gamma_{1} ,1;1,0} \\ {\beta _{1} ,\beta _{1} ,\alpha_{1} ;\gamma_{1} ,\gamma_{1} ;1,1} \end{array}\right|\begin{array}{c} {\lambda_k T_{1}^{\beta _{1} } } \\ {\delta T_{1}^{\alpha_{1} } } \end{array}\right)+\nonumber\\
&+\psi_{k} \lambda_{k}T_{1}^{\beta_{1}+1}\Gamma(\gamma_{1}) E_{2} \left(\left. \begin{array}{l} {\gamma_{1} ,\gamma_{1} ,1;1,0} \\ {\beta _{1}+2,\beta _{1} ,\alpha_{1} ;\gamma_{1} ,\gamma_{1} ;1,1} \end{array}\right|\begin{array}{c} {\lambda_k T_{1}^{\beta _{1} } } \\ {\delta T_{1}^{\alpha_{1} } } \end{array}\right)-\lambda_{k}\psi_{k}T_{1}^{2\beta_{1}} \Gamma^{2}(\gamma_{1})\times \nonumber\\
& \times E_{2} \left(\left. \begin{array}{l} {\gamma_{1} ,\gamma_{1} ,1;1,0} \\ {\beta _{1} ,\beta _{1} ,\alpha _{1};\gamma_{1} ,\gamma_{1} ;1,1} \end{array}\right|\begin{array}{c} {\lambda_k T_{1}^{\beta _{1} } } \\ {\delta T_{1}^{\alpha_{1} } } \end{array}\right)E_{2} \left(\left. \begin{array}{l} {\gamma_{1} ,\gamma_{1} ,1;1,0} \\ {\beta _{1}+2,\beta _{1} ,\alpha_{1} ;\gamma_{1} ,\gamma_{1} ;1,1} \end{array}\right|\begin{array}{c} {\lambda_k T_{1}^{\beta _{1} } } \\ {\delta T_{1}^{\alpha } } \end{array}\right)+\psi_{k}T_{1}+\nonumber\\
&+2\lambda_{k} E_{2} \left(\left. \begin{array}{l} {\gamma_{1} ,\gamma_{1} ,1;1,0} \\ {\beta _{1}+1,\beta _{1} ,\alpha_{1} ;\gamma_{1} ,\gamma_{1} ;1,1} \end{array}\right|\begin{array}{c} {\lambda_k T_{1}^{\beta _{1} } } \\ {\delta T_{1}^{\alpha_{1} } } \end{array}\right)E_{2} \left(\left. \begin{array}{l} {\gamma_{2} ,\gamma_{2} ,1;1,0} \\ {\beta _{2}+1,\beta _{2} ,\alpha_{2} ;\gamma_{2} ,\gamma_{2} ;1,1} \end{array}\right|\begin{array}{c} {\lambda_k \left(T_{2}- T_{1} \right)^{\beta _{2} } } \\ {\delta \left(T_{2}- T_{1} \right)^{\alpha_{2} } } \end{array}\right)\times\nonumber \\ &\times\psi_{k}T_{1}^{\beta_{1}}\left(T_{2}-T_{1} \right)^{\beta_{2}} \Gamma(\gamma_{1})\Gamma(\gamma_{2})+\lambda_{k}^{2}\psi_{k}\left( E_{2} \left(\left. \begin{array}{l} {\gamma_{1} ,\gamma_{1} ,1;1,0} \\ {\beta _{1} +1,\beta _{1} ,\alpha_{1};\gamma_{1} ,\gamma_{1} ;1,1} \end{array}\right|\begin{array}{c} {\lambda_k T_{1}^{\beta _{1} } } \\ {\delta T_{1}^{\alpha_{1} } } \end{array}\right) \right)^{2}\times\nonumber\\ & \times T_{1}^{2\beta_{1}}\left(T_{2}-T_{1} \right)^{\beta_{2}} \Gamma^{2}(\gamma_{1})\Gamma(\gamma_{2})E_{2} \left(\left. \begin{array}{l} {\gamma_{2} ,\gamma_{2} ,1;1,0} \\ {\beta _{2}+1,\beta _{2} ,\alpha_{2};\gamma_{2} ,\gamma_{2} ; 1,1} \end{array}\right|\begin{array}{c} {\lambda_k \left(T_{2}- T_{1} \right)^{\beta _{2} } } \\ {\delta \left(T_{2}- T_{1} \right)^{\alpha_{2} } } \end{array}\right)-\lambda_{k}\psi_{k}\times\nonumber\\ & \times E_{2} \left(\left. \begin{array}{l} {\gamma_{1} ,\gamma_{1} ,1;1,0} \\ {\beta _{1},\beta _{1} ,\alpha_{1} ;\gamma_{1} ,\gamma_{1} ;1,1} \end{array}\right|\begin{array}{c} {\lambda_k T_{1}^{\beta _{1} } } \\ {\delta T_{1}^{\alpha_{1} } } \end{array}\right)E_{2} \left(\left. \begin{array}{l} {\gamma_{2} ,\gamma_{2} ,1;1,0} \\ {\beta _{2}+1,\beta _{2} ,\alpha_{2} ;\gamma_{2} ,\gamma_{2} ;1,1} \end{array}\right|\begin{array}{c} {\lambda_k \left(T_{2}- T_{1} \right)^{\beta _{2} } } \\ {\delta \left(T_{2}- T_{1} \right)^{\alpha_{2} } } \end{array}\right)\times\nonumber\\ & \times T_{1}^{\beta_{1}}\left(T_{2}-T_{1} \right)^{\beta_{2}} \Gamma(\gamma_{1})\Gamma(\gamma_{2})-\lambda_{k}^{2}T_{1}^{2\beta_{1}}\psi_{k} \Gamma^{2}(\gamma_{1})\Gamma(\gamma_{2})E_{2} \left(\left. \begin{array}{l} {\gamma_{1} ,\gamma_{1} ,1;1,0} \\ {\beta _{1},\beta _{1} ,\alpha_{1};\gamma_{1} ,\gamma_{1} ;1,1} \end{array}\right|\begin{array}{c} {\lambda_k T_{1}^{\beta _{1} } } \\ {\delta T_{1}^{\alpha_{1} } } \end{array}\right)\times\nonumber\\ & \times  E_{2} \left(\left. \begin{array}{l} {\gamma_{1} ,\gamma_{1} ,1;1,0} \\ {\beta _{1}+2,\beta _{1} ,\alpha_{1} ;\gamma_{1} ,\gamma_{1} ;1,1} \end{array}\right|\begin{array}{c} {\lambda_k T_{1}^{\beta _{1} } } \\ {\delta T_{1}^{\alpha_{1} } } \end{array}\right)E_{2} \left(\left. \begin{array}{l} {\gamma_{2} ,\gamma_{2} ,1;1,0} \\ {\beta _{2}+1,\beta _{2} ,\alpha_{2}; \gamma_{2} ,\gamma_{2} ;1,1} \end{array}\right|\begin{array}{c} {\lambda_k \left(T_{2}- T_{1} \right)^{\beta _{2} } } \\ {\delta \left(T_{2}- T_{1} \right)^{\alpha_{2} } } \end{array}\right)\times\nonumber\\ & \times \left(T_{2}-T_{1} \right)^{\beta_{2}}+\psi_{k}\Gamma(\gamma_{2})\left(T_{2}-T_{1} \right)^{\beta_{2}}E_{2} \left(\left. \begin{array}{l} {\gamma_{2} ,\gamma_{2} ,1;1,0} \\ {\beta _{2}+1,\beta _{2} ,\alpha_{2};\gamma_{2} ,\gamma_{2} ;1,1} \end{array}\right|\begin{array}{c} {\lambda_k \left(T_{2}- T_{1} \right)^{\beta _{2} } } \\ {\delta \left(T_{2}- T_{1} \right)^{\alpha_{2} } } \end{array}\right)-\nonumber\\
& -\varphi_{k}\Gamma(\gamma_{2})\left(T_{2}-T_{1} \right)^{\beta_{2}}E_{2} \left(\left. \begin{array}{l} {\gamma_{2} ,\gamma_{2} ,1;1,0} \\ {\beta _{2}+1,\beta _{2} ,\alpha_{2} ;\gamma_{2} ,\gamma_{2} ;1,1} \end{array}\right|\begin{array}{c} {\lambda_k \left(T_{2}- T_{1} \right)^{\beta _{2} } } \\ {\delta \left(T_{2}- T_{1} \right)^{\alpha_{2} } } \end{array}\right)-2\lambda_{k}\varphi_{k}\left(T_{2}-T_{1} \right)^{\beta_{2}}T_1^{\beta_{1}}\Gamma(\gamma_{1})\times\nonumber\\ 
& \times \Gamma(\gamma_{2})E_{2} \left(\left. \begin{array}{l} {\gamma_{1} ,\gamma_{1} ,1;1,0} \\ {\beta _{1}+1,\beta _{1} ,\alpha_{1} ;\gamma_{1} ,\gamma_{1} ;1,1} \end{array}\right|\begin{array}{c} {\lambda_k T_{1}^{\beta _{1} } } \\ {\delta T_{1}^{\alpha_{1} } } \end{array}\right) E_{2} \left(\left. \begin{array}{l} {\gamma_{2} ,\gamma_{2} ,1;1,0} \\ {\beta _{2}+1,\beta _{2} ,\alpha_{2} ;\gamma_{2} ,\gamma_{2} ;1,1} \end{array}\right|\begin{array}{c} {\lambda_k \left(T_{2}- T_{1} \right)^{\beta _{2} } } \\ {\delta \left(T_{2}- T_{1} \right)^{\alpha_{2} } } \end{array}\right)- \nonumber\\
&-\lambda_{k}^{2}E_{2} \left(\left. \begin{array}{l} {\gamma_{2} ,\gamma_{2} ,1;1,0} \\ {\beta _{2}+1,\beta _{2} ,\alpha_{2} ;\gamma_{2} ,\gamma_{2} ; 1,1} \end{array}\right|\begin{array}{c} {\lambda_k \left(T_{2}- T_{1} \right)^{\beta _{2} } } \\ {\delta \left(T_{2}- T_{1} \right)^{\alpha_{2} } } \end{array}\right)\left( E_{2} \left(\left. \begin{array}{l} {\gamma_{1} ,\gamma_{1} ,1;1,0} \\ {\beta _{1}+1,\beta _{1} ,\alpha _{1};\gamma_{1} ,\gamma_{1} ;1,1} \end{array}\right|\begin{array}{c} {\lambda_k T_{1}^{\beta _{1} } } \\ {\delta T_{1}^{\alpha_{1} } } \end{array}\right) \right)^{2}\times\nonumber\\
& \times \varphi_{k}\left(T_{2}-T_{1} \right)^{\beta_{2}}T_1^{2\beta_{1}}\Gamma^{2}(\gamma_{1}) \Gamma(\gamma_{2})+\lambda_{k}\varphi_{k}\Gamma(\gamma_{1}) \Gamma(\gamma_{2})T_1^{\beta_{1}}E_{2} \left(\left. \begin{array}{l} {\gamma_{1} ,\gamma_{1} ,1;1,0} \\ {\beta _{1},\beta _{1} ,\alpha_{1};\gamma_{1} ,\gamma_{1} ;1,1} \end{array}\right|\begin{array}{c} {\lambda_k T_{1}^{\beta _{1} } } \\ {\delta T_{1}^{\alpha_{1} } } \end{array}\right)\times \nonumber\\
& \times \left(T_{2}-T_{1} \right)^{\beta_{2}}E_{2} \left(\left. \begin{array}{l} {\gamma_{2} ,\gamma_{2} ,1;1,0} \\ {\beta _{2}+1,\beta _{2} ,\alpha_{2} ;\gamma_{2} ,\gamma_{2} ;1,1} \end{array}\right|\begin{array}{c} {\lambda_k \left(T_{2}- T_{1} \right)^{\beta _{2} } } \\ {\delta \left(T_{2}- T_{1} \right)^{\alpha_{2} } } \end{array}\right)+\lambda_{k}^{2}\varphi_{k}\left(T_{2}-T_{1} \right)^{\beta_{2}}T_1^{2\beta_{1}}\Gamma^{2}(\gamma_{1}) \times\nonumber\\
&\times \Gamma(\gamma_{2}) E_{2} \left(\left. \begin{array}{l} {\gamma_{1} ,\gamma_{1} ,1;1,0} \\ {\beta _{1},\beta _{1} ,\alpha_{1} ;\gamma_{1} ,\gamma_{1} ;1,1} \end{array}\right|\begin{array}{c} {\lambda_k T_{1}^{\beta _{1} } } \\ {\delta T_{1}^{\alpha_{1} } } \end{array}\right)E_{2} \left(\left. \begin{array}{l} {\gamma_{1} ,\gamma_{1} ,1;1,0} \\ {\beta _{1}+2,\beta _{1} ,\alpha_{1} ;\gamma_{1} ,\gamma_{1} ;1,1} \end{array}\right|\begin{array}{c} {\lambda_k T_{1}^{\beta _{1} } } \\ {\delta T_{1}^{\alpha_{1} } } \end{array}\right)\times\nonumber\\
& \times E_{2} \left(\left. \begin{array}{l} {\gamma_{2} ,\gamma_{2} ,1;1,0} \\ {\beta _{2}+1,\beta _{2} ,\alpha_{2} ;\gamma_{2} ,\gamma_{2} ;1,1} \end{array}\right|\begin{array}{c} {\lambda_k \left(T_{2}- T_{1} \right)^{\beta _{2} } } \\ {\delta \left(T_{2}- T_{1} \right)^{\alpha_{2} } } \end{array}\right).
    \end{align}
By substituting \eqref{42} and \eqref{43} into the equation \eqref{26}, we determine the coefficient ${}_{5}A_{k}$ as follows:
\begin{align} \label{44}
& {}_{5}A_{k}=\frac{1}{\widetilde\Delta_{k}}\Bigg[-2\lambda_{k}\varphi_{k}\Gamma(\gamma_{2})\left(\xi- T_{1} \right)^{\beta_{2}}E_{2} \left(\left. \begin{array}{l} {\gamma_{2} ,\gamma_{2} ,1;1,0} \\ {\beta _{2}+1,\beta _{2} ,\alpha_{2} ;\gamma_{2} ,\gamma_{2} ;1,1} \end{array}\right|\begin{array}{c} {\lambda_k \left(\xi- T_{1} \right)^{\beta_{2}} } \\ {\delta \left(\xi- T_{1} \right)^{\alpha_{2} } } \end{array}\right)-4\lambda_{k}^{2}\varphi_{k}\times\nonumber\\
& \times \Gamma(\gamma_{1}) \Gamma(\gamma_{2})E_{2} \left(\left. \begin{array}{l} {\gamma_{2} ,\gamma_{2} ,1;1,0} \\ {\beta _{2}+1,\beta _{2} ,\alpha_{2} ;\gamma_{2} ,\gamma_{2} ;1,1} \end{array}\right|\begin{array}{c} {\lambda_k \left(\xi- T_{1} \right)^{\beta_{2}} } \\ {\delta \left(\xi- T_{1} \right)^{\alpha_{2} } } \end{array}\right)E_{2} \left(\left. \begin{array}{l} {\gamma_{1} ,\gamma_{1} ,1;1,0} \\ {\beta _{1}+1,\beta _{1} ,\alpha_{1} ;\gamma_{1} ,\gamma_{1} ;1,1} \end{array}\right|\begin{array}{c} {\lambda_k T_{1}^{\beta _{1} } } \\ {\delta T_{1}^{\alpha_{1} } } \end{array}\right)\times\nonumber\\
& \times T_1^{\beta_{1}}\left(\xi- T_{1} \right)^{\beta_{2}}-2\lambda_{k}^{3}\varphi_{k}T_1^{2\beta_{1}}\Gamma^{2}(\gamma_{1}) \Gamma(\gamma_{2})E_{2} \left(\left. \begin{array}{l} {\gamma_{2} ,\gamma_{2} ,1;1,0} \\ {\beta _{2}+1,\beta _{2} ,\alpha_{2} ;\gamma_{2} ,\gamma_{2} ;1,1} \end{array}\right|\begin{array}{c} {\lambda_k \left(\xi- T_{1} \right)^{\beta_{2}} } \\ {\delta \left(\xi- T_{1} \right)^{\alpha_{2} } } \end{array}\right)\times\nonumber\\
& \times \left(\xi- T_{1} \right)^{\beta_{2}}\left( E_{2} \left(\left. \begin{array}{l} {\gamma_{1} ,\gamma_{1} ,1;1,0} \\ {\beta _{1}+1,\beta _{1} ,\alpha_{1} ;\gamma_{1} ,\gamma_{1} ;1,1} \end{array}\right|\begin{array}{c} {\lambda_k T_{1}^{\beta _{1} } } \\ {\delta T_{1}^{\alpha_{1} } } \end{array}\right) \right)^{2}+4\lambda_{k}^{2}\varphi_{k}\Gamma(\gamma_{1}) \Gamma(\gamma_{2})T_1^{\beta_{1}}\left(\xi- T_{1} \right)^{\beta_{2}}\times\nonumber\\
& \times E_{2} \left(\left. \begin{array}{l} {\gamma_{2} ,\gamma_{2} ,1;1,0} \\ {\beta _{2}+1,\beta _{2} ,\alpha_{2} ;\gamma_{2} ,\gamma_{2} ;1,1} \end{array}\right|\begin{array}{c} {\lambda_k \left(\xi- T_{1} \right)^{\beta_{2}} } \\ {\delta \left(\xi- T_{1} \right)^{\alpha_{2} } } \end{array}\right)E_{2} \left(\left. \begin{array}{l} {\gamma_{1} ,\gamma_{1} ,1;1,0} \\ {\beta _{1},\beta _{1} ,\alpha_{1} ;\gamma_{1} ,\gamma_{1} ;1,1} \end{array}\right|\begin{array}{c} {\lambda_k T_{1}^{\beta _{1} } } \\ {\delta T_{1}^{\alpha_{1} } } \end{array}\right)+\lambda_{k}^{3}\times\nonumber\\
& \varphi_{k}T_1^{\beta_{1}}\Gamma^{2}(\gamma_{1}) \Gamma(\gamma_{2})\left( \left(\xi- T_{1} \right)^{\beta_{2}}+T_1^{\beta_{1}} \right)E_{2} \left(\left. \begin{array}{l} {\gamma_{2} ,\gamma_{2} ,1;1,0} \\ {\beta _{2}+1,\beta _{2} ,\alpha_{2} ;\gamma_{2} ,\gamma_{2} ;1,1} \end{array}\right|\begin{array}{c} {\lambda_k \left(\xi- T_{1} \right)^{\beta_{2}} } \\ {\delta \left(\xi- T_{1} \right)^{\alpha_{2} } } \end{array}\right)\times \nonumber\\
& \times E_{2} \left(\left. \begin{array}{l} {\gamma_{1} ,\gamma_{1} ,1;1,0} \\ {\beta _{1},\beta _{1} ,\alpha_{1} ;\gamma_{1} ,\gamma_{1} ;1,1} \end{array}\right|\begin{array}{c} {\lambda_k T_{1}^{\beta _{1} } } \\ {\delta T_{1}^{\alpha_{1} } } \end{array}\right)E_{2} \left(\left. \begin{array}{l} {\gamma_{1} ,\gamma_{1} ,1;1,0} \\ {\beta _{1}+2,\beta _{1} ,\alpha_{1} ;\gamma_{1} ,\gamma_{1} ; 1,1} \end{array}\right|\begin{array}{c} {\lambda_k T_{1}^{\beta _{1} } } \\ {\delta T_{1}^{\alpha_{1} } } \end{array}\right)-\lambda_{k}^{2}\varphi_{k}\times\nonumber\\
& \times \Gamma(\gamma_{1}) \Gamma(\gamma_{2})T_1^{\beta_{1}}\left(T_{2}-T_{1} \right)^{\beta_{2}}E_{2} \left(\left. \begin{array}{l} {\gamma_{2} ,\gamma_{2} ,1;1,0} \\ {\beta _{2}+1,\beta _{2} ,\alpha_{2} ;\gamma_{2} ,\gamma_{2} ;1,1} \end{array}\right|\begin{array}{c} {\lambda_k \left(T_{2}- T_{1} \right)^{\beta _{2} } } \\ {\delta \left(T_{2}- T_{1} \right)^{\alpha_{2} } } \end{array}\right)\times\nonumber\\
& \times E_{2} \left(\left. \begin{array}{l} {\gamma_{1} ,\gamma_{1} ,1;1,0} \\ {\beta _{1},\beta _{1} ,\alpha_{1} ;\gamma_{1} ,\gamma_{1} ;1,1} \end{array}\right|\begin{array}{c} {\lambda_k T_{1}^{\beta _{1} } } \\ {\delta T_{1}^{\alpha_{1} } } \end{array}\right)-\lambda_{k}^{3}\varphi_{k}\Gamma^{2}(\gamma_{1}) \Gamma(\gamma_{2})T_1^{2\beta_{1}}\left(T_{2}-T_{1} \right)^{\beta_{2}}\times \nonumber\\
&\times E_{2} \left(\left. \begin{array}{l} {\gamma_{2} ,\gamma_{2} ,1;1,0} \\ {\beta _{2}+1,\beta _{2} ,\alpha_{2} ;\gamma_{2} ,\gamma_{2} ;1,1} \end{array}\right|\begin{array}{c} {\lambda_k \left(T_{2}- T_{1} \right)^{\beta _{2} } } \\ {\delta \left(T_{2}- T_{1} \right)^{\alpha_{2} } } \end{array}\right)E_{2} \left(\left. \begin{array}{l} {\gamma_{1} ,\gamma_{1} ,1;1,0} \\ {\beta _{1}+2,\beta _{1} ,\alpha_{1} ;\gamma_{1} ,\gamma_{1} ;1,1} \end{array}\right|\begin{array}{c} {\lambda_k T_{1}^{\beta _{1} } } \\ {\delta T_{1}^{\alpha_{1} } } \end{array}\right)\times\nonumber\\
& \times E_{2} \left(\left. \begin{array}{l} {\gamma_{1} ,\gamma_{1} ,1;1,0} \\ {\beta _{1},\beta _{1} ,\alpha_{1} ;\gamma_{1} ,\gamma_{1} ;1,1} \end{array}\right|\begin{array}{c} {\lambda_k T_{1}^{\beta _{1} } } \\ {\delta T_{1}^{\alpha_{1} } } \end{array}\right)+\psi_{k}-\varphi_{k}-2\lambda_{k}\varphi_{k}T_{1}^{\beta _{1} }\Gamma(\gamma_{1}) \times \nonumber\\
& \times E_{2} \left(\left. \begin{array}{l} {\gamma_{1} ,\gamma_{1} ,1;1,0} \\ {\beta _{1}+1,\beta _{1} ,\alpha_{1} ;\gamma_{1} ,\gamma_{1} ;1,1} \end{array}\right|\begin{array}{c} {\lambda_k T_{1}^{\beta _{1} } } \\ {\delta T_{1}^{\alpha_{1} } } \end{array}\right)-\lambda_{k}^{2}\left( E_{2} \left(\left. \begin{array}{l} {\gamma_{1} ,\gamma_{1} ,1;1,0} \\ {\beta _{1}+1,\beta _{1} ,\alpha_{1} ;\gamma_{1} ,\gamma_{1} ;1,1} \end{array}\right|\begin{array}{c} {\lambda_k T_{1}^{\beta _{1} } } \\ {\delta T_{1}^{\alpha_{1} } } \end{array}\right) \right)^{2}\times \nonumber\\
& \times \varphi_{k}T_{1}^{2\beta _{1} }\Gamma^{2}(\gamma_{1})+\lambda_{k}\varphi_{k} T_{1}^{\beta _{1} }\Gamma(\gamma_{1})E_{2} \left(\left. \begin{array}{l} {\gamma_{1} ,\gamma_{1} ,1;1,0} \\ {\beta _{1},\beta _{1} ,\alpha_{1} ;\gamma_{1} ,\gamma_{1} ;1,1} \end{array}\right|\begin{array}{c} {\lambda_k T_{1}^{\beta _{1} } } \\ {\delta T_{1}^{\alpha_{1} } } \end{array}\right)+\lambda_{k}^{2}\varphi_{k}T_{1}^{2\beta _{1} }\Gamma^{2}(\gamma_{1})\times \nonumber\\
& \times  E_{2} \left(\left. \begin{array}{l} {\gamma_{1} ,\gamma_{1} ,1;1,0} \\ {\beta _{1},\beta _{1} ,\alpha_{1} ;\gamma_{1} ,\gamma_{1} ;1,1} \end{array}\right|\begin{array}{c} {\lambda_k T_{1}^{\beta _{1} } } \\ {\delta T_{1}^{\alpha_{1} } } \end{array}\right)E_{2} \left(\left. \begin{array}{l} {\gamma_{1} ,\gamma_{1} ,1;1,0} \\ {\beta _{1}+2,\beta _{1} ,\alpha_{1} ;\gamma_{1} ,\gamma_{1} ;1,1} \end{array}\right|\begin{array}{c} {\lambda_k T_{1}^{\beta _{1} } } \\ {\delta T_{1}^{\alpha } } \end{array}\right)-\lambda_k^{2}\psi_{k}T_{1}^{2\beta _{1} }\Gamma^{2}(\gamma_{1})\times\nonumber\\
& \times\left( E_{2} \left(\left. \begin{array}{l} {\gamma_{1} ,\gamma_{1} ,1; 1,0} \\ {\beta _{1}+1,\beta _{1} ,\alpha_{1}; \gamma_{1} ,\gamma_{1} ;1,1} \end{array}\right|\begin{array}{c} {\lambda_k T_{1}^{\beta _{1} } } \\ {\delta T_{1}^{\alpha_{1} } } \end{array}\right) \right)^{2}+\lambda_{k}\psi_{k}E_{2} \left(\left. \begin{array}{l} {\gamma_{1} ,\gamma_{1} ,1;1,0} \\ {\beta _{1},\beta _{1} ,\alpha_{1} ;\gamma_{1} ,\gamma_{1} ;1,1} \end{array}\right|\begin{array}{c} {\lambda_k T_{1}^{\beta _{1} } } \\ {\delta T_{1}^{\alpha_{1} } } \end{array}\right)\times\nonumber\\
& \times T_{1}^{\beta _{1} }\Gamma(\gamma_{1})+E_{2} \left(\left. \begin{array}{l} {\gamma_{1} ,\gamma_{1} ,1;1,0} \\ {\beta _{1},\beta _{1} ,\alpha_{1} ;\gamma_{1} ,\gamma_{1} ;1,1} \end{array}\right|\begin{array}{c} {\lambda_k T_{1}^{\beta _{1} } } \\ {\delta T_{1}^{\alpha_{1} } } \end{array}\right)E_{2} \left(\left. \begin{array}{l} {\gamma_{1} ,\gamma_{1} ,1;1,0} \\ {\beta _{1}+2,\beta _{1} ,\alpha_{1} ;\gamma_{1} ,\gamma_{1} ;1,1} \end{array}\right|\begin{array}{c} {\lambda_k T_{1}^{\beta _{1} } } \\ {\delta T_{1}^{\alpha_{1} } } \end{array}\right)\times\nonumber\\
& \times \lambda_{k}^{2}\psi_{k}T_{1}^{2\beta _{1} }\Gamma^{2}(\gamma_{1})-2\lambda_{k}^{2}\psi_{k}\Gamma(\gamma_{1}) \Gamma(\gamma_{2})T_1^{\beta_{1}}\left(T_{2}-T_{1} \right)^{\beta_{2}}E_{2} \left(\left. \begin{array}{l} {\gamma_{1} ,\gamma_{1} ,1;1,0} \\ {\beta _{1}+1,\beta _{1} ,\alpha_{1} ;\gamma_{1} ,\gamma_{1} ;1,1} \end{array}\right|\begin{array}{c} {\lambda_k T_{1}^{\beta _{1} } } \\ {\delta T_{1}^{\alpha_{1} } } \end{array}\right)\times \nonumber\\
& \times E_{2} \left(\left. \begin{array}{l} {\gamma_{2} ,\gamma_{2} ,1;1,0} \\ {\beta _{2}+1,\beta _{2} ,\alpha_{2} ;\gamma_{2} ,\gamma_{2} ;1,1} \end{array}\right|\begin{array}{c} {\lambda_k \left(T_{2}- T_{1} \right)^{\beta _{2} } } \\ {\delta \left(T_{2}- T_{1} \right)^{\alpha_{2} } }\end{array}\right)-\lambda_{k}^{3}\psi_{k}\Gamma^{2}(\gamma_{1}) \Gamma(\gamma_{2})T_1^{2\beta_{1}}\left(T_{2}-T_{1} \right)^{\beta_{2}}\times \nonumber \\
& \times \left( E_{2} \left(\left. \begin{array}{l} {\gamma_{1} ,\gamma_{1} ,1;1,0} \\ {\beta _{1}+1,\beta _{1} ,\alpha_{1} ;\gamma_{1} ,\gamma_{1} ;1,1} \end{array}\right|\begin{array}{c} {\lambda_k T_{1}^{\beta _{1} } } \\ {\delta T_{1}^{\alpha_{1} } } \end{array}\right) \right)^{2}E_{2} \left(\left. \begin{array}{l} {\gamma_{2} ,\gamma_{2} ,1;1,0} \\ {\beta _{2}+1,\beta _{2} ,\alpha_{2} ;\gamma_{2} ,\gamma_{2} ;1,1} \end{array}\right|\begin{array}{c} {\lambda_k \left(T_{2}- T_{1} \right)^{\beta _{2} } } \\ {\delta \left(T_{2}- T_{1} \right)^{\alpha_{2} } } \end{array}\right)+\nonumber\\
& +\lambda_{k}^{2}\psi_{k}T_1^{\beta_{1}} E_{2} \left(\left. \begin{array}{l} {\gamma_{1} ,\gamma_{1} ,1;1,0} \\ {\beta _{1},\beta _{1} ,\alpha_{1} ;\gamma_{1} ,\gamma_{1} ;1,1} \end{array}\right|\begin{array}{c} {\lambda_k T_{1}^{\beta _{1} } } \\ {\delta T_{1}^{\alpha_{1} } } \end{array}\right) E_{2} \left(\left. \begin{array}{l} {\gamma_{2} ,\gamma_{2}, 1;1,0} \\ {\beta _{2}+1,\beta _{2} ,\alpha_{2} ;\gamma_{2} ,\gamma_{2} ;1,1} \end{array}\right|\begin{array}{c} {\lambda_k \left(T_{2}- T_{1} \right)^{\beta _{2} } } \\ {\delta \left(T_{2}- T_{1} \right)^{\alpha_{2} } } \end{array}\right)+\nonumber\\
& \times \Gamma(\gamma_{1}) \Gamma(\gamma_{2})\left(T_{2}-T_{1} \right)^{\beta_{2}}+\lambda_{k}^{3}\psi_{k} \Gamma(\gamma_{2})T_1^{2\beta_{1}}\left(T_{2}- T_{1} \right)^{\beta _{2} } E_{2} \left(\left. \begin{array}{l} {\gamma_{1} ,\gamma_{1} ,1;1,0} \\ {\beta _{1},\beta _{1} ,\alpha_{1} ;\gamma_{1} ,\gamma_{1} ;1,1} \end{array}\right|\begin{array}{c} {\lambda_k T_{1}^{\beta _{1} } } \\ {\delta T_{1}^{\alpha_{1} } } \end{array}\right)\times\nonumber\\
& \times \Gamma^{2}(\gamma_{1}) E_{2} \left(\left. \begin{array}{l} {\gamma_{1} ,\gamma_{1} ,1;1,0} \\ {\beta _{1}+2,\beta _{1} ,\alpha_{1} ;\gamma_{1} ,\gamma_{1} ;1,1} \end{array}\right|\begin{array}{c} {\lambda_k T_{1}^{\beta _{1} } } \\ {\delta T_{1}^{\alpha_{1} } } \end{array}\right)E_{2} \left(\left. \begin{array}{l} {\gamma_{2} ,\gamma_{2} ,1;1,0} \\ {\beta _{2}+1,\beta _{2} ,\alpha_{2} ;\gamma_{2} ,\gamma_{2} ;1,1} \end{array}\right|\begin{array}{c} {\lambda_k \left(T_{2}- T_{1} \right)^{\beta _{2} } } \\ {\delta \left(T_{2}- T_{1} \right)^{\alpha_{2} } } \end{array}\right) -\nonumber\\
& - \lambda_{k}\psi_{k}\Gamma(\gamma_{2})\left(T_{2}- T_{1} \right)^{\beta _{2} } E_{2} \left(\left. \begin{array}{l} {\gamma_{2} ,\gamma_{2} ,1;1,0} \\ {\beta _{2}+1,\beta _{2} ,\alpha_{2} ;\gamma_{2} ,\gamma_{2} ;1,1} \end{array}\right|\begin{array}{c} {\lambda_k \left(T_{2}- T_{1} \right)^{\beta _{2} } } \\ {\delta \left(T_{2}- T_{1} \right)^{\alpha_{2} } } \end{array}\right)+\lambda_{k}\varphi_{k}\left(T_{2}- T_{1} \right)^{\beta _{2} }\times\nonumber\\
&\times\Gamma(\gamma_{2})E_{2} \left(\left. \begin{array}{l} {\gamma_{2} ,\gamma_{2} ,1;1,0} \\ {\beta _{2}+1,\beta _{2} ,\alpha_{2}; \gamma_{2} ,\gamma_{2} ;1,1} \end{array}\right|\begin{array}{c} {\lambda_k \left(T_{2}- T_{1} \right)^{\beta _{2} } } \\ {\delta \left(T_{2}- T_{1} \right)^{\alpha_{2} } } \end{array}\right)-2\lambda_{k}^{2}\varphi_{k}T_{1}^{\beta_{1}}\left(T_{2}- T_{1} \right)^{\beta _{2} }\Gamma(\gamma_{1}) \Gamma(\gamma_{2})\times\nonumber\\
& \times  E_{2} \left(\left. \begin{array}{l} {\gamma_{1} ,\gamma_{1} ,1;1,0} \\ {\beta _{1}+1,\beta _{1} ,\alpha_{1} ;\gamma_{1} ,\gamma_{1} ;1,1} \end{array}\right|\begin{array}{c} {\lambda_k T_{1}^{\beta _{1} } } \\ {\delta T_{1}^{\alpha_{1} } } \end{array}\right)E_{2} \left(\left. \begin{array}{l} {\gamma_{2} ,\gamma_{2} ,1;1,0} \\ {\beta _{2}+1,\beta _{2} ,\alpha_{2} ;\gamma_{2} ,\gamma_{2} ;1,1} \end{array}\right|\begin{array}{c} {\lambda_k \left(T_{2}- T_{1} \right)^{\beta _{2} } } \\ {\delta \left(T_{2}- T_{1} \right)^{\alpha_{2} } } \end{array}\right)+\nonumber\\
&+\lambda_{k}^{3}\varphi_{k}T_{1}^{2\beta_{1}}\left(T_{2}- T_{1} \right)^{\beta _{2} }\Gamma^{2}(\gamma_{1}) \Gamma(\gamma_{2}) \left( E_{2} \left(\left. \begin{array}{l} {\gamma_{1} ,\gamma_{1} ,1;1,0} \\ {\beta _{1}+1,\beta _{1} ,\alpha_{1} ;\gamma_{1} ,\gamma_{1} ;1,1} \end{array}\right|\begin{array}{c} {\lambda_k T_{1}^{\beta _{1} } } \\ {\delta T_{1}^{\alpha_{1} } } \end{array}\right)\right)^{2}\times\nonumber\\
& \times E_{2} \left(\left. \begin{array}{l} {\gamma_{2} ,\gamma_{2} ,1;1,0} \\ {\beta _{2}+1,\beta _{2} ,\alpha_{2} ;\gamma_{2} ,\gamma_{2} ;1,1} \end{array}\right|\begin{array}{c} {\lambda_k \left(T_{2}- T_{1} \right)^{\beta _{2} } } \\ {\delta \left(T_{2}- T_{1} \right)^{\alpha_{2} } } \end{array}\right)\Bigg].
\end{align}
   Now, to prove the uniform convergence of solutions in the form of a multiple series,we will write the following estimates using the estimate for the bivariate Mittag-Leffler type function $E_{2}(x,y)$ \cite{12}, as presented in \cite{14}.

\begin{equation} \label{45}
	\begin{aligned}
		& \left| {{E}_{2}}\left( \left. \begin{matrix}
			\gamma_{1} ,\gamma_{1} ,1;1,0  \\
			{{\beta }_{1}},{{\beta }_{1}},\alpha_{1} ;\gamma_{1} ,\gamma_{1} ;1,1  \\
		\end{matrix} \right|\begin{matrix}
			\lambda_{k}{{t}^{{{\beta }_{1}}}}  \\
			\delta {{t}^{\alpha_{1} }}  \\
		\end{matrix} \right) \right|\le C_{1}, \\
	\end{aligned}
\end{equation}

\begin{equation} \label{46}
	\begin{aligned}
		&\left| {{E}_{2}}\left( \left. \begin{matrix}
			\gamma_{1} ,\gamma_{1} ,1;1,0  \\
			{{\beta }_{1}}+1,{{\beta }_{1}},\alpha_{1} ;\gamma_{1} ,\gamma_{1} ;1,1  \\
		\end{matrix} \right|\begin{matrix}
			\lambda_{k}{{t}^{{{\beta }_{1}}}}  \\
			\delta {{t}^{\alpha_{1} }}  \\
		\end{matrix} \right) \right|\le C_{2}, \\
	\end{aligned}
\end{equation}

\begin{equation} \label{47}
	\begin{aligned}
		& \left| {{E}_{2}}\left( \left. \begin{matrix}
			\gamma_{1} ,\gamma_{1} ,1;1,0  \\
			{{\beta }_{1}}+2,{{\beta }_{1}},\alpha_{1} ;\gamma_{1} ,\gamma_{1} ;1,1  \\
		\end{matrix} \right|\begin{matrix}
			\lambda_{k}{{t}^{{{\beta }_{1}}}}  \\
			\delta {{t}^{\alpha_{1} }}  \\
		\end{matrix} \right) \right|\le C_{3}, \\
	\end{aligned}
\end{equation}

\begin{equation} \label{48}
	\begin{aligned}
		& \left| {{E}_{2}}\left( \left. \begin{matrix}
			\gamma_{2} ,\gamma_{2} ,1;1,0  \\
			{{\beta }_{2}}+1,{{\beta }_{2}},\alpha_{2} ;\gamma_{2} ,\gamma_{2} ;1,1  \\
		\end{matrix} \right|\begin{matrix}
			\lambda_{k}{{(t-T_{1})}^{{{\beta }_{2}}}}  \\
			\delta {{(t-T_{1})}^{\alpha_{2} }}  \\
		\end{matrix} \right) \right|\le C_{4},  \\
	\end{aligned}
\end{equation}

\begin{equation} \label{49}
	\begin{aligned}
		& \left| {{E}_{2}}\left( \left. \begin{matrix}
			\gamma_{3} ,\gamma_{3} ,1;1,0  \\
			{{\beta }_{3}}+1,{{\beta }_{3}},\alpha_{3} ;\gamma_{3} ,\gamma_{3} ;1,1  \\
		\end{matrix} \right|\begin{matrix}
			\lambda_{k}{{(t-T_{2})}^{{{\beta }_{3}}}}  \\
			\delta {{(t-T_{2})}^{\alpha_{3} }}  \\
		\end{matrix} \right) \right|\le C_{5},  \\
	\end{aligned}
\end{equation}

\begin{equation} \label{50}
	\begin{aligned}
		& \left| {{E}_{2}}\left( \left. \begin{matrix}
			\gamma_{3} ,\gamma_{3} ,1;1,0  \\
			{{\beta }_{3}}+2,{{\beta }_{3}},\alpha_{3} ;\gamma_{3} ,\gamma_{3} ;1,1  \\
		\end{matrix} \right|\begin{matrix}
			\lambda_{k}{{(t-T_{2})}^{{{\beta }_{3}}}}  \\
			\delta {{(t-T_{2})}^{\alpha_{3} }}  \\
		\end{matrix} \right) \right|\le C_{6}. \\
	\end{aligned}
\end{equation}

For convenience, we performed the above calculations without substituting the value of $\lambda_{k}$. Now, to show the uniform convergence of the solution, we take $\lambda_{k}$ as $-\mu_{k}^2$. Here, it is known that $\mu_{k}=k\pi-\frac{\pi}{4}$.

\section{Convergence of infinite series.}
\begin{theorem} \cite{9}
    Let $f(x)$ be a function defined on the interval [0,1] such that $f(x)$ is differentiable $2s$ times $s>1$ and such that 
    \begin{enumerate}[label=\alph*)]
    \item $f(0)=f'(0)=...=f^{(2s-1)}(0)=0$,
    \item $f^{(2s)}(x)$ is bounded (this derivative may not exist at certain points),
    \item $f(1)=f'(1)=...=f^{(2s-2)}(1)=0$.
    Then the following inequality is satisfied by the Fourier-Bessel coefficients of $f(x)$:
    $$\left| c_{n} \right|\le \frac{C}{\lambda^{(2s-\frac{1}{2})}}, (C=const)$$
\end{enumerate}
\end{theorem}

     The proof of Theorem 5.1 can be found in reference \cite{9}.
\begin{lemma}
    If $\left\{  \varphi(r),\psi(r)\right\} \in  C^{m+2}[0;1]$ and the conditions of Theorem 5.1 are satisfied, then the inequalities hold for the special cases of the series
    \[\sum\nolimits_{n=1}^{\infty} \mu_k^m |\varphi_k|,\sum\nolimits_{n=1}^{\infty} \mu_k^m |\psi_k|\].
 $$\sum_{n=1}^{\infty }\left| \varphi_{k} \right|\le \frac{C}{\mu_{k}^{2s-0,5}},\sum_{n=1}^{\infty }\left| \psi_{k} \right|\le \frac{C}{\mu_{k}^{2s-0,5}},$$
 $$\sum_{n=1}^{\infty }\mu_{k}^{2}\left| \varphi_{k} \right|\le \frac{C}{\mu_{k}^{2s-2,5}},\sum_{n=1}^{\infty }\mu_{k}^{2}\left| \psi_{k} \right|\le \frac{C}{\mu_{k}^{2s-2,5}},$$
    $$\sum_{n=1}^{\infty }\mu_{k}^{4}\left| \varphi_{k} \right|\le \frac{C}{\mu_{k}^{2s-4,5}},\sum_{n=1}^{\infty }\mu_{k}^{4}\left| \psi_{k} \right|\le \frac{C}{\mu_{k}^{2s-4,5},}$$
    $$\sum_{n=1}^{\infty }\mu_{k}^{6}\left| \varphi_{k} \right|\le \frac{C}{\mu_{k}^{2s-6,5}},\sum_{n=1}^{\infty }\mu_{k}^{6}\left| \psi_{k} \right|\le \frac{C}{\mu_{k}^{2s-6,5}},$$
$$\sum_{n=1}^{\infty }\mu_{k}^{8}\left| \varphi_{k} \right|\le \frac{C}{\mu_{k}^{2s-8,5}},\sum_{n=1}^{\infty }\mu_{k}^{8}\left| \psi_{k} \right|\le \frac{C}{\mu_{k}^{2s-8,5}}.$$

\end{lemma}

\begin{proof}
    It is known from the asymptotic of $\mu_{k}=(k-\frac{1}{4}) \pi$ $(\mu_{k} \sim k)$. Based on this, the series $\sum_{k=1}^{\infty}\frac{C}{\mu_{k}^p}$ converges when $p>1$. In our case, for the convergence of the highest-order ($m=8$) series:
$$2s-8,5>1 \implies 2s>9,5$$
    That is, if the function is differentiable at least 10 times ($2s=10$), then convergent majorant series exist for all the series listed above:
    $$\mu_{k}^{m}\left| \varphi_{k} \right| \le \frac{C}{\mu_{k}^{10-m-0,5}},\mu_{k}^{m}\left| \psi_{k} \right| \le \frac{C}{\mu_{k}^{10-m-0,5}}$$
    For each $m \in \{0, 2, 4, 6, 8\}$, the exponent is greater than 1 ($p \ge 1.5$). By the Weierstrass M-test, all these series converge absolutely and uniformly on the interval $[0, 1]$.
\end{proof}

In order to prove a uniform convergence of infinite series \eqref{9} let us start with the following estimation

$$\left| u(t,r) \right|\le \sum_{k=1}^{\infty}\left| {}_{1}U_{k}(t)J_{0}(\lambda_{k}r) \right|\le \sum_{k=1}^{\infty}\left| {}_{1}U_{k}(t)\right|,\,\,\,\,\,\, (\left| J_{0}(x) \le 1 \right|, \forall x \in \mathbb{R})
$$
To complete this estimation, we need to estimate $\left| {}_{1}U_{k}(t)\right|$. For this aim, first using \eqref{15} we will obtain
\begin{align*} \label{51}
 &\left| {}_{1}U_{k} (t) \right|\le  \left| {}_{2} A_{k} \right| t+\mu_{k}^{2}\left| \varphi_{k} \right| t^{\beta _{1} }  \Gamma\left(\gamma_{1} \right)  
 \left| E_{2} \left(\left. \begin{array}{l} {\gamma_{1} ,\gamma_{1} ,1;1,0} \\ {\beta _{1} +1,\beta _{1} ,\alpha_{1} ;\gamma_{1} ,\gamma_{1} ;1,1} \end{array}\right|\begin{array}{c} {-\mu_{k}^2 t^{\beta _{1} } } \\ {\delta t^{\alpha_{1} } } \end{array}\right) \right|+\\&+\left| \varphi_{k} \right|
+\mu_{k}^2\left| {}_{2} A_{k} \right|   t^{\beta _{1} +1} \Gamma\left(\gamma_{1} \right) \left| E_{2} \left(\left. \begin{array}{l} {\gamma_{1} ,\gamma_{1} ,1;1,0} \\ {\beta _{1} +2,\beta _{1} ,\alpha_{1} ;\gamma_{1} ,\gamma_{1} ;1,1} \end{array}\right|\begin{array}{c} {-\mu_{k}^2 t^{\beta _{1} } } \\ {\delta t^{\alpha_{1} } } \end{array}\right) \right|+\\
&+\Gamma\left(\gamma_{1} \right)  \left| f_{k} \right| t^{\beta_{1}}\left| E_{2} \left(\left. \begin{array}{l} {\gamma_{1} ,\gamma_{1} ,1;1,0} \\ {\beta _{1} +1,\beta _{1} ,\alpha_{1} ;\gamma_{1} ,\gamma_{1} ;1,1} \end{array}\right|\begin{array}{c} {-\mu_{k}^2t^{\beta _{1} } } \\ {\delta t^{\alpha_{1} } } \end{array}\right) \right|.  
\end{align*}
Now using \eqref{46},\eqref{47} we get
\begin{align}
&\left| {}_{1}U_{k} (t) \right|\le  \left| {}_{2} A_{k} \right| t+\mu^{2}_{k}\left| \varphi_{k} \right| t^{\beta _{1} }  \Gamma\left(\gamma_{1} \right)  
 C_{2}+\mu^{2}_{k}\left| {}_{2} A_{k} \right|  t^{\beta _{1} +1}  \Gamma\left(\gamma_{1} \right) C_{3}+\left| \varphi_{k} \right|+\Gamma\left(\gamma_{1} \right)  \left| f_{k} \right| t^{\beta_{1}}C_{2}.
\end{align}
    Based on \eqref{40} and \eqref{45}-\eqref{47}, considering $\frac{1}{\left|\widetilde{\Delta}_{k}\right|}<\varepsilon$ ($\varepsilon$ is a positive number), we deduce
\begin{equation} \label{52}
    \begin{split}
  &\left| {}_{2}A_{k} \right|\le \left| \psi_{k} \right|\varepsilon+\left| \varphi_{k} \right|\varepsilon+\left| \psi_{k}  \right|\varepsilon T_{1}^{\beta_{1}-1} \Gamma(\gamma_{1})C_{1}  +\left| \varphi_{k} \right|\varepsilon T_{1}^{\beta_{1}-1} \Gamma(\gamma_{1})C_{1}+2\mu^{2}_{k}\left| \varphi_{k} \right|\Gamma\left(\gamma_{2} \right)
 \left(\xi-T_{1} \right)^{\beta _{2}} \varepsilon C_{4}+\\
&+\mu^{2}_{k} \left| \varphi_{k} \right|\Gamma(\gamma_{1})T_{1}^{\beta_{1}}\varepsilon C_{2}+2\mu^{4}_{k} \left| \varphi_{k} \right|\Gamma\left(\gamma_{1} \right)\Gamma\left(\gamma_{2} \right)\left(\xi-T_{1} \right)^{\beta _{2}}T_{1}^{\beta _{1} } \varepsilon C_{2}C_{4}+\mu^{2}_{k}\left| \psi_{k} \right|T_{1}^{\beta _{1} }\Gamma(\gamma_{1})\varepsilon C_{2}.
    \end{split}
\end{equation}
Let us introduce the following notations:
\begin{align*}
   & S_{1}=\varepsilon+\varepsilon \Gamma(\gamma_{1})T_{1}^{\beta_{1}-1}C_{1}, \\
& S_{2}=2\Gamma(\gamma_{2})\left( \xi-T_{1} \right)^{\beta_{2}}\varepsilon C_{4},\\
& S_{3}=2\Gamma(\gamma_{1})\Gamma(\gamma_{2})T_{1}^{\beta_{1}}\left( \xi-T_{1} \right)^{\beta_{2}}\varepsilon C_{2} C_{4},\\
& S_{4}=T_{1}^{\beta_{1}}\varepsilon \Gamma(\gamma_{1})C_{2}.
\end{align*}
Using the above notations, we rewrite \eqref{52}
\begin{equation} \label{53}
    \left| {}_{2}A_{k} \right|\le \left| \varphi_{k} \right| S_{1}+\mu^{2}_{k}\left| \varphi_{k} \right| \left( S_{2}+S_4 \right)+\left| \psi_{k} \right|S_{1}+\mu^{2}_{k}\left| \psi_{k} \right|S_{4}+\mu^{4}_{k}\left| \varphi_{k} \right|S_{3}. 
\end{equation}
Similarly, according to \eqref{42} and \eqref{45}-\eqref{48}, we obtain the estimate of $f_{k}$ as follows:
\begin{align} \label{54}
&\left| f_{k} \right|\le \left| \psi_{k} \right|\varepsilon+\left| \varphi_{k} \right|\varepsilon+2\mu_{k}^2 \left| \varphi_{k} \right|\Gamma(\gamma_{1})T_{1}^{\beta_{1}}\varepsilon C_{2}+\mu_{k}^4  \left| \varphi_{k} \right|\Gamma^{2}\left(\gamma_{1} \right)T_{1}^{2\beta _{1} }\varepsilon C_{2}^{2} + \mu_{k}^2\left| \varphi_{k} \right|\Gamma(\gamma_{2})T_{1}^{\beta_{1}}\left(\xi-T_{1} \right)^{\beta _{2}}\times\nonumber\\
&\times \varepsilon C_{4}+2\mu_{k}^4 \left| \varphi_{k} \right|\Gamma\left(\gamma_{1} \right)\Gamma\left(\gamma_{2} \right)\left(\xi-T_{1} \right)^{\beta _{2}}T_{1}^{\beta _{1} }\varepsilon C_{2}C_{4}+\mu_{k}^6 \left| \varphi_{k} \right|\Gamma^{2}\left(\gamma_{1} \right)\Gamma\left(\gamma_{2} \right)\left(\xi-T_{1} \right)^{\beta _{2}}T_{1}^{2\beta _{1} }\varepsilon C_{2}^{2}C_{4}+\nonumber\\
&+\mu_{k}^2 \left| \psi_{k} \right|\varepsilon T_{1}+ \mu_{k}^4 \left| \psi_{k} \right| T_{1}^{\beta _{1}+1 } \varepsilon C_{3}+\mu_{k}^2 \left| \psi_{k} \right|\Gamma(\gamma_{1})T_{1}^{\beta_{1}}\varepsilon C_{2}+\mu_{k}^2\left| \varphi_{k}
 \right|\Gamma(\gamma_{1})T_{1}^{\beta_{1}}\varepsilon C_{1}+\mu_{k}^4 \left| \varphi_{k} \right|\Gamma^{2}(\gamma_{1})T_{1}^{2\beta_{1}}\times\nonumber\\
& \times\varepsilon C_{1}C_{3}+\mu_{k}^4\left| \varphi_{k} \right|\Gamma\left(\gamma_{1} \right)\Gamma\left(\gamma_{2} \right)\left(\xi-T_{1} \right)^{\beta _{2}}T_{1}^{\beta_{1}}\varepsilon C_{1}C_{4}+\mu_{k}^6  \left| \varphi_{k} \right| \Gamma^{2}\left(\gamma_{1} \right)\Gamma\left(\gamma_{2} \right)T_{1}^{2\beta_{1}}\varepsilon C_{1}C_{3}C_{4}.
\end{align}

Let us rewrite \eqref{54} by introducing the following notations:
\begin{align} \label{55}
&\left| f_{k} \right|\le \left| \varphi_{k} \right|\varepsilon+\mu_{k}^2 \left| \varphi_{k} \right|S_{5}+\mu_{k}^4 \left| \varphi_{k} \right|\left( S_{3}+S_{6} \right)+\mu_{k}^6\left| \varphi_{k} \right|S_{7} + \left| \psi_{k}  \right|\varepsilon+\mu_{k}^2\left| \psi_{k}  \right|\left( \varepsilon T_{1}+S_{4} \right).
\end{align}
where,
\begin{align*}
    & S_{5}=2T_{1}^{\beta_{1}} \Gamma(\gamma_{1})\varepsilon C_{2}+\Gamma(\gamma_{2})\left(\xi-T_{1} \right)^{\beta _{2}} \varepsilon C_{4}+T_{1}^{\beta_{1}} \Gamma(\gamma_{1})\varepsilon C_{1}, \\
& S_{6}=T_{1}^{2\beta_{1}} \Gamma^{2}(\gamma_{1})\varepsilon C_{2}^{2}+T_{1}^{2\beta_{1}} \Gamma^{2}(\gamma_{1})\varepsilon C_{1}C_{3}+\Gamma(\gamma_{1})\Gamma(\gamma_{2})\left(\xi-T_{1} \right)^{\beta _{2}}T_{1}^{\beta_{1}} \varepsilon C_{1} C_{4},\\
& S_{7}=\Gamma^{2}(\gamma_{1})\Gamma(\gamma_{2})\left(\xi-T_{1} \right)^{\beta _{2}}T_{1}^{2\beta_{1}}\varepsilon C_{2}^{2}C_{4}+\Gamma^{2}(\gamma_{1})\Gamma(\gamma_{2})T_{1}^{2\beta_{1}}\varepsilon C_{1}C_{3}C_{4}
\end{align*}
Substituting \eqref{53} and \eqref{55} into \eqref{51}, we obtain the estimate ${}_{1}U_{k} (t) $ taking the following form
\begin{align*}
   &\left| {}_{1}U_{k} (t) \right|\le  \left| \varphi_{{k}} \right|+\left| \varphi_{{k}} \right|tS_{1}+\mu_{k}^2\left| \varphi_{k} \right| t  
 \left( S_{2}+S_{4} \right)+\mu_{k}^4\left| \varphi_{k} \right| t S_{3}+\left| \psi_{k} \right|t S_{1}+\mu_{k}^2\left| \psi_{k} \right|t S_{4}+\mu_{k}^2\left| \varphi_{k} \right|t^{\beta_{1}}\times\\&\times\Gamma(\gamma_{1})C_2+\mu_{k}^2\left| \varphi_{k} \right|t^{\beta_{1}+1}\Gamma(\gamma_{1})C_3S_1+\mu_{k}^4\left| \varphi_{k} \right|  t^{\beta_{1}+1}\Gamma(\gamma_{1})C_3\left( S_2+S_{4} \right)+\mu_{k}^6\left| \varphi_{k} \right| t^{\beta_{1}+1}\Gamma(\gamma_{1})C_3 S_{3}+\\
&+ \mu_{k}^2\left| \psi_{k} \right| t^{\beta_{1}+1}\Gamma(\gamma_{1})C_3S_1+\mu_{k}^4\left| \psi_{k} \right|  t^{\beta_{1}+1}\Gamma(\gamma_{1})C_3 S_{4}+\left| \varphi_{{k}} \right|\Gamma(\gamma_{1})t^{\beta_{1}}\varepsilon C_2+\mu_{k}^2\left| \varphi_{{k}} \right|\Gamma(\gamma_{1})t^{\beta_{1}} S_{5}C_2+\\&+\mu_{k}^4\left| \varphi_{k} \right|  t^{\beta_{1}}\Gamma(\gamma_{1})C_2\left( S_3+S_{6} \right)+\mu_{k}^6 \left| \varphi_{k} \right| t^{\beta_{1}}\Gamma(\gamma_{1})C_2 S_{7}+\left| \psi_{{k}} \right|\Gamma(\gamma_{1})t^{\beta_{1}}\varepsilon C_2+\mu_{k}^2\left| \psi_{k} \right|t^{\beta_{1}}\Gamma(\gamma_{1}) C_2\times\\
&\times\left( \varepsilon T_{1}+S_{4} \right).
\end{align*}
    After simplification, we have the following
\begin{align}
  &\sum_{k=1}^{\infty}\left| {}_{1}U_{k} (t) \right|\le  \sum_{k=1}^{\infty}\left| \varphi_{{k}} \right|\Big(1+tS_{1}+\Gamma(\gamma_{1})t^{\beta_{1}}\varepsilon C_2\Big)+\sum_{k=1}^{\infty}\mu_{k}^2\left| \varphi_{k} \right| \Big( t  
 \left( S_{2}+S_{4} \right)+\Gamma(\gamma_{1}) t^{\beta_{1}} C_2+\Gamma(\gamma_{1})\times\nonumber\\& \times t^{\beta_{1}+1} C_1C_3+\Gamma(\gamma_{1}) t^{\beta_{1}} C_2S_{5}\Big)+\sum_{k=1}^{\infty}\mu_{k}^4\left| \varphi_{k} \right| \Big(t S_{3}+ t^{\beta_{1}+1} \Gamma(\gamma_{1})C_3\left( S_2+S_{4} \right)+\Gamma(\gamma_{1}) t^{\beta_{1}}C_{2}\left( S_{3}+S_{6} \right)\Big)+\nonumber\\&+\sum_{k=1}^{\infty}\mu_{k}^6\left| \varphi_{k}  \right|\Big(\Gamma(\gamma_{1}) t^{\beta_{1}+1} S_3C_{3}+\Gamma(\gamma_{1}) t^{\beta_{1}} S_7C_{2}\Big)+\sum_{k=1}^{\infty}\left| \psi_{k} \right|\Big(tS_{1}+\Gamma(\gamma_{1}) t^{\beta_{1}} \varepsilon C_{2}\Big)+\sum_{k=1}^{\infty}\mu_{k}^2\left| \psi_{k} \right|\times\nonumber\\& \times\Big(tS_{4}+ t^{\beta_{1}+1} \Gamma(\gamma_{1})S_1C_{3}+\Gamma(\gamma_{1}) t^{\beta_{1}}  C_{2}(\varepsilon T_{1}+S_4)\Big)+\sum_{k=1}^{\infty}\mu_{k}^4\left| \psi_{k} \right|\Gamma(\gamma_{1}) t^{\beta_{1}+1} S_4C_{3}.
\end{align}
From the estimates in Lemma 5.1 , we obtain the following estimate for ${}_{1}U_{k} (t) $
\begin{align*}
    &\sum_{k=1}^{\infty}\left| {}_{1}U_{k} (t) \right|\le  \sum_{k=1}^{\infty}\frac{C}{\mu_{k}^{2s-0,5}}\Big(1+2tS_{1}+2\Gamma(\gamma_{1})t^{\beta_{1}}\varepsilon C_2\Big)+\sum_{k=1}^{\infty}\frac{C}{\mu_{k}^{2s-2,5}} \Big( t  
 \left( S_{2}+2S_{4} \right)+\Gamma(\gamma_{1}) t^{\beta_{1}} C_2+\nonumber\\
&+\Gamma(\gamma_{1})t^{\beta_{1}+1} \left( C_1+S_{1} \right)C_3+\Gamma(\gamma_{1}) t^{\beta_{1}}  C_{2}(\varepsilon T_{1}+S_4+S_5)\Big)+\sum_{k=1}^{\infty}\frac{C}{\mu_{k}^{2s-4,5}} \Big(t S_{3}+ t^{\beta_{1}+1} \Gamma(\gamma_{1})C_3\times\\
&\times\left( S_2+2S_{4} \right)
+\Gamma(\gamma_{1}) t^{\beta_{1}}C_{2}\left( S_{3}+S_{6} \right)\Big)+\sum_{k=1}^{\infty}\frac{C}{\mu_{k}^{2s-6,5}}\Big(\Gamma(\gamma_{1}) t^{\beta_{1}+1} S_3C_{3}+\Gamma(\gamma_{1}) t^{\beta_{1}} S_7C_{2}\Big).
\end{align*}

Now we will demonstrate the convergence of the solution ${}_{2}U_{k} (t) $
$$\left| u(t,r) \right|\le \sum_{k=1}^{\infty}\left| {}_{2}U_{k}(t)J_{0}(\lambda_{k}r) \right|\le \sum_{k=1}^{\infty}\left| {}_{2}U_{k}(t)\right|,\,\,\,\,(\left| J_{0}(x) \le 1 \right|, \forall x \in \mathbb{R})$$
\begin{align} \label{57}
 &\left| {}_{2}U_{k} \left(t\right) \right| \le  \left| {}_{3} A_{k} \right| +\mu_{k}^2\left| {}_{3} A_{k} \right| \left(t-T_{1} \right)^{\beta _{2} } \Gamma\left(\gamma_{2} \right)  
\left| E_{2} \left(\left. \begin{array}{l} {\gamma_{2} ,\gamma_{2} ,1;1,0} \\ {\beta _{2} +1,\beta _{2} ,\alpha _{2};\gamma_{2} ,\gamma_{2} ;1,1} \end{array}\right|\begin{array}{c} {-\mu_{k}^2 \left(t-T_{1} \right)^{\beta _{2} } } \\ {\delta \left(t-T_{1} \right)^{\alpha_{2} } } \end{array}\right) \right|+\nonumber\\
& +\Gamma\left(\gamma_{2} \right) \left| f_{k} \right|
 \left(t-T_{1} \right)^{\beta _{2}} \left| E_{2} \left(\left. \begin{array}{l} {\gamma_{2} ,\gamma_{2} ,1;1,0} \\ {\beta _{2} +1,\beta _{2} ,\alpha_{2} ;\gamma_{2} ,\gamma_{2} ;1,1} \end{array}\right|\begin{array}{c} {-\mu_{k}^2 \left(t-T_{1} \right)^{\beta _{2} } } \\ {\delta \left(t-T_{1} \right)^{\alpha_{2} } } \end{array}\right) \right|.    
\end{align}
Using the estimate \eqref{48}, we rewrite \eqref{57} as follows.
\begin{equation} \label{58}
    \left| {}_{2}U_{k} \left(t\right) \right| \le  \left| {}_{3} A_{k} \right| +\mu_{k}^2\left| {}_{3} A_{k} \right| \Gamma\left(\gamma_{2} \right)  
\left(t-T_{1} \right)^{\beta _{2} }C_{4}  +\Gamma\left(\gamma_{2} \right) \left| f_{k} \right|
 \left(t-T_{1} \right)^{\beta _{2}}C_{4}.
\end{equation}
To demonstrate the uniform convergence of the solution ${}_{2}U_{k} \left(t\right)$ , we need to find an estimate of ${}_{3} A_{k}$, which we determine as follows. For this, we use estimates \eqref{45}-\eqref{48} and $\frac{1}{\left| \widetilde{\Delta}_{k} \right|}< \varepsilon $
\begin{align*}
    &\left| {}_{3}A_{k} \right|\le \left| \psi_{k}  \right|\varepsilon T_{1}^{\beta_{1}} \Gamma(\gamma_{1})\varepsilon C_{2}+\mu_{k}^2\left| \psi_{k}  \right|\varepsilon T_{1}^{2\beta_{1}} \Gamma^{2}(\gamma_{1})  C_{2}^{2}+\mu_{k}^2 \left| \varphi_{k} \right|\Gamma\left(\gamma_{1} \right)\Gamma\left(\gamma_{2} \right) \left(\xi-T_{1} \right)^{\beta _{2}}T_{1}^{\beta _{1} } C_{1}C_{4} \varepsilon+\\
&+\mu_{k}^4\left| \varphi_{k} \right|\Gamma^{2}\left(\gamma_{1} \right)\Gamma\left(\gamma_{2} \right)\left(\xi-T_{1} \right)^{\beta _{2}}T_{1}^{2\beta _{1} } C_{1}C_{3}C_{4} \varepsilon+\left| \varphi_{k} \right|\Gamma\left(\gamma_{2} \right)C_{4}\varepsilon\left(\xi-T_{1} \right)^{\beta _{2}}+2\mu_{k}^2\left| \varphi_{k} \right|\Gamma\left(\gamma_{1} \right)\Gamma\left(\gamma_{2} \right)\times\\
& \times T_{1}^{\beta _{1} }\left(\xi-T_{1} \right)^{\beta _{2}}C_{2}C_{4} \varepsilon+\mu_{k}^4 \left| \varphi_{k} \right|\Gamma^{2}\left(\gamma_{1} \right)\Gamma\left(\gamma_{2} \right)\left(\xi-T_{1} \right)^{\beta _{2}}T_{1}^{2\beta _{1} } C_{2}^{2}C_{4} \varepsilon+\left| \psi_{k}  \right|\varepsilon T_{1}^{\beta_{1}} \Gamma(\gamma_{1})\varepsilon C_{1}+\mu_{k}^2\left| \psi_{k}  \right|\varepsilon \times\\& \times T_{1}^{2\beta_{1}} \Gamma^{2}(\gamma_{1})  C_{1}C_{3}+\left| \psi_{k}  \right|\varepsilon T_{1}+\mu_{k}^2\left| \psi_{k}  \right|\varepsilon T_{1}^{\beta_{1}+1} \Gamma(\gamma_{1})C_{3}. 
\end{align*}
As a result of the simplifications, we obtain the following
\begin{align} \label{59}
 &\left| {}_{3}A_{k} \right|\le \left| \varphi_{k} \right|\Gamma\left(\gamma_{2} \right)\left(\xi-T_{1} \right)^{\beta _{2}}\varepsilon C_{4}+\mu_{k}^2 \left| \varphi_{k} \right|\Bigg(\Gamma\left(\gamma_{1} \right)\Gamma\left(\gamma_{2} \right)\left(\xi-T_{1} \right)^{\beta _{2}}T_{1}^{\beta _{1} } C_{1}C_{4} \varepsilon+2\Gamma\left(\gamma_{1} \right)\Gamma\left(\gamma_{2} \right)\times\nonumber\\
&\times\left(\xi-T_{1} \right)^{\beta _{2}}T_{1}^{\beta _{1} } C_{2}C_{4} \varepsilon \Bigg)+\mu_{k}^4\left| \varphi_{k} \right|\Bigg(\Gamma^{2}\left(\gamma_{1} \right)\Gamma\left(\gamma_{2} \right) \left(\xi-T_{1} \right)^{\beta _{2}}T_{1}^{2\beta _{1} } C_{1}C_{3}C_{4} \varepsilon +\Gamma^{2}\left(\gamma_{1} \right)\left(\xi-T_{1} \right)^{\beta _{2}}\times\nonumber\\
&\times \Gamma\left(\gamma_{2} \right)T_{1}^{2\beta _{1} }\varepsilon C_{2}^{2}C_{4}\Bigg)+\left| \psi_{k} \right|\Bigg( T_{1}^{\beta _{1} } \Gamma\left(\gamma_{1} \right) C_{2}\varepsilon+T_{1}^{\beta _{1} } \Gamma\left(\gamma_{1} \right) C_{1}\varepsilon+T_{1} \varepsilon \Bigg)+ \mu_{k}^2 \left| \psi_{k} \right| \Bigg(T_{1}^{2\beta _{1} } \Gamma^{2}\left(\gamma_{1} \right) C_{2}^{2}\varepsilon+\nonumber\\
&\times T_{1}^{2\beta _{1} } \Gamma^{2}\left(\gamma_{1} \right) C_{1}C_{3}\varepsilon+T_{1}^{\beta _{1}+1 } \Gamma\left(\gamma_{1} \right)C_{3}\varepsilon \Bigg) .  
\end{align}
We will introduce the following notations:
\begin{align*}
    & S_{8}=\Gamma^{2}\left(\gamma_{1} \right)\Gamma\left(\gamma_{2} \right)\left(\xi-T_{1} \right)^{\beta _{2}}T_{1}^{2\beta _{1} } C_{4} \varepsilon \Big(C_{1}C_{3}+C_{2}^{2}\Big), \\
& S_{9}=T_{1}^{2\beta _{1} } \Gamma^{2}\left(\gamma_{1} \right) C_{2}^{2}\varepsilon+T_{1}^{2\beta _{1} } \Gamma^{2}\left(\gamma_{1} \right)C_{1} C_{3}\varepsilon+T_{1}^{\beta _{1}+1 } \Gamma\left(\gamma_{1} \right) C_{3}\varepsilon.
\end{align*}
Based on these notations, we will rewrite \eqref{59}
\begin{align} \label{60}
  &\left| {}_{3}A_{k} \right|\le \left| \varphi_{k} \right|\frac{S_{2}}{2}+\mu_{k}^2 \left| \varphi_{k} \right|\frac{S_{3}(C_1+2C_2)}{2C_{2}}+\mu_{k}^4 \left| \varphi_{k} \right|S_{8}+\left| \psi_{k} \right|(S_{4}+S_{1}T_{1})+\mu_{k}^2\left| \psi_{k} \right|S_{9}.   
\end{align}
We substitute \eqref{60} into \eqref{58} and obtain the following estimate
\begin{align*}
     &\sum_{k=1}^{\infty}\left| {}_{2}U_{k} (t) \right| \le \sum_{k=1}^{\infty}\left| \varphi_{k} \right| \Bigg(\frac{S_{2}}{2}+\Gamma\left(\gamma_{2} \right)\left(t-T_{1} \right)^{\beta _{2}} \varepsilon C_{4}\Bigg)+\sum_{k=1}^{\infty}\mu_{k}^2 \left| \varphi_{k} \right| \Bigg(\frac{S_{3}(C_1+2C_2)}{2C_{2}}+\left(t-T_{1} \right)^{\beta _{2}}\times\nonumber\\
&\times\Gamma\left(\gamma_{2} \right)\frac{S_{2}C_{4}}{2}+\Gamma\left(\gamma_{2} \right)\left(t-T_{1} \right)^{\beta _{2}}C_{4}S_{5}\Bigg)+\sum_{k=1}^{\infty}\mu_{k}^4\left| \varphi_{k} \right|\Bigg(S_{8}+\Gamma\left(\gamma_{2} \right)\left(t-T_{1} \right)^{\beta _{2}}C_{4}\frac{S_{3}(C_1+2C_2)}{2C_{2}}+\nonumber\\
&+\Gamma\left(\gamma_{2} \right)\left(t-T_{1} \right)^{\beta_{2}}C_{4}\Big(S_{3}+S_{6}\Big)\Bigg)+\sum_{k=1}^{\infty}\mu_{k}^6\left| \varphi_{k} \right|\Bigg(\Gamma\left(\gamma_{2} \right)\left(t-T_{1} \right)^{\beta _{2}}S_{8}C_{4}+\Gamma\left(\gamma_{2} \right)\left(t-T_{1} \right)^{\beta _{2}}S_{7}C_{4}\Bigg)+\nonumber\\&+\sum_{k=1}^{\infty}\left| \psi_{k} \right|\Bigg((S_{4}+S_{1}T_{1})+\Gamma\left(\gamma_{2} \right)\left(t-T_{1} \right)^{\beta _{2}} \varepsilon C_{4} \Bigg)+\sum_{k=1}^{\infty}\mu_{k}^2\left| \psi_{k} \right|\Bigg(S_{9}+\Gamma\left(\gamma_{2} \right)\left(t-T_{1} \right)^{\beta _{2}} C_{4}(S_{4}+S_{1}T_{1})+\nonumber\\
&+\Gamma\left(\gamma_{2} \right)\left(t-T_{1} \right)^{\beta _{2}} C_{4} \Big(\varepsilon T_{1}+S_{4}\Big)\Bigg)+\sum_{k=1}^{\infty}\mu_{k}^4\left| \psi_{k} \right|\Gamma\left(\gamma_{2} \right)\left(t-T_{1} \right)^{\beta _{2}} C_{4}S_{9}.
\end{align*}
From the estimates in Lemma 5.1 , we obtain the following estimate for ${}_{2}U_{k} (t) $
\begin{align*}
&\sum_{k=1}^{\infty}\left| {}_{2}U_{k} (t) \right| \le \sum_{k=1}^{\infty}\frac{C}{\mu_{k}^{2s-0,5}} \Bigg(\frac{S_{2}}{2}+2\Gamma\left(\gamma_{2} \right)\left(t-T_{1} \right)^{\beta _{2}} \varepsilon C_{4}(S_{4}+S_{1}T_{1})\Bigg)+\sum_{k=1}^{\infty}\frac{C}{\mu_{k}^{2s-2,5}} \Bigg(\frac{S_{3}(C_1+2C_2)}{2C_{2}}+\\
&+S_{9}+\Gamma\left(\gamma_{2} \right)\left(t-T_{1} \right)^{\beta _{2}} C_{4}(2S_{4}+S_{1}T_{1}+\frac{S_{2}}{2}+\varepsilon T_{1}+S_{5})\Bigg)+\sum_{k=1}^{\infty}\frac{C}{\mu_{k}^{2s-4,5}}\Bigg(S_{8}+\Gamma\left(\gamma_{2} \right)\left(t-T_{1} \right)^{\beta _{2}}C_{4}\times \\
&\times\Bigg(\frac{S_{3}(C_1+2C_2)}{2C_{2}}+S_{3}+S_{6}+S_{9}\Bigg)\Bigg)+\sum_{k=1}^{\infty}\frac{C}{\mu_{k}^{2s-6,5}}\Bigg(\Gamma\left(\gamma_{2} \right)\left(t-T_{1} \right)^{\beta _{2}}S_{8}C_{4}+\Gamma\left(\gamma_{2} \right)\left(t-T_{1} \right)^{\beta _{2}}S_{7}C_{4}\Bigg).
\end{align*}

Now we will demonstrate that a solution ${}_{3}U_{k} (t)$, similar to the one above, converges uniformly.
$$\left| u(t,r) \right|\le \sum_{k=1}^{\infty}\left| {}_{3}U_{k}(t)J_{0}(\lambda_{k}r) \right|\le \sum_{k=1}^{\infty}\left| {}_{3}U_{k}(t)\right|, \,\,\,\,(\left| J_{0}(x) \le 1 \right|, \forall x \in \mathbb{R})$$
\begin{align*}
    &\left| {}_{3}U_{k} (t) \right| \le \left| {}_{4} A_{k} \right| +\left| {}_{5} A_{k} \right| \left(t-T_{2} \right)+\mu_{k}^2\left| {}_{4} A_{k} \right|\Gamma\left(\gamma_{3} \right)    
 \left| E_{2} \left(\left. \begin{array}{l} {\gamma_{3} ,\gamma_{3} ,1;1,0} \\ {\beta _{3} +1,\beta _{3} ,\alpha_{3} ;\gamma_{3} ,\gamma_{3} ;1,1} \end{array}\right|\begin{array}{c} -\mu_{k}^2\left(t-T_{2} \right)^{\beta _{3}  } \\ {\delta \left(t-T_{2} \right)^{\alpha_{3} } } \end{array}\right) \right|\times\nonumber\\
&\times    \left(t-T_{2} \right)^{\beta _{3} }  + \mu_{k}^2\left| {}_{5} A_{k} \right|   \left(t-T_{2} \right)^{\beta _{3}+1 }\Gamma (\gamma_{3})\left| E_{2} \left(\left. \begin{array}{l} {\gamma_{3} ,\gamma_{3} ,1;1,0} \\ {\beta _{3} +2,\beta _{3} ,\alpha_{3};\gamma_{3} ,\gamma_{3};1,1} \end{array}\right|\begin{array}{c} -\mu_{k}^2 \left(t-T_{2} \right)^{\beta _{3}  } \\ {\delta \left(t-T_{2} \right)^{\alpha_{3} } } \end{array}\right) \right|\times\nonumber\\
& +\Gamma (\gamma_{3})\left(t-T_{2} \right)^{\beta _{3} } \left| f_{k } \right|\left| E_{2} \left(\left. \begin{array}{l} {\gamma_{3} ,\gamma_{3} ,1;1,0} \\ {\beta _{3} +1,\beta _{3} ,\alpha_{3} ;\gamma_{3} ,\gamma_{3} ;1,1} \end{array}\right|\begin{array}{c} -\mu_{k}^2\left(t-T_{2} \right)^{\beta _{3}  } \\ {\delta \left(t-T_{2} \right)^{\alpha_{3} } } \end{array}\right) \right|.
\end{align*}
Based on estimates \eqref{45} through \eqref{40}, we can write the following
\begin{align} \label{61}
    &\left| {}_{3}U_{k} (t) \right|\le \left| {}_{4} A_{k} \right| +\left| {}_{5} A_{k} \right| \left(t-T_{2} \right)+\mu_{k}^2\left| {}_{4} A_{k} \right|    \left(t-T_{2} \right)^{\beta _{3} }  \Gamma\left(\gamma_{3} \right)C_{5}+ \mu_{k}^2\left| {}_{5} A_{k} \right|  \left(t-T_{2} \right)^{\beta _{3}+1 }\Gamma (\gamma_{3})C_{6}+\nonumber\\
&+\Gamma (\gamma_{3})\left(t-T_{2} \right)^{\beta _{3} } \left| f_{k } \right|C_{5}.
\end{align}
We also calculate the estimates of the coefficients ${}_{4}A_{k}$ and ${}_{5}A_{k}$. First, using estimates \eqref{45}-\eqref{48}, we determine the estimate of coefficient ${}_{4}A_{k}$ as follows:
\begin{align} \label{62}
  &\left| {}_{4}A_{k} \right|\le \left| \varphi_{k} \right|\Big(\Gamma\left(\gamma_{2} \right)\left(\xi-T_{1} \right)^{\beta _{2}}\varepsilon C_{4}+\Gamma\left(\gamma_{2} \right)\left(T_{2}-T_{1} \right)^{\beta _{2}}\varepsilon C_{4}\Big)+\mu_{k}^2\left| \varphi_{k} \right|\Big(\Gamma\left(\gamma_{1} \right)\Gamma\left(\gamma_{2} \right)T_{1}^{\beta _{1} } \times\nonumber \\
& \times \left(\xi-T_{1} \right)^{\beta _{2}}\varepsilon C_{1}C_{4}+2\Gamma\left(\gamma_{1} \right)\Gamma\left(\gamma_{2} \right)\left(\xi-T_{1} \right)^{\beta _{2}}T_{1}^{\beta _{1} } \varepsilon C_{2}C_{4}+2\Gamma\left(\gamma_{1} \right)\Gamma\left(\gamma_{2} \right)\left(T_{2}-T_{1} \right)^{\beta _{2}}T_{1}^{\beta _{1} } \varepsilon C_{2}C_{4}+\nonumber\\
& +\Gamma\left(\gamma_{1} \right)\Gamma\left(\gamma_{2} \right)\left(T_{2}-T_{1} \right)^{\beta _{2}}T_{1}^{\beta _{1} } \varepsilon C_{1}C_{4}\Big)+\mu_{k}^4\left| \varphi_{k} \right|\Big(\Gamma^{2}\left(\gamma_{1} \right)\Gamma\left(\gamma_{2} \right)T_{1}^{2\beta _{1} }\left(\xi-T_{1} \right)^{\beta _{2}}\varepsilon C_{2}^{2}C_{4}+\nonumber\\
&+ \Gamma^{2}\left(\gamma_{1} \right)\Gamma\left(\gamma_{2} \right)T_{1}^{\beta _{1} }\left(\xi-T_{1} \right)^{\beta _{2}}\varepsilon C_{1}C_{3}C_{4}+\Gamma^{2}\left(\gamma_{1} \right)\Gamma\left(\gamma_{2} \right)T_{1}^{2\beta _{1} }\left(T_{2}-T_{1} \right)^{\beta _{2}}\varepsilon C_{2}^{2}C_{4}+\Gamma^{2}\left(\gamma_{1} \right)\Gamma\left(\gamma_{2} \right)\times\nonumber \\
& \times T_{1}^{2\beta _{1} }\left(T_{2}-T_{1} \right)^{\beta _{2}}\varepsilon C_{1}C_{3}C_{4}\Big)+\left| \psi_{k} \right|\Big(T_{1}^{\beta _{1} }\Gamma\left(\gamma_{1} \right)\varepsilon C_{2}+T_{1}^{\beta _{1} }\Gamma\left(\gamma_{1} \right)\varepsilon C_{1}+T_{1}\varepsilon+\Gamma\left(\gamma_{2} \right)\left(T_{2}-T_{1} \right)^{\beta _{2}}\times\nonumber \\
&\times\varepsilon C_{4}\Big)+\mu_{k}^2\left| \psi_{k} \right|\Big(T_{1}^{\beta _{1}+1 }\Gamma\left(\gamma_{1} \right)\varepsilon C_{3}+T_{1}^{2\beta _{1} }\Gamma^{2}\left(\gamma_{1} \right)\varepsilon C_{1}C_{3}+2\Gamma\left(\gamma_{1} \right)\Gamma\left(\gamma_{2} \right)\left(T_{2}-T_{1} \right)^{\beta _{2}}\times\nonumber\\
&\times T_{1}^{\beta _{1} } \varepsilon C_{2}C_{4}+T_{1}^{2\beta _{1} }\Gamma^{2}\left(\gamma_{1} \right)\varepsilon C_{2}^{2}+\Gamma\left(\gamma_{1} \right)\Gamma\left(\gamma_{2} \right)\left(T_{2}-T_{1} \right)^{\beta _{2}}T_{1}^{\beta _{1} } \varepsilon C_{1}C_{4}\Big)+\mu_{k}^4\left| \psi_{k} \right|\times\nonumber\\
& \times \Big(\Gamma^{2}\left(\gamma_{1} \right)\Gamma\left(\gamma_{2} \right)T_{1}^{2\beta _{1} }\left(T_{2}-T_{1} \right)^{\beta _{2}}\varepsilon C_{2}^{2}C_{4}+\Gamma^{2}\left(\gamma_{1} \right)\Gamma\left(\gamma_{2} \right)T_{1}^{2\beta _{1} }\left(T_{2}-T_{1} \right)^{\beta _{2}}\varepsilon C_{1}C_{3}C_{4}\Big). 
\end{align}
Let us introduce the following notations:
\begin{align*}
    & S_{10}=\Gamma\left(\gamma_{2} \right)\left(T_{2}-T_{1} \right)^{\beta _{2}}\varepsilon C_{4},\\
& S_{11}=\Gamma\left(\gamma_{1} \right)\Gamma\left(\gamma_{2} \right)T_{1}^{\beta _{1} } \varepsilon C_{4}\left( 2C_{2}+C_{1} \right)\left( \left(\xi-T_{1} \right)^{\beta _{2}}+\left(T_{2}-T_{1} \right)^{\beta _{2}} \right),\\
& S_{12}=T_{1}^{2\beta _{1} }\Gamma^{2}\left(\gamma_{1} \right)\Gamma\left(\gamma_{2} \right)\left(T_{2}-T_{1} \right)^{\beta _{2}}\varepsilon \left( C_{2}^{2}C_{4}+C_{1}C_{3}C_{4}\right),\\
& S_{13}=T_{1}^{\beta _{1}+1 }\Gamma\left(\gamma_{1} \right)\varepsilon C_{3}+T_{1}^{2\beta _{1} }\Gamma^{2}\left(\gamma_{1} \right)\varepsilon C_{1}C_{3}+2\Gamma\left(\gamma_{1} \right)\Gamma\left(\gamma_{2} \right)T_{1}^{\beta _{1} }\left(T_{2}-T_{1} \right)^{\beta _{2}} \varepsilon C_{2} C_{4}+\\
&\Gamma\left(\gamma_{1} \right)\Gamma\left(\gamma_{2} \right) T_{1}^{\beta _{1} }\left(T_{2}-T_{1} \right)^{\beta _{2}} \varepsilon C_{1} C_{4}+ T_{1}^{2\beta _{1} }\Gamma^{2}\left(\gamma_{1} \right)\varepsilon C_{2}^{2}.
\end{align*}
We rewrite \eqref{62} according to the corresponding notations
\begin{align} \label{63}
  &\left| {}_{4}A_{k} \right|\le \left| \varphi_{k} \right|\left( \frac{S_{2}}{2}+S_{10} \right)+\mu_{k}^2\left| \varphi_{k} \right|S_{11}+\mu_{k}^4\left| \varphi_{k} \right|\left( S_{7}+S_{12} \right)+ \left| \psi_{k} \right|\left( S_{4}+T_{1}S_{1}+S_{10} \right)+\nonumber\\
&+\mu_{k}^2\left| \psi_{k} \right|S_{13}+\mu_{k}^4\left| \psi_{k} \right|S_{12}.  
\end{align}
By substituting the estimates \eqref{55} and \eqref{63} into \eqref{26}, we determine the estimate of ${}_{5}A_{k}$ as follows
\begin{align} \label{64}
 &\left| {}_{5}A_{k} \right| \le \left| \varphi_{k} \right| \varepsilon+\mu_{k}^2\left| \varphi_{k} \right|\left( \frac{S_{2}}{2}+S_{5}+S_{10} \right)+\mu_{k}^4\left| \varphi_{k} \right|\left( S_{3}+S_{6}+S_{11} \right)+\mu_{k}^6\left| \varphi_{k} \right|\left( 2S_{7}+S_{12} \right)+\nonumber\\
&+\left| \psi_{k} \right| \varepsilon+\mu_{k}^2\left| \psi_{k} \right|\left( 2S_{4}+T_{1}S_{1}+S_{10}+\varepsilon T_{1} \right)+\mu_{k}^4\left| \psi_{k} \right|S_{13}+\mu_{k}^6\left| \psi_{k} \right|S_{12}.   
\end{align}
Substituting the estimates \eqref{55},\eqref{63} and \eqref{64} into \eqref{61}, we obtain the following estimate for the solution ${}_{3}U_{k}(t)$
\begin{align*}
 &\sum_{k=1}^{\infty}\left| {}_{3}U_{k} (t) \right| \le \sum_{k=1}^{\infty}\left| \varphi_{k} \right| \Bigg(\frac{S_{2}}{2}+S_{10}+\left(t-T_{2} \right) \varepsilon+\Gamma\left(\gamma_{3} \right)\left(t-T_{2} \right)^{\beta _{3}} \varepsilon C_{5}\Bigg)+\sum_{k=1}^{\infty}\mu_{k}^2 \left| \varphi_{k} \right|\Bigg(S_{11}+\\
& +( t-T_{2})(\frac{S_{2}}{2}+S_{5}+S_{10})+\Gamma\left(\gamma_{3} \right)\left(t-T_{2} \right)^{\beta _{3}} \Big( (\frac{S_{2}}{2}+S_{10})C_{5}+\left(t-T_{2} \right) \varepsilon C_{6} +S_{5}C_{5}\Big )\Bigg)+\\
&+\sum_{k=1}^{\infty}\mu_{k}^4 \left| \varphi_{k} \right|\Bigg(S_{7}+S_{12}+( t-T_{2})(S_{3}+S_{6}+S_{11})+\Gamma\left(\gamma_{3} \right)\left(t-T_{2} \right)^{\beta _{3}}\Big(S_{11}C_{5}+\\
& +( t-T_{2})C_{6}(\frac{S_{2}}{2}+S_{10}+S_{5})+C_{5}(S_{3}+S_{6})\Big)\Bigg)+\sum_{k=1}^{\infty}\mu_{k}^6 \left| \varphi_{k} \right|\Bigg(( t-T_{2})(2S_{7}+S_{12})+\\
& +\Gamma\left(\gamma_{3} \right)\left(t-T_{2} \right)^{\beta _{3}}\Big(C_{5}(S_{7}+S_{12})+( t-T_{2})C_{6}(S_{3}+S_{6}+S_{11})+S_{7}C_{5}\Big)\Bigg)+\sum_{k=1}^{\infty}\mu_{k}^8 \left| \varphi_{k} \right|\times\\
& \times \Gamma\left(\gamma_{3} \right)\left(t-T_{2} \right)^{\beta _{3}+1}C_{6}(2S_{7}+S_{12})+\sum_{k=1}^{\infty}\left| \psi_{k} \right|\Bigg(S_{4}+T_{1}S_{1}+S_{10}+( t-T_{2})\varepsilon+\left(t-T_{2} \right)^{\beta _{3}}\times\\
& \times\Gamma\left(\gamma_{3} \right) C_{5} \varepsilon \Bigg)+\sum_{k=1}^{\infty}\mu_{k}^2\left| \psi_{k} \right|\Bigg(S_{13}+( t-T_{2})(2S_{4}+S_{1}T_{1}+S_{10}+\varepsilon T_{1})+\Gamma\left(\gamma_{3} \right)\left(t-T_{2} \right)^{\beta _{3}}\times\\
& \times\Big((S_{4}+S_{1}T_{1}+S_{10})C_{5}+( t-T_{2})\varepsilon C_{6}+C_{5}(\varepsilon T_{1}+S_{4})\Big)\Bigg)+\sum_{k=1}^{\infty}\mu_{k}^4\left| \psi_{k} \right|\Bigg(S_{12}+S_{13}\times\\
& \times( t-T_{2})+\Gamma\left(\gamma_{3} \right)\left(t-T_{2} \right)^{\beta _{3}}\Big(S_{13}C_{5}+( t-T_{2})\varepsilon C_{6}(2S_{4}+S_{1}T_{1}+S_{10}+\varepsilon T_{1})\Big)\Bigg)+\\
& +\sum_{k=1}^{\infty}\mu_{k}^6\left| \psi_{k} \right|\Bigg(( t-T_{2})S_{12}+\Gamma\left(\gamma_{3} \right)\left(t-T_{2} \right)^{\beta _{3}}\Big(C_{5}S_{12}+( t-T_{2})S_{13}C_{6}\Big)\Bigg)+\\
&+\sum_{k=1}^{\infty}\mu_{k}^8\left| \psi_{k} \right|\left(t-T_{2} \right)^{\beta _{3}+1}\Gamma\left(\gamma_{3} \right)S_{12}C_{6}.   
\end{align*}
From the estimates in Lemma 5.1 , we obtain the following estimate for ${}_{3}U_{k} (t) $
\begin{align*}
&\sum_{k=1}^{\infty}\left| {}_{3}U_{k} (t) \right| \le \sum_{k=1}^{\infty}\frac{C}{\mu_{k}^{2s-0,5}} \Bigg(\frac{S_{2}}{2}+2S_{10}+2\left(t-T_{2} \right) \varepsilon+2\Gamma\left(\gamma_{3} \right)\left(t-T_{2} \right)^{\beta _{3}} \varepsilon C_{5}+S_{4}+T_{1}S_{1})+\\
&+\sum_{k=1}^{\infty}\frac{C}{\mu_{k}^{2s-2,5}}\Bigg(S_{11}+\Gamma\left(\gamma_{3} \right)\left(t-T_{2} \right)^{\beta _{3}} \Big( (\frac{S_{2}}{2}+S_{10})C_{5}+S_{5}C_{5}+(S_{4}+S_{1}T_{1}+S_{10})C_{5}+C_{5}\times \\
&\times(\varepsilon T_{1}+S_{4})\Big)
+S_{13}+( t-T_{2})(2S_{4}+2 \varepsilon C_{6}+S_{1}T_{1}+2S_{10}+\varepsilon T_{1}+\frac{S_{2}}{2}+S_{5})\Bigg)+\sum_{k=1}^{\infty}\frac{C}{\mu_{k}^{2s-4,5}}\times\\
&\times\Bigg(S_{7}+2S_{12}+( t-T_{2})(S_{3}+S_{6}+S_{11}+S_{13})+\Gamma\left(\gamma_{3} \right)\left(t-T_{2} \right)^{\beta _{3}}\Big(S_{11}C_{5}+( t-T_{2})C_{6}(\frac{S_{2}}{2}+\\
&+S_{10}+S_{5})+C_{5}(S_{3}+S_{6})+S_{13}C_{5}+( t-T_{2})\varepsilon C_{6}(2S_{4}+S_{1}T_{1}+S_{10}+\varepsilon T_{1})\Big)\Bigg)+\sum_{k=1}^{\infty}\frac{C}{\mu_{k}^{2s-6,5}}\times\\
& \times\Bigg(2( t-T_{2})(S_{7}+S_{12})+\Gamma\left(\gamma_{3} \right)\left(t-T_{2} \right)^{\beta _{3}}\Big(C_{5}(S_{7}+S_{12})+( t-T_{2})C_{6}(S_{3}+S_{6}+S_{11})+S_{7}C_{5}+\\
&+C_{5}S_{12}+( t-T_{2})S_{13}C_{6}\Big)\Bigg)
+\sum_{k=1}^{\infty}\frac{C}{\mu_{k}^{2s-8,5}}2\Gamma\left(\gamma_{3} \right)\left(t-T_{2} \right)^{\beta _{3}+1}C_{6}(S_{7}+S_{12}).
\end{align*}

\begin{theorem}
If  $\alpha_2=1,\,\gamma_2=\beta_2$,\,$\widetilde\Delta_{k}\neq 0$, $\left\{  \varphi(r),\psi(r)\right\} \in  C^{10}[0;1]$ and the equalities 
\[
\varphi^{\left( i \right)}\left( 0 \right)=0,\psi^{\left( i \right)}\left( 0 \right)=0\, (i=\overline{1,9}),\]
\[\varphi^{\left( j \right)}\left( 1 \right)=0,\psi^{\left( j \right)}\left( 1 \right)=0\, (j=\overline{1,8})\] are satisfied. The unique solution to problem \eqref{1}-\eqref{6} exists in the form of the series \eqref{15},\eqref{16},\eqref{17}. 
\end{theorem}

\section{Conclusion}
This article investigates the inverse source problem for a mixed wave-diffusion-wave equation with a the Prabhakar-Caputo fractional-order derivative in a cylindrical domain. The method of separation of variables is applied. The solution is sought in the form of a Fourier-Bessel series, and its uniform convergence is proven.
\section{Appendix}
\subsection{A1: Simplification of $\widetilde\Delta$.}
By substituting the notations \eqref{28},\eqref{31},\eqref{32} and \eqref{34} into \eqref{37}, we find the value of $\widetilde{\Delta}_{k}$ as follows:
\begin{align*}
&\widetilde{\Delta}_{k}=\Gamma\left(\gamma_{2} \right)\left(\xi-T_{1} \right)^{\beta _{2}} E_{2} \left(\left. \begin{array}{l} {\gamma_{2} ,\gamma_{2} ,1;1,0} \\ {\beta _{2} +1,\beta _{2} ,\alpha_{2} ;\gamma_{2} ,\gamma_{2} ;1,1} \end{array}\right|\begin{array}{c} {\lambda_{k} \left(\xi-T_{1} \right)^{\beta _{2} } } \\ {\delta \left(\xi-T_{1} \right)^{\alpha_{2} } } \end{array}\right)+T_{1}+T_{1}^{\beta_{1}} \Gamma(\gamma_{1})\times\\
&\times E_{2} \left(\left. \begin{array}{l} {\gamma_{1} ,\gamma_{1} ,1;1,0} \\ {\beta _{1} +1,\beta _{1} ,\alpha _{1};\gamma_{1} ,\gamma_{1} ;1,1} \end{array}\right|\begin{array}{c} {\lambda_k T_{1}^{\beta _{1} } } \\ {\delta T_{1}^{\alpha_{1} } } \end{array}\right)+T_{1}^{\beta_{1}+1} \lambda_{k} \Gamma(\gamma_{1})E_{2} \left(\left. \begin{array}{l} {\gamma_{1} ,\gamma_{1} ,1;1,0} \\ {\beta _{1} +2,\beta _{1} ,\alpha_{1} ;\gamma_{1} ,\gamma_{1} ;1,1} \end{array}\right|\begin{array}{c} {\lambda_k T_{1}^{\beta _{1} } } \\ {\delta T_{1}^{\alpha_{1} } } \end{array}\right)+ \\
&+2 \lambda_{k}T_{1}^{\beta_{1}}\Gamma\left(\gamma_{1} \right)\Gamma\left(\gamma_{2} \right)\left(\xi-T_{1} \right)^{\beta _{2}} E_{2} \left(\left. \begin{array}{l} {\gamma_{2} ,\gamma_{2} ,1;1,0} \\ {\beta _{2} +1,\beta _{2} ,\alpha_{2} ;\gamma_{2} ,\gamma_{2} ;1,1} \end{array}\right|\begin{array}{c} {\lambda_{k} \left(\xi-T_{1} \right)^{\beta _{2} } } \\ {\delta \left(\xi-T_{1} \right)^{\alpha_{2} } } \end{array}\right)\times \\
&\times E_{2} \left(\left. \begin{array}{l} {\gamma_{1} ,\gamma_{1} ,1;1,0} \\ {\beta _{1} +1,\beta _{1} ,\alpha_{1} ;\gamma_{1} ,\gamma_{1} ;1,1} \end{array}\right|\begin{array}{c} {\lambda_k T_{1}^{\beta _{1} } } \\ {\delta T_{1}^{\alpha_{1} } } \end{array}\right)+
\Gamma^{2}\left(\gamma_{1} \right)\left( E_{2} \left(\left. \begin{array}{l} {\gamma_{1} ,\gamma_{1} ,1;1,0} \\ {\beta _{1} +1,\beta _{1} ,\alpha_{1} ;\gamma_{1} ,\gamma_{1} ;1,1} \end{array}\right|\begin{array}{c} {\lambda_k T_{1}^{\beta _{1} } } \\ {\delta T_{1}^{\alpha_{1} } } \end{array}\right) \right)^{2}\times\\
& \times \lambda_{k}T_{1}^{2\beta_{1}}+\lambda_{k}^{2} \cdot\Gamma\left(\gamma_{2} \right)\left(\xi-T_{1} \right)^{\beta _{2}} E_{2} \left(\left. \begin{array}{l} {\gamma_{2} ,\gamma_{2} ,1;1,0} \\ {\beta _{2} +1,\beta _{2} ,\alpha_{2} ;\gamma_{2} ,\gamma_{2} ;1,1} \end{array}\right|\begin{array}{c} {\lambda_{k} \left(\xi-T_{1} \right)^{\beta _{2} } } \\ {\delta \left(\xi-T_{1} \right)^{\alpha_{2} } } \end{array}\right)\cdot T_{1}^{2\beta_{1}}\Gamma^{2}\left(\gamma_{1} \right) \times \\
&\times\left( E_{2} \left(\left. \begin{array}{l} {\gamma_{1} ,\gamma_{1} ,1;1,0} \\ {\beta _{1} +1,\beta _{1} ,\alpha_{1} ;\gamma_{1} ,\gamma_{1} ;1,1} \end{array}\right|\begin{array}{c} {\lambda_k T_{1}^{\beta _{1} } } \\ {\delta T_{1}^{\alpha_{1} } } \end{array}\right) \right)^{2}+2\lambda_{k}\Gamma\left(\gamma_{2} \right)\left(\xi-T_{1} \right)^{\beta _{2}} \times  \\ 
& \times E_{2} \left(\left. \begin{array}{l} {\gamma_{2} ,\gamma_{2} ,1;1,0} \\ {\beta _{2} +1,\beta _{2} ,\alpha_{2} ;\gamma_{2} ,\gamma_{2} ;1,1} \end{array}\right|\begin{array}{c} {\lambda_{k} \left(\xi-T_{1} \right)^{\beta _{2} } } \\ {\delta \left(\xi-T_{1} \right)^{\alpha_{2} } } \end{array}\right)\Bigg(T_{1}+T_{1}^{\beta_{1}+1}\lambda_{k}\Gamma\left(\gamma_{1} \right)\times\\
&\times
E_{2} \left(\left. \begin{array}{l} {\gamma_{1} ,\gamma_{1} ,1;1,0} \\ {\beta _{1} +2,\beta _{1} ,\alpha_{1} ;\gamma_{1} ,\gamma_{1} ;1,1} \end{array}\right|\begin{array}{c} {\lambda_k T_{1}^{\beta _{1} } } \\ {\delta T_{1}^{\alpha_{1} } } \end{array}\right)\Bigg)-\Bigg(T_{1}+T_{1}^{\beta_{1}+1}\lambda_{k}\Gamma\left(\gamma_{1} \right)\times \\
& \times 
E_{2} \left(\left. \begin{array}{l} {\gamma_{1} ,\gamma_{1} ,1;1,0} \\ {\beta _{1} +2,\beta _{1} ,\alpha_{1} ;\gamma_{1} ,\gamma_{1} ;1,1} \end{array}\right|\begin{array}{c} {\lambda_k T_{1}^{\beta _{1} } } \\ {\delta T_{1}^{\alpha_{1} } } \end{array}\right)\Bigg)T_{1}^{\beta_{1}-1}\Gamma\left(\gamma_{1} \right)E_{2} \left(\left. \begin{array}{l} {\gamma_{1} ,\gamma_{1} ,1;1,0} \\ {\beta _{1},\beta _{1} ,\alpha_{1} ;\gamma_{1} ,\gamma_{1} ;1,1} \end{array}\right|\begin{array}{c} {\lambda_k T_{1}^{\beta _{1} } } \\ {\delta T_{1}^{\alpha_{1} } } \end{array}\right)-\\
&-\lambda_{k}\Gamma\left(\gamma_{2} \right)\left(\xi-T_{1} \right)^{\beta _{2}} E_{2} \left(\left. \begin{array}{l} {\gamma_{2} ,\gamma_{2} ,1;1,0} \\ {\beta _{2} +1,\beta _{2} ,\alpha_{2} ;\gamma_{2} ,\gamma_{2} ;1,1} \end{array}\right|\begin{array}{c} {\lambda_{k} \left(\xi-T_{1} \right)^{\beta _{2} } } \\ {\delta \left(\xi-T_{1} \right)^{\alpha_{2} } } \end{array}\right)\Bigg(T_{1}+T_{1}^{\beta_{1}+1} \lambda_{k}\Gamma(\gamma_{1})\times\\
&\times E_{2} \left(\left. \begin{array}{l} {\gamma_{1} ,\gamma_{1} ,1;1,0} \\ {\beta _{1} +2,\beta _{1} ,\alpha_{1} ;\gamma_{1} ,\gamma_{1} ;1,1} \end{array}\right|\begin{array}{c} {\lambda_k T_{1}^{\beta _{1} } } \\ {\delta T_{1}^{\alpha_{1} } } \end{array}\right)\Bigg)T_{1}^{\beta_{1}-1}  \Gamma(\gamma_{1})E_{2} \left(\left. \begin{array}{l} {\gamma_{1} ,\gamma_{1} ,1;1,0} \\ {\beta _{1} ,\beta _{1} ,\alpha_{1} ;\gamma_{1} ,\gamma_{1} ;1,1} \end{array}\right|\begin{array}{c} {\lambda_k T_{1}^{\beta _{1} } } \\ {\delta T_{1}^{\alpha_{1} } } \end{array}\right).
\end{align*}
As a result of simplifying similar terms, the value of $\widetilde{\Delta}_{k}$ can be written as \eqref{39}.

\subsection{A2: Simplification of ${}_4A_k$.}
We substitute \eqref{41} and \eqref{42} into equation \eqref{23}:
\begin{align*}
 & {}_{4}A_{k}=\frac{1}{\widetilde{\Delta}_{k}}\Bigg[\psi_{k}T_{1}^{\beta_{1}} \Gamma(\gamma_{1}) E_{2} \left(\left. \begin{array}{l} {\gamma_{1} ,\gamma_{1} ,1;1,0} \\ {\beta _{1} +1,\beta _{1} ,\alpha_{1} ;\gamma_{1} ,\gamma_{1} ;1,1} \end{array}\right|\begin{array}{c} {\lambda_k T_{1}^{\beta _{1} } } \\ {\delta T_{1}^{\alpha_{1} } } \end{array}\right)+ \lambda_{k} \psi_{k}T_{1}^{2\beta_{1}} \Gamma^{2}(\gamma_{1}) \times\\
&\times\left( E_{2} \left(\left. \begin{array}{l} {\gamma_{1} ,\gamma_{1} ,1;1,0} \\ {\beta _{1} +1,\beta _{1} ,\alpha_{1} ;\gamma_{1} ,\gamma_{1} ;1,1} \end{array}\right|\begin{array}{c} {\lambda_k T_{1}^{\beta _{1} } } \\ {\delta T_{1}^{\alpha_{1} } } \end{array}\right) \right)^{2}-\lambda_{k}\varphi_{k}\Gamma\left(\gamma_{1} \right)\Gamma\left(\gamma_{2} \right)T_{1}^{\beta_{1}}\left(\xi-T_{1} \right)^{\beta _{2}} \times \\
&\times E_{2} \left(\left. \begin{array}{l} {\gamma_{1} ,\gamma_{1} ,1;1,0} \\ {\beta _{1},\beta _{1} ,\alpha_{1} ;\gamma_{1} ,\gamma_{1} ;1,1} \end{array}\right|\begin{array}{c} {\lambda_k T_{1}^{\beta _{1} } } \\ {\delta T_{1}^{\alpha_{1}} } \end{array}\right)\times\\
&\times E_{2} \left(\left. \begin{array}{l} {\gamma_{2} ,\gamma_{2} ,1;1,0} \\ {\beta _{2} +1,\beta _{2} ,\alpha_{2} ;\gamma_{2} ,\gamma_{2} ;1,1} \end{array}\right|\begin{array}{c} {\lambda_{k} \left(\xi-T_{1} \right)^{\beta _{2} } } \\ {\delta \left(\xi-T_{1} \right)^{\alpha_{2} } } \end{array}\right)-\lambda_{k}^{2}T_{1}^{2\beta_{1}}\Gamma\left(\gamma_{2} \right)\varphi_{k} \Gamma^{2}\left(\gamma_{1} \right)\left(\xi-T_{1} \right)^{\beta _{2}}\times\\
& \times E_{2} \left(\left. \begin{array}{l} {\gamma_{1} ,\gamma_{1} ,1;1,0} \\ {\beta _{1},\beta _{1} ,\alpha_{1} ;\gamma_{1} ,\gamma_{1} ;1,1} \end{array}\right|\begin{array}{c} {\lambda_k T_{1}^{\beta _{1} } } \\ {\delta T_{1}^{\alpha_{1} } } \end{array}\right)E_{2} \left(\left. \begin{array}{l} {\gamma_{1} ,\gamma_{1} ,1;1,0} \\ {\beta _{1}+2,\beta _{1} ,\alpha_{1} ;\gamma_{1} ,\gamma_{1} ;1,1} \end{array}\right|\begin{array}{c} {\lambda_k T_{1}^{\beta _{1} } } \\ {\delta T_{1}^{\alpha_{1} } } \end{array}\right)\times\\
& \times E_{2} \left(\left. \begin{array}{l} {\gamma_{2} ,\gamma_{2} ,1;1,0} \\ {\beta _{2} +1,\beta _{2} ,\alpha_{2} ;\gamma_{2} ,\gamma_{2} ;1,1} \end{array}\right|\begin{array}{c} {\lambda_{k} \left(\xi-T_{1} \right)^{\beta _{2} } } \\ {\delta \left(\xi-T_{1} \right)^{\alpha_{2} } } \end{array}\right)+\varphi_{k}\Gamma\left(\gamma_{2} \right)\left(\xi-T_{1} \right)^{\beta _{2}}\times \\
&\times E_{2} \left(\left. \begin{array}{l} {\gamma_{2} ,\gamma_{2} ,1;1,0} \\ {\beta _{2} +1,\beta _{2} ,\alpha_{2} ;\gamma_{2} ,\gamma_{2} ;1,1} \end{array}\right|\begin{array}{c} {\lambda_{k} \left(\xi-T_{1} \right)^{\beta _{2} } } \\ {\delta \left(\xi-T_{1} \right)^{\alpha_{2} } } \end{array}\right)+2\varphi_{k}\lambda_{k}\Gamma\left(\gamma_{1} \right)\Gamma\left(\gamma_{2} \right)T_{1}^{\beta_{1}}\left(\xi-T_{1} \right)^{\beta _{2}}\times \\
& \times E_{2} \left(\left. \begin{array}{l} {\gamma_{2} ,\gamma_{2} ,1;1,0} \\ {\beta _{2} +1,\beta _{2} ,\alpha_{2} ;\gamma_{2} ,\gamma_{2} ;1,1} \end{array}\right|\begin{array}{c} {\lambda_{k} \left(\xi-T_{1} \right)^{\beta _{2} } } \\ {\delta \left(\xi-T_{1} \right)^{\alpha_{2} } } \end{array}\right)E_{2} \left(\left. \begin{array}{l} {\gamma_{1} ,\gamma_{1} ,1;1,0} \\ {\beta _{1} +1,\beta _{1} ,\alpha_{1} ;\gamma_{1} ,\gamma_{1} ;1,1} \end{array}\right|\begin{array}{c} {\lambda_k T_{1}^{\beta _{1} } } \\ {\delta T_{1}^{\alpha_{1} } } \end{array}\right)+\\
&+\varphi_{k}\lambda_{k}^{2}\Gamma^{2}\left(\gamma_{1} \right)\Gamma\left(\gamma_{2} \right)T_{1}^{2\beta_{1}}\left(\xi-T_{1} \right)^{\beta _{2}}E_{2} \left(\left. \begin{array}{l} {\gamma_{2} ,\gamma_{2} ,1;1,0} \\ {\beta _{2} +1,\beta _{2} ,\alpha_{2} ;\gamma_{2} ,\gamma_{2} ;1,1} \end{array}\right|\begin{array}{c} {\lambda_{k} \left(\xi-T_{1} \right)^{\beta _{2} } } \\ {\delta \left(\xi-T_{1} \right)^{\alpha_{2} } } \end{array}\right) \times\\
& \times\left( E_{2} \left(\left. \begin{array}{l} {\gamma_{1} ,\gamma_{1} ,1;1,0} \\ {\beta _{1} +1,\beta _{1} ,\alpha_{1} ;\gamma_{1} ,\gamma_{1} ;1,1} \end{array}\right|\begin{array}{c} {\lambda_k T_{1}^{\beta _{1} } } \\ {\delta T_{1}^{\alpha_{1} } } \end{array}\right) \right)^{2}
- \psi_{k}T_{1}^{\beta_{1}} \Gamma(\gamma_{1}) E_{2} \left(\left. \begin{array}{l} {\gamma_{1} ,\gamma_{1} ,1;1,0} \\ {\beta _{1} ,\beta _{1} ,\alpha_{1} ;\gamma_{1} ,\gamma_{1} ;1,1} \end{array}\right|\begin{array}{c} {\lambda_k T_{1}^{\beta _{1} } } \\ {\delta T_{1}^{\alpha_{1} } } \end{array}\right)-\\
& -\psi_{k} T_{1}^{2\beta_{1}}\Gamma^{2}(\gamma_{1})\lambda_{k}E_{2} \left(\left. \begin{array}{l} {\gamma_{1} ,\gamma_{1} ,1;1,0} \\ {\beta _{1},\beta _{1} ,\alpha_{1} ;\gamma_{1} ,\gamma_{1} ;1,1} \end{array}\right|\begin{array}{c} {\lambda_k T_{1}^{\beta _{1} } } \\ {\delta T_{1}^{\alpha_{1} } } \end{array}\right) E_{2} \left(\left. \begin{array}{l} {\gamma_{1} ,\gamma_{1} ,1;1,0} \\ {\beta _{1}+2,\beta _{1} ,\alpha_{1} ;\gamma_{1} ,\gamma_{1} ;1,1} \end{array}\right|\begin{array}{c} {\lambda_k T_{1}^{\beta _{1} } } \\ {\delta T_{1}^{\alpha_{1} } } \end{array}\right)+ \\
&+\psi_{k}T_{1}+\psi_{k} \lambda_{k}T_{1}^{\beta_{1}+1}\Gamma(\gamma_{1})E_{2} \left(\left. \begin{array}{l} {\gamma_{1} ,\gamma_{1} ,1;1,0} \\ {\beta _{1}+2,\beta _{1} ,\alpha_{1} ;\gamma_{1} ,\gamma_{1} ;1,1} \end{array}\right|\begin{array}{c} {\lambda_k T_{1}^{\beta _{1} } } \\ {\delta T_{1}^{\alpha_{1} } } \end{array}\right)\Bigg]+\frac{1}{\widetilde{\Delta}_{k}}\Bigg[\psi_{k}-\varphi_{k}-\\
&-2\lambda_{k} \varphi_{k}T_{1}^{\beta_{1}} \Gamma(\gamma_{1})E_{2} \left(\left. \begin{array}{l} {\gamma_{1} ,\gamma_{1} ,1;1,0} \\ {\beta _{1} +1,\beta _{1} ,\alpha_{1} ;\gamma_{1} ,\gamma_{1} ;1,1} \end{array}\right|\begin{array}{c} {\lambda_k T_{1}^{\beta _{1} } } \\ {\delta T_{1}^{\alpha_{1} } } \end{array}\right)-\lambda_{k}^{2} \varphi_{k} T_{1}^{2\beta_{1}}\Gamma^{2}(\gamma_{1}) \times \nonumber\\
&\times \left( E_{2} \left(\left. \begin{array}{l} {\gamma_{1} ,\gamma_{1} ,1;1,0} \\ {\beta _{1} +1,\beta _{1} ,\alpha_{1} ;\gamma_{1} ,\gamma_{1} ;1,1} \end{array}\right|\begin{array}{c} {\lambda_k T_{1}^{\beta _{1} } } \\ {\delta T_{1}^{\alpha_{1} } } \end{array}\right) \right)^{2}-\lambda_{k}\varphi_{k}\Gamma\left(\gamma_{2} \right)\left(\xi-T_{1} \right)^{\beta _{2}} \times \nonumber\\
&\times E_{2} \left(\left. \begin{array}{l} {\gamma_{2} ,\gamma_{2} ,1;1,0} \\ {\beta _{2} +1,\beta _{2} ,\alpha_{2} ;\gamma_{2} ,\gamma_{2} ;1,1} \end{array}\right|\begin{array}{c} {\lambda_{k} \left(\xi-T_{1} \right)^{\beta _{2} } } \\ {\delta \left(\xi-T_{1} \right)^{\alpha_{2} } }\end{array}\right)-2\lambda_{k}^{2}\varphi_{k}T_{1}^{\beta_{1}}\Gamma\left(\gamma_{1} \right)\Gamma\left(\gamma_{2} \right)\left(\xi-T_{1} \right)^{\beta _{2}}\times\nonumber\\
&\times E_{2} \left(\left. \begin{array}{l} {\gamma_{2} ,\gamma_{2} ,1;1,0} \\ {\beta _{2} +1,\beta _{2} ,\alpha_{2} ;\gamma_{2} ,\gamma_{2} ;1,1} \end{array}\right|\begin{array}{c} {\lambda_{k} \left(\xi-T_{1} \right)^{\beta _{2} } } \\ {\delta \left(\xi-T_{1} \right)^{\alpha_{2}} } \end{array}\right)E_{2} \left(\left. \begin{array}{l} {\gamma_{1} ,\gamma_{1} ,1;1,0} \\ {\beta _{1} +1,\beta _{1} ,\alpha_{1} ;\gamma_{1} ,\gamma_{1} ;1,1} \end{array}\right|\begin{array}{c} {\lambda_k T_{1}^{\beta _{1} } } \\ {\delta T_{1}^{\alpha_{1} } } \end{array}\right)-\\
& -\lambda_{k}^{3}T_{1}^{2\beta_{1}}\varphi_{k}\Gamma^{2}\left(\gamma_{1} \right)\Gamma\left(\gamma_{2} \right)\left(\xi-T_{1} \right)^{\beta _{2}}E_{2} \left(\left. \begin{array}{l} {\gamma_{2} ,\gamma_{2} ,1;1,0} \\ {\beta _{2} +1,\beta _{2} ,\alpha_{2} ;\gamma_{2} ,\gamma_{2} ;1,1} \end{array}\right|\begin{array}{c} {\lambda_{k} \left(\xi-T_{1} \right)^{\beta _{2} } } \\ {\delta \left(\xi-T_{1} \right)^{\alpha_{2} } } \end{array}\right)\times\nonumber\\
& \times \left( E_{2} \left(\left. \begin{array}{l} {\gamma_{1} ,\gamma_{1} ,1;1,0} \\ {\beta _{1} +1,\beta _{1} ,\alpha_{1} ;\gamma_{1} ,\gamma_{1} ;1,1} \end{array}\right|\begin{array}{c} {\lambda_k T_{1}^{\beta _{1} } } \\ {\delta T_{1}^{\alpha_{1} } } \end{array}\right) \right)^{2}+\psi_{k}\lambda_k T_{1}^{\beta _{1} }+\psi_{k}\lambda_k^{2}T_{1}^{\beta _{1}+1 }\Gamma\left(\gamma_{1} \right) \times\nonumber\\
& \times E_{2} \left(\left. \begin{array}{l} {\gamma_{1} ,\gamma_{1} ,1;1,0} \\ {\beta _{1} +2,\beta _{1} ,\alpha_{1} ;\gamma_{1} ,\gamma_{1} ;1,1} \end{array}\right|\begin{array}{c} {\lambda_k T_{1}^{\beta _{1} } } \\ {\delta T_{1}^{\alpha_{1} } } \end{array}\right)+\psi_{k}\lambda_kT_{1}^{\beta _{1} }\Gamma\left(\gamma_{1} \right) \times \nonumber\\
&\times E_{2} \left(\left. \begin{array}{l} {\gamma_{1} ,\gamma_{1} ,1;1,0} \\ {\beta _{1} +1,\beta _{1} ,\alpha_{1} ;\gamma_{1} ,\gamma_{1} ;1,1} \end{array}\right|\begin{array}{c} {\lambda_k T_{1}^{\beta _{1} } } \\ {\delta T_{1}^{\alpha_{1} } } \end{array}\right)+\lambda_{k}\varphi_{k}\Gamma(\gamma_{1})T_{1}^{\beta_{1}}E_{2} \left(\left. \begin{array}{l} {\gamma_{1} ,\gamma_{1} ,1;1,0} \\ {\beta _{1},\beta _{1} ,\alpha_{1} ;\gamma_{1} ,\gamma_{1} ;1,1} \end{array}\right|\begin{array}{c} {\lambda_k T_{1}^{\beta _{1} } } \\ {\delta T_{1}^{\alpha_{1} } } \end{array}\right)+\nonumber\\
&+\lambda_{k}^{2} \varphi_{k}\Gamma^{2}(\gamma_{1})T_{1}^{2\beta_{1}} E_{2} \left(\left. \begin{array}{l} {\gamma_{1} ,\gamma_{1} ,1;1,0} \\ {\beta _{1},\beta _{1} ,\alpha_{1} ;\gamma_{1} ,\gamma_{1} ;1,1} \end{array}\right|\begin{array}{c} {\lambda_k T_{1}^{\beta _{1} } } \\ {\delta T_{1}^{\alpha_{1} } } \end{array}\right)E_{2} \left(\left. \begin{array}{l} {\gamma_{1} ,\gamma_{1} ,1;1,0} \\ {\beta _{1}+2,\beta _{1} ,\alpha_{1} ;\gamma_{1} ,\gamma_{1} ;1,1} \end{array}\right|\begin{array}{c} {\lambda_k T_{1}^{\beta _{1} } } \\ {\delta T_{1}^{\alpha_{1} } } \end{array}\right)+\nonumber\\
&+\lambda_{k}^{2}\varphi_{k}\Gamma(\gamma_{1})\Gamma\left(\gamma_{2} \right)\left(\xi-T_{1} \right)^{\beta _{2}}T_{1}^{\beta_{1}}E_{2} \left(\left. \begin{array}{l} {\gamma_{2} ,\gamma_{2} ,1;1,0} \\ {\beta _{2} +1,\beta _{2} ,\alpha_{2} ;\gamma_{2} ,\gamma_{2} ;1,1} \end{array}\right|\begin{array}{c} {\lambda_{k} \left(\xi-T_{1} \right)^{\beta _{2} } } \\ {\delta \left(\xi-T_{1} \right)^{\alpha_{2} } } \end{array}\right)\times\nonumber\\
& \times E_{2} \left(\left. \begin{array}{l} {\gamma_{1} ,\gamma_{1} ,1;1,0} \\ {\beta _{1},\beta _{1} ,\alpha_{1} ;\gamma_{1} ,\gamma_{1} ;1,1} \end{array}\right|\begin{array}{c} {\lambda_k T_{1}^{\beta _{1} } } \\ {\delta T_{1}^{\alpha_{1} } } \end{array}\right)+T_{1}^{2\beta_{1}}E_{2} \left(\left. \begin{array}{l} {\gamma_{2} ,\gamma_{2} ,1;1,0} \\ {\beta _{2} +1,\beta _{2} ,\alpha_{2} ;\gamma_{2} ,\gamma_{2} ;1,1} \end{array}\right|\begin{array}{c} {\lambda_{k} \left(\xi-T_{1} \right)^{\beta _{2} } } \\ {\delta \left(\xi-T_{1} \right)^{\alpha_{2} } } \end{array}\right)\times\nonumber\\
&\times\lambda_{k}^{3}\varphi_{k}\Gamma^{2}(\gamma_{1})\Gamma(\gamma_{2})E_{2} \left(\left. \begin{array}{l} {\gamma_{1} ,\gamma_{1} ,1;1,0} \\ {\beta _{1},\beta _{1} ,\alpha_{1} ;\gamma_{1} ,\gamma_{1} ;1,1} \end{array}\right|\begin{array}{c} {\lambda_k T_{1}^{\beta _{1} } } \\ {\delta T_{1}^{\alpha_{1} } } \end{array}\right)E_{2} \left(\left. \begin{array}{l} {\gamma_{1} ,\gamma_{1} ,1;1,0} \\ {\beta _{1}+2,\beta _{1} ,\alpha_{1} ;\gamma_{1} ,\gamma_{1} ;1,1} \end{array}\right|\begin{array}{c} {\lambda_k T_{1}^{\beta _{1} } } \\ {\delta T_{1}^{\alpha_{1} } } \end{array}\right)\Bigg]. \nonumber
\end{align*}
and after simplifying similar terms, the coefficient ${}_{4}A_{k}$ can be written as \eqref{43}.
\bibliographystyle{amsplain}

\end{document}